\documentclass[11pt]{amsart}

\usepackage[english]{babel}

\usepackage[letterpaper,top=2cm,bottom=2cm,left=3cm,right=3cm,marginparwidth=1.75cm]{geometry}

\usepackage{amsmath,amsthm,amsfonts}
\usepackage{graphicx}
\usepackage{color} 
\usepackage[colorlinks=true, allcolors=blue]{hyperref}
\usepackage{euscript,mathrsfs}
\usepackage[all]{xy}
\usepackage[dvipsnames]{xcolor}
\usepackage{bbm}

\makeatletter
\def\@adminfootnotes{%
 \let\@makefnmark\relax  \let\@thefnmark\relax
 \ifx\@empty\@date\else \@footnotetext{\@setdate}\fi
  \ifx\@empty\@subjclass\else \@footnotetext{\@setsubjclass}\fi
  \ifx\@empty\@keywords\else \@footnotetext{\@setkeywords}\fi
  \ifx\@empty\thankses\else \@footnotetext{%
    \def\par{\let\par\@par}\@setthanks}%
 \fi
}
\makeatother

\usepackage{tikz}
\usetikzlibrary{cd}
\usetikzlibrary{decorations}
\usetikzlibrary{decorations.markings}
\usetikzlibrary{decorations.pathreplacing}
\usetikzlibrary{decorations.pathmorphing}
\usetikzlibrary{arrows}
\usetikzlibrary{arrows.meta,shapes,positioning,matrix,calc}
\usetikzlibrary{shapes.callouts}
\tikzstyle{densely dotted}=[dash pattern=on \pgflinewidth off 0.5pt]
\tikzset{anchorbase/.style={baseline={([yshift=-0.5ex]current bounding box.center)}},
tinynodes/.style={font=\tiny,text height=0.25ex,text depth=0.05ex},
smallnodes/.style={font=\scriptsize,text height=0.75ex,text depth=0.15ex},
usual/.style={line width=0.9,color=black},
dusual/.style={line width=0.9,color=spinach,densely dashed},
pole/.style={line width=3.0,color=specialgray},
crossline/.style={preaction={draw=white,line width=5.0pt,-},preaction={draw=black,line width=0.9pt,-}},
crosspole/.style={preaction={draw=white,line width=6.0pt,-},preaction={draw=specialgray,line width=3.0pt,-}},
mor/.style={line width=0.75,color=black,fill=cream},
blob/.style={circle,fill,minimum size=5.0pt,inner sep=0pt,outer sep=0pt},
blobed/.style n args={3}{decoration={markings,post length=0.5mm,pre length=0.5mm,
mark=at position #1 with {\node[blob,#3,label=left:$#2\!$]at (0,0){};}
},postaction={decorate}},
rblobed/.style n args={3}{decoration={markings,post length=0.5mm,pre length=0.5mm,
mark=at position #1 with {\node[blob,#3,label=right:$\!#2$]at (0,0){};}
},postaction={decorate}},
}
\tikzstyle{tikzdot}=[fill, circle, inner sep=1.6pt]
\newcommand{\tikzdiagh}[2][]{\tikz[#1,thick,baseline={([yshift=1ex+#2]current bounding box.center)}]}
\newcommand{\tikzdiagc}[1][]{\tikzdiagh[#1]{-1ex}}

\newtheorem{thm}{Theorem}[section]
\newtheorem{lem}[thm]{Lemma}
\newtheorem{cor}[thm]{Corollary}
\newtheorem{prop}[thm]{Proposition}

\theoremstyle{definition}

\theoremstyle{remark}
\newtheorem{rem}[thm]{Remark}

\numberwithin{equation}{section}

\def\Z{\mathbb Z}
\def\eR{\EuScript{R}}
\def\lra{\longrightarrow}
\def\mcC{\mathcal{C}}
\def\tr{\mathsf{tr}}

\newcommand{\dmod}{\mathsf{-gmod}}
\newcommand{\kk}{\mathbf{k}}
\newcommand{\undR}{\underline{R}}
\newcommand{\undB}{\underline{B}}
\newcommand{\minuss}{\underline{s}}
\newcommand{\undone}{\underline{1}}
\newcommand{\undotimes}{\underline{\otimes}}
\newcommand{\oneb}{\mathbbm{1}}
\newcommand{\xra}[1]{\xrightarrow{#1}}
\newcommand{\sbim}{\mathrm{SBim}} 
\newcommand{\HOM}{\mathrm{HOM}}
\newcommand{\Hom}{\mathrm{Hom}}

\DeclareMathOperator{\id}{id}

\begin{document}

\title{Odd two-variable Soergel bimodules and Rouquier complexes}


\author{Mikhail Khovanov} 
\address{Department of Mathematics, Columbia University, New York, NY 10027, USA}
\email{khovanov@math.columbia.edu}

\author{Krzysztof Putyra} 
\address{Institut f\"ur Mathematik, Universit\"at Z\"urich, Winterthurerstrasse 190 CH-8057 Z\"urich, Switzerland }
\email{krzysztof.putyra@math.uzh.ch}

\author{Pedro Vaz} 
\address{Institut de Recherche en Math\'ematique et Physique, Universit\'e catholique de Louvain, Chemin du Cyclotron 2, bte L7.01.02
1348 Louvain-la-Neuve, Belgium}
\email{pedro.vaz@uclouvain.be}

\thanks{The authors are grateful to the anonymous referee for a thorough and valuable work on the earlier version of the paper and to Cailan Li for additional corrections.
  M.K. was partially supported by NSF grant DMS-1807425 and Simons Foundation sabbatical award \#817792 (Simons Fellows program). 
  P.V. was supported by the Fonds de la Recherche Scientifique - FNRS under Grant no. MIS-F.4536.19.}

\date{February 7, 2023}

\begin{abstract}
We consider the odd analogue of the category of Soergel bimodules. In the odd case and already for two variables, the transposition bimodule cannot be merged into the generating Soergel bimodule, forcing one into a monoidal category with a larger Grothendieck ring compared to the even case. We establish biadjointness of suitable functors and develop graphical calculi in the 2-variable case for the odd Soergel category and the related singular Soergel 2-category.  We describe the odd analogue of the Rouquier complexes and establish their invertibility in the homotopy category. For three variables, the absence of a direct sum decomposition of the tensor product of generating Soergel bimodules presents an obstacle for the Reidemeister III relation to hold in the homotopy category.
\end{abstract}

\maketitle

\tableofcontents

%
%

\section{Introduction}
\label{sec_intro}

In this note we propose an odd analogue of Soergel bimodules for Coxeter type $A_1$. Soergel bimodules for $A_1$ are certain bimodules for the algebra of polynomials in two variables. 
In the odd case it's role is played by the algebra of skew-symmetric polynomials $R=\kk\langle x_1,x_2\rangle /(x_1x_2+x_2x_1)$. The substitute for the generating Soergel bimodule $B$ over the polynomial algebra consists of two $R$-bimodules $B$ and $\undB$ that constitute a biadjoint pair (that is, the functor of tensoring with $\undB$ is both left and right adjoint to tensoring with $B$). 

Starting in Section~\ref{subsec-biadjointness} we develop a graphical calculus for the category of odd Soergel bimodules in two variables and define a pair of mutually-inverse functors $\eR,\eR'$ on the homotopy category of graded $R$-modules given by the tensor product with complexes of bimodules formed from suitable bimodule maps $B\stackrel{m}{\lra} R\{-1\}$ and $R\{1\}\lra \undB$, where $\{ \pm 1\}$ is a grading shift. 

These complexes of bimodules and corresponding functors $\eR,\eR'$ are odd analogues of the Rouquier complexes that in the even case give rise to a braid group action on the homotopy category of modules over the $n$-variable polynomial algebra. In Section~\ref{sec-obstacle} we explain an obstacle that exists in the odd case to having the braid relation $\eR_i\eR_{i+1} \eR_i\cong \eR_{i+1}\eR_i\eR_{i+1}$. The lack of this braid relation blocks an attempt, from which this note originated, to define odd HOMFLYPT link homology via braid closures and odd Soergel bimodules, analogous to the original construction of HOMFLYPT link homology via Hochschild homology of Soergel bimodules~\cite{Kh-Soergel}. It's not known  either whether the odd counterpart of bigraded $SL(N)$ link homology~\cite{KhRI} exists for $N>2$. For the definition and structure of odd $SL(2)$ link homology see~\cite{ORS,P-odd,P-phd,NV,NP}. 

In Section~\ref{sec_groth} we identify the Grothendieck ring of the category of odd Soergel bimodules for $A_1$ and compute a natural semilinear form and trace on that ring. 

Rouquier functors on  the even Soergel category are closely related to the invertible functors of twisting by a relative spherical object in the  Fukaya--Floer categories and in the derived categories of coherent sheaves~\cite{KhT}. It should be interesting to explore odd counterparts of such functors; one can, for instance, ask whether there exists an odd counterpart of quiver varieties and associated derived categories of coherent sheaves on them and, more generally, an odd counterpart of algebraic geometry. A simpler problem is to understand the relation between functors $\eR,\eR'$ and  recently constructed odd counterpart of Chuang--Rouquier symmetries and Rickard complexes~\cite{BK,ELV}.

%
%

\section{Bimodules for two strands}
\label{sec_two_strands}


\subsection{Anticommuting polynomials and odd Demazure operators}
\label{odd_demazure} 

Let $\kk$ be a commutative ring and denote by $R=\kk\langle x_1,x_2\rangle /(x_1x_2+x_2x_1)$ the algebra of anticommuting polynomials in two variables. Let $S_2$ be the symmetric group on two letters with generator $s$, acting on $R$ by $s(x_i)=-x_{s(i)}$ and $s(fg)=s(f)s(g)$ for $f,g\in R$.

Define the \emph{odd Demazure operator} $\partial:R\lra R$  (see~\cite{EKL,KKO,EL}) as follows: 
\begin{itemize}
\item $\partial(1)=0,$  $\partial(x_1)=\partial(x_2)=1$.
\item The twisted Leibniz rule holds 
\[\partial (fg)=(\partial f)g+s(f)\partial g .
\]
\end{itemize}

Note that ring $R$ does not have unique factorizations, for instance $(x_1+x_2)^2=(x_1-x_2)^2$ in $R$. 

The equation 
\begin{equation}\label{eq_skew_comm} 
\partial\circ s = -   s \circ \partial 
\end{equation}
follows via the Leibniz rule above and checking it on generators of $R$.

Let $R^s=\ker (\partial)={\rm im}(\partial)\subset R$. Equality $\ker(\partial)={\rm im}(\partial)$ is straighforward to check. The twisted Leibniz rule then implies that $R^s$ is a subring of $R$. 
The ring $R^s$ has generators  $E_1=x_1-x_2,E_2=x_1x_2$ and the  defining relation is that these generators anticommute, 
\[ R^s \ = \ \mathbf{k}\langle E_1,E_2\rangle /(E_1E_2+E_2E_1), \ \ E_1=x_1-x_2, \ E_2=x_1x_2. 
\] 
The action of $S_2$ on $R$ restricts to an action on $R^s$, with 
\[ s(E_1) = E_1, \ \ s(E_2) = - E_2. 
\]

Define the \emph{transposition bimodule} $\undR$ to be free rank one as a left and as a right $R$-module, with the generator $\undone$ and relations $f\undone = \undone s(f)$. It has a subbimodule $\undR^s:= R^s \undone = \undone R^s\subset \undR$. Denote $\undone \in \undR$ by $\undone^s$ when viewed as an element of $\undR^s\subset \undR$.  We fix bimodule isomorphisms 
\begin{eqnarray}
\label{eq_iso_two_undR_1}
 \undR \otimes_R \undR & \cong & R, \ \ \undone \otimes \undone \longmapsto 1, \\
  \label{eq_iso_two_undR_2}
 \undR^s \otimes_{R^s} \undR^s & \cong & R^s, \ \  \undone^s\otimes \undone^s \longmapsto 1^s, \\
  \label{eq_iso_two_undR_3}
 \undR^s \otimes_{R^s} \undR & \cong & {}_{R^s} R_R, \ \  \undone^s\otimes \undone \longmapsto 1,  \\
 \label{eq_iso_two_undR_4}
 \undR \otimes_{R^s} \undR^s & \cong & {}_R R_{R^s}, \ \  \undone \otimes \undone^s \longmapsto 1. 
\end{eqnarray}
The first map is an isomorphism of $R$-bimodules, the second -- that of $R^s$-bimodules.

The odd Demazure operator can be written as a map
\begin{equation}\label{eq_d_prime} 
\partial': R\lra \undR^s, \ \  \partial'(f)=\undone \partial(f).
\end{equation} 
It is then naturally a map of $R^s$-bimodules. We can also write it as a bimodule map 
\begin{equation}\label{eq_part_two}
\partial'': \undR \lra R^s, \ \  \partial''(f \undone g) = \partial(s(f)g) .
\end{equation}

The ring $R$ is a free rank two left and right module over $R^s$, 
\[ {}_{R^s} R \cong R^s \cdot 1 \oplus R^s\cdot x_i, \ \  R_{R^s} \cong 1 \cdot R^s\oplus x_i\cdot R^s, \ \  i\in \{1,2\}.  
\]

Note that  in the bimodule $R\otimes_{R^s} R$ we have
\begin{equation}\label{eq_x_onetwo}
     x_1\otimes 1 - 1 \otimes x_1 \ = \ x_2\otimes 1 - 1 \otimes x_2,
\end{equation}
since $x_1-x_2\in R^s$.

We make $R$ into a graded ring, with $\deg(x_1)=\deg(x_2)=2$. Then $R^s$ has an induced grading, and $\partial$ is a degree $-2$ map. Bimodules $\undR$ and $\undR^s$ are naturally graded, with $\deg(\undone)=\deg(\undone^s)=0$.


\subsection{Biadjointness}
\label{subsec-biadjointness}

We consider the graded bimodules ${}_R R_{R^s}$, ${}_{R^s} R_R$ and ${}_R\underline{R}_R$ and introduce the following four functors, where $\mathsf{gmod}$ stands for the category of graded modules and degree zero maps:  
\begin{itemize}
\item  $F_{\uparrow}:R^s\dmod \lra R\dmod$ is  the induction functor of tensoring with the  graded bimodule ${}_R R_{R^s}$.
\item $F_{\downarrow}: R\dmod \lra R^s\dmod$ is the restriction functor; it is isomorphic to tensoring with the bimodule ${}_{R^s} R_R$. 
\item  $F_- : R\dmod \lra R\dmod$ is the functor of tensoring with $\underline{R}$. 
\item  $F_{\minuss}: R^s\dmod\lra R^s\dmod$ is the functor of tensoring with $\undR^s$. 
\end{itemize}
The endofunctors $F_-$ and $F_{\minuss}$ are involutive. 
Fix a functor isomorphism 
\begin{equation} \label{eq_iso_down} F_{\minuss}\circ F_{\downarrow} \cong F_{\downarrow}\circ F_-
\end{equation} 
given by the bimodule isomorphism 
\[   \undR^s \otimes_{R^s} R \cong {}_{R_s} R\otimes_R \undR \cong \ {}_{R^s}\undR 
\] 
which takes $\undone^s \otimes f$ to $s(f)\otimes \undone$ to $\undone f$. The last term is $\undR$ viewed as $(R^s,R)$-bimodule with the standard left action of $R_s$ and right action of $R$. Likewise, there's an isomorphism 
\begin{equation}\label{eq_iso_up}  F_{\uparrow}\circ  F_{\minuss}\cong F_- \circ F_{\uparrow}
\end{equation}  
via the corresponding bimodule isomorphisms 
\[   R\otimes_{R^s} \undR^s  \cong \undR \otimes_R R_{R_s} \cong \undR_{R^s}, \ \  
f\otimes \undone^s \longmapsto 
\undone \otimes s(f) \longmapsto \undone s(f).
\]
These isomorphisms can be thought of as ``sliding'' involutive functors $F_-$ and $F_{\minuss}$ through the induction and restriction functors $F_{\uparrow}$ and $F_{\downarrow}$. 

We depict natural transformations between compositions of these functors by drawing planar diagrams, with regions labelled by categories $R\dmod$ (white regions) and $R^s\dmod$ (shaded regions), following the usual string diagram notation. 
Identity natural transformation of $F_{\uparrow}$ (respectively $F_{\downarrow}$) is denoted by a vertical line, with the shaded region to the left (respectively, to the right), see equation (\ref{eq_ids}) below. 
We denote the identity functor on a category $\mcC$ by 
$\oneb_{\mcC}$ or just by $\oneb$. 

\begin{equation}\label{eq_ids}
\begin{split}
  \xy (0,1)*{
\tikzdiagc[scale=.7]{
\node[rectangle,fill=blue!10,minimum width=.75cm,minimum height=1.4cm] (r) at (.5,0) {};
  \draw[thick,-to] (0,-1) -- (0, .1);\draw[thick] (0,.1) -- (0, 1);  
\draw[thick, dotted] (-1,1) to (1,1);\draw[thick, dotted] (-1,-1) to (1,-1);
\node at (0,-1.3) {\small \text{$F_{\uparrow}$}}; \node at (0,1.3) {\small \text{$F_{\uparrow}$}}; 
\node at (-1.65,0) {\small \text{$R\dmod$}};\node at (1.65,0) {\small \text{$R^s\dmod$}};
}} \endxy
\mspace{45mu}
\xy (0,0)*{
\tikzdiagc[scale=.7]{
\node[rectangle,fill=blue!10,minimum width=.75cm,minimum height=1.4cm] (r) at (-.5,0) {};
  \draw[thick,-to] (0,1) -- (0,-.1);\draw[thick] (0,-.1) -- (0, -1);  
\draw[thick, dotted] (-1,1) to (1,1);\draw[thick, dotted] (-1,-1) to (1,-1);
\node at (1.6,-1) {\small \text{$F_{\downarrow}$}}; \node at (1.6,1) {\small \text{$F_{\downarrow}$}}; 
\draw[-implies,double equal sign distance]  (1.6,-.4) to (1.6,.4);
\node at (2.1,0) {\small \text{$\id$}};
}} \endxy
\\[1ex] 
\xy (0,1)*{
\tikzdiagc[scale=.7]{
  \draw[thick,Orange,densely dashed] (0,-1) -- (0,1);  
\draw[thick, dotted] (-1,1) to (1,1);\draw[thick, dotted] (-1,-1) to (1,-1);
\node at (1.6,-1) {\small \text{$F_{-}$}}; \node at (1.6,1) {\small \text{$F_{-}$}}; 
\draw[-implies,double equal sign distance]  (1.6,-.4) to (1.6,.4);
\node at (2.1,0) {\small \text{$\id$}};
}} \endxy
\mspace{45mu}
\xy (0,0)*{
\tikzdiagc[scale=.7]{
\node[rectangle,fill=blue!10,minimum width=1.5cm,minimum height=1.4cm] (r) at (0,0) {};
  \draw[thick,Orange,densely dashed] (0,-1) -- (0,1);  
\draw[thick, dotted] (-1,1) to (1,1);\draw[thick, dotted] (-1,-1) to (1,-1);
\node at (1.6,-1) {\small \text{$F_{\minuss}$}}; \node at (1.6,1) {\small \text{$F_{\minuss}$}}; 
\draw[-implies,double equal sign distance]  (1.6,-.4) to (1.6,.4);
\node at (2.1,0) {\small \text{$\id$}};
}} \endxy
\end{split}
\end{equation}

The identity natural transformation of $F_-$ and $F_{\minuss}$ is denoted by a vertical dashed orange line, in a white or shaded region, respectively, see (\ref{eq_ids}) above. 

Sliding isomorphisms above are shown as crossings of strands, which are mu\-tu\-ally-inverse isomorphisms, see equations (\ref{eq_fig_crossings}-\ref{eq_fig_crossings}) below. 

\begin{gather}\label{eq_fig_crossings}
  \xy (0,0)*{
\tikzdiagc[scale=.5]{
\fill[blue!10] (-1.25,-1) to (-1,-1) -- (1,-1) -- (-1,1) to (-1.25,1)--cycle;
\draw[thick,-to] (-1,1) -- (.5,-.5);\draw[thick] (.45,-.45) -- (1,-1); 
\draw[thick,Orange,densely dashed] (-1,-1)-- (1, 1);  
\draw[thick, dotted] (-1.25,1) to (1.25,1);\draw[thick, dotted] (-1.25,-1) to (1.25,-1);
\node at (-1,1.3) {\small \text{$F_{\downarrow}$}};\node at (1,-1.5) {\small \text{$F_{\downarrow}$}};
\node at (-1,-1.5) {\small \text{$F_{\minuss}$}};\node at (1,1.3) {\small \text{$F_{-}$}};
}} \endxy
\mspace{40mu}
\xy (0,0)*{
\tikzdiagc[scale=.5]{
\fill[blue!10] (-1.25,-1) to (-1,-1) -- (1,1) -- (-1,1) to (-1.25,1)--cycle;
\draw[thick] (-1,-1) -- (-.45,-.45);\draw[thick,to-] (-.5,-.5) -- (1,1);  
\draw[thick,Orange,densely dashed] (-1,1)-- (1,-1);  
\draw[thick, dotted] (-1.25,1) to (1.25,1);\draw[thick, dotted] (-1.25,-1) to (1.25,-1);
\node at (1,1.3) {\small \text{$F_{\downarrow}$}};\node at (-1,-1.5) {\small \text{$F_{\downarrow}$}};
\node at (1,-1.5) {\small \text{$F_{-}$}};\node at (-1,1.3) {\small \text{$F_{\minuss}$}};
}} \endxy
\mspace{40mu}
\xy (0,0)*{
\tikzdiagc[scale=.5]{
\fill[blue!10] (1.25,-1) to (1,-1) -- (-1,1) to (1.25,1)--cycle;
\draw[thick] (-1,1) -- (-.45,.45);\draw[thick,to-] (-.5,.5) -- (1,-1); 
\draw[thick,Orange,densely dashed] (-1,-1)-- (1, 1);  
\draw[thick, dotted] (-1.25,1) to (1.25,1);\draw[thick, dotted] (-1.25,-1) to (1.25,-1);
\node at (-1,1.3) {\small \text{$F_{\uparrow}$}};\node at (1,-1.5) {\small \text{$F_{\uparrow}$}};
\node at (-1,-1.5) {\small \text{$F_{-}$}};\node at (1,1.3) {\small \text{$F_{\minuss}$}};
}} \endxy
\mspace{40mu}
\xy (0,0)*{
\tikzdiagc[scale=.5]{
\fill[blue!10] (-1,-1) --(1.25,-1) -- (1.25,1) -- (1,1) to (-1,-1)--cycle;
\draw[thick,-to] (-1,-1) -- (.5,.5);\draw[thick] (.45,.45) -- (1,1);  
\draw[thick,Orange,densely dashed] (-1,1)-- (1,-1);  
\draw[thick, dotted] (-1.25,1) to (1.25,1);\draw[thick, dotted] (-1.25,-1) to (1.25,-1);
\node at (1,1.3) {\small \text{$F_{\uparrow}$}};\node at (-1,-1.5) {\small \text{$F_{\uparrow}$}};
\node at (-1,1.3) {\small \text{$F_{-}$}};\node at (1,-1.5) {\small \text{$F_{\minuss}$}};
}} \endxy
\\[1ex]  
\xy (0,1)*{
\tikzdiagc[scale=.5]{
\fill[blue!10] (-1.25,-1.5)to (.75,-1.5) to[out=90,in=-90] (-.6,0) to[out=90,in=-90] (.75, 1.5) to (-1.25,1.5)--cycle;
\draw[thick,Orange,densely dashed] (-.75,-1.5) to[out=90,in=-90] (.6,0) to[out=90,in=-90] (-.75,1.5); 
\draw[thick] (.75,-1.5) to[out=90,in=-90] (-.6,0); 
\draw[thick,to-]  (-.6,-.1) to[out=90,in=-90] (.75, 1.5); 
\draw[thick, dotted] (-1.25,1.5) to (1.25,1.5);\draw[thick, dotted] (-1.25,-1.5) to (1.25,-1.5);
}} \endxy
 = 
\xy (0,0)*{
\tikzdiagc[scale=.5]{
\fill[blue!10] (-1.25,-1.5) to (.75,-1.5) to (.75,1.5) to (-1.25,1.5)--cycle;
 \draw[thick] (.75,-1.5) to (.75,0); \draw[thick,to-] (.75,-.1) to (.75,1.5); 
\draw[thick,Orange,densely dashed] (-.75,-1.5) to (-.75, 1.5); 
\draw[thick, dotted] (-1.25,1.5) to (1.25,1.5);\draw[thick, dotted] (-1.25,-1.5) to (1.25,-1.5);
}} \endxy
\mspace{60mu}
\xy (0,0)*{
\tikzdiagc[scale=.5]{
\fill[blue!10] (-1.25,-1.5)to (-.75,-1.5) to[out=90,in=-90] (.6,-.1) to[out=90,in=-90] (-.75, 1.5) to (-1.25,1.5)--cycle;
\draw[thick,Orange,densely dashed] (.75,-1.5) to[out=90,in=-90] (-.6,0) to[out=90,in=-90] (.75,1.5); 
\draw[thick] (-.75,-1.5) to[out=90,in=-90] (.6,0); 
\draw[thick,to-]  (.6,-.1) to[out=90,in=-90] (-.75, 1.5); 
\draw[thick, dotted] (-1.25,1.5) to (1.25,1.5);\draw[thick, dotted] (-1.25,-1.5) to (1.25,-1.5);
}} \endxy
 = 
\xy (0,0)*{
\tikzdiagc[scale=.5]{
\fill[blue!10] (-1.25,-1.5) to (-.75,-1.5) to (-.75,1.5) to (-1.25,1.5)--cycle;
 \draw[thick] (-.75,-1.5) to (-.75,0); \draw[thick,to-] (-.75,-.1) to (-.75,1.5); 
\draw[thick,Orange,densely dashed] (.75,-1.5) to (.75, 1.5); 
\draw[thick, dotted] (-1.25,1.5) to (1.25,1.5);\draw[thick, dotted] (-1.25,-1.5) to (1.25,-1.5);
}} \endxy
\\[1.5ex] \label{eq_fig_crossings-last}
\xy (0,0)*{
\tikzdiagc[scale=.5]{
\fill[blue!10] (1.25,-1.5) to (.75,-1.5) to[out=90,in=-90] (-.6,0) to[out=90,in=-90] (.75, 1.5) to (1.25,1.5)--cycle;
\draw[thick,Orange,densely dashed] (-.75,-1.5) to[out=90,in=-90] (.6,0) to[out=90,in=-90] (-.75,1.5); 
\draw[thick,-to] (.75,-1.5) to[out=90,in=-90] (-.6,0); 
\draw[thick]  (-.6,-.1) to[out=90,in=-90] (.75, 1.5); 
\draw[thick, dotted] (-1.25,1.5) to (1.25,1.5);\draw[thick, dotted] (-1.25,-1.5) to (1.25,-1.5);
}} \endxy
 = 
\xy (0,0)*{
\tikzdiagc[scale=.5]{
\fill[blue!10] (1.25,-1.5) to (.75,-1.5) to (.75,1.5) to (1.25,1.5)--cycle;
 \draw[thick,-to] (.75,-1.5) to (.75,.1); \draw[thick] (.75,0) to (.75,1.5); 
\draw[thick,Orange,densely dashed] (-.75,-1.5) to (-.75, 1.5); 
\draw[thick, dotted] (-1.25,1.5) to (1.25,1.5);\draw[thick, dotted] (-1.25,-1.5) to (1.25,-1.5);
}} \endxy
\mspace{60mu}
\xy (0,0)*{
\tikzdiagc[scale=.5]{
\fill[blue!10] (1.25,-1.5)to (-.75,-1.5) to[out=90,in=-90] (.6,0) to[out=90,in=-90] (-.75, 1.5) to (1.25,1.5)--cycle;
\draw[thick,Orange,densely dashed] (.75,-1.5) to[out=90,in=-90] (-.6,0) to[out=90,in=-90] (.75,1.5); 
\draw[thick,-to] (-.75,-1.5) to[out=90,in=-90] (.6,.1); 
\draw[thick]  (.6,0) to[out=90,in=-90] (-.75, 1.5); 
\draw[thick, dotted] (-1.25,1.5) to (1.25,1.5);\draw[thick, dotted] (-1.25,-1.5) to (1.25,-1.5);
}} \endxy
 = 
\xy (0,0)*{
\tikzdiagc[scale=.5]{
\fill[blue!10] (1.25,-1.5) to (-.75,-1.5) to (-.75,1.5) to (1.25,1.5)--cycle;
 \draw[thick,-to] (-.75,-1.5) to (-.75,.1); \draw[thick] (-.75,0) to (-.75,1.5); 
\draw[thick,Orange,densely dashed] (.75,-1.5) to (.75, 1.5); 
\draw[thick, dotted] (-1.25,1.5) to (1.25,1.5);\draw[thick, dotted] (-1.25,-1.5) to (1.25,-1.5);
}} \endxy
\end{gather}

\begin{prop}\label{prop_adj_two}
 The following pairs of functors are adjoint pairs: $(F_{\uparrow},F_{\downarrow})$ and $(F_{\downarrow},F_{-} \circ F_{\uparrow})$.
\end{prop}
\begin{proof}
 Induction functor $F_{\uparrow}$ is left adjoint to the restriction functor $F_{\downarrow}$. Adjointness natural transformations come from standard bimodule homomorphisms 
 \begin{equation}\label{eq_map_mult} 
 \alpha_0\ : \ R\otimes_{R^s} R\lra R, \ \ f\otimes g \longmapsto fg, \ f,g\in R, 
 \end{equation}
 and 
 \[
 \beta_0 \ : \ R^s \lra {}_{R^s} R_{R^s}, \ \ 
 f\longmapsto f, \ f\in R^s. 
 \]  

These adjointness maps and the corresponding adjointness isotopy relations are shown in equation (\ref{eq_adj_zero}). 
We use $\alpha_0,\beta_0,$ etc. to denote both bimodule maps and the  corresponding natural transformations of functors ($\id_R$ stands for the identity functor in the category $R\dmod$).  
The composition $F_{\downarrow}\circ F_{\uparrow}$ is written $F_{\downarrow\uparrow}$, for brevity.  

\begin{equation}\label{eq_adj_zero}
\begin{split}
  \xy (0,-1.3)*{
\tikzdiagc[scale=.5]{
\fill[blue!10]  (-1,-1) to[out=75,in=180] (0,0) to[out=0,in=105] (1,-1)--cycle;
\draw[thick] (-1,-1) to[out=75,in=180] (0,0) to[out=0,in=105] (1,-1);
\draw[thick,-to] (.01,0) -- (.02,0);  
\draw[thick, dotted] (-1.25,1) to (1.25,1);\draw[thick, dotted] (-1.25,-1) to (1.25,-1);
\node at (1,-1.5) {\small \text{$F_{\downarrow}$}};
\node at (-1,-1.5) {\small \text{$F_{\uparrow}$}};
\node at (2,-1) {\small \text{$F_{\uparrow\downarrow}$}}; \node at (2,1) {\small \text{$\id_R$}}; 
\draw[-implies,double equal sign distance]  (2,-.4) to (2,.4);
\node at (2.7,0) {\small \text{$\alpha_0$}};
}} \endxy
\mspace{65mu}
  \xy (0,.45)*{
\tikzdiagc[scale=.5]{
\fill[blue!10]  (-1.25,1) to (1.25,1) to (1.25,-1) to (-1.25,-1) --cycle;
\fill[white]  (-1,1) to[out=-75,in=180] (0,0) to[out=0,in=-105] (1,1)--cycle;
\draw[thick] (-1,1) to[out=-75,in=180] (0,0) to[out=0,in=-105] (1,1);
\draw[thick,-to] (.01,0) -- (.02,0);  
\draw[thick, dotted] (-1.25,1) to (1.25,1);\draw[thick, dotted] (-1.25,-1) to (1.25,-1);
\node at (-1, 1.5) {\small \text{$F_{\downarrow}$}};
\node at (1, 1.5) {\small \text{$F_{\uparrow}$}};
\node at (2,-1) {\small \text{$\id_{R^s}$}}; \node at (2,1) {\small \text{$F_{\downarrow\uparrow}$}}; 
\draw[-implies,double equal sign distance]  (2,-.4) to (2,.4);
\node at (2.7,0) {\small \text{$\beta_0$}};
}} \endxy \mspace{15mu}
\\[1ex] 
\xy (0,.9)*{
\tikzdiagc[scale=.5,xscale=.75]{
\fill[blue!10]  (.5,-1) to (1.5,-1) to (1.5,1) to (1,1) to [out=-90,in=0] (.5,-.5)--cycle;
  \fill[blue!10]  (-1,-1) to[out=90,in=180] (-.5,.5) to[out=0,in=180] (.5,-.5) to (.5,-1)--cycle;
\draw[thick] (-1,-1) to[out=90,in=180] (-.5,.5) to[out=0,in=180] (.5,-.5) to[out=0,in=-90] (1,1);
\draw[thick,-to] (-1,-.15) -- (-1,-.1);  
\draw[thick, dotted] (-1.5,1) to (1.5,1);\draw[thick, dotted] (-1.5,-1) to (1.5,-1);
}} \endxy
=
\xy (0,0)*{
\tikzdiagc[scale=.5,xscale=.75]{
\fill[blue!10]  (1.5,-1) to (0,-1) to (0,1) to (1.5,1)--cycle;
\draw[thick,-to] (0,-1) to (0,0);
\draw[thick] (0,0) -- (0,1);  
\draw[thick, dotted] (-1.5,1) to (1.5,1);\draw[thick, dotted] (-1.5,-1) to (1.5,-1);
}} \endxy
\mspace{60mu}
\xy (0,0)*{
\tikzdiagc[scale=.5,xscale=.75,yscale=-1]{
\fill[blue!10]  (-1,-1) to[out=90,in=180] (-.5,.5) to (-.5,1) to (-1.5,1) to (-1.5,-1)--cycle;
\fill[blue!10]  (-.5,.5) to[out=0,in=180] (.5,-.5) to[out=0,in=-90] (1,1) to (-.5,1)--cycle;
\draw[thick] (-1,-1) to[out=90,in=180] (-.5,.5) to[out=0,in=180] (.5,-.5) to[out=0,in=-90] (1,1);
\draw[thick,-to] (-1,-.15) -- (-1,-.1);  
\draw[thick, dotted] (-1.5,1) to (1.5,1);\draw[thick, dotted] (-1.5,-1) to (1.5,-1);
}} \endxy
=
\xy (0,0)*{
\tikzdiagc[scale=.5,xscale=-.75,yscale=-1]{
\fill[blue!10]  (1.5,-1) to (0,-1) to (0,1) to (1.5,1)--cycle;
\draw[thick,-to] (0,-1) to (0,0);
\draw[thick] (0,0) -- (0,1);  
\draw[thick, dotted] (-1.5,1) to (1.5,1);\draw[thick, dotted] (-1.5,-1) to (1.5,-1);
}} \endxy
\end{split}
\end{equation}

For the second adjoint pair of functors, we consider corresponding bimodules and define bimodule homomorphisms 
 \[  \alpha_1 \ : \ {}_{R^s} R \otimes_R \undR \otimes_R R_{R^s} \  \lra \  R^s, \ \ \alpha_1(f \otimes \undone \otimes g ) = \partial(s(f)g). 
 \] 
 This is just the map $\partial''$ in (\ref{eq_part_two}), under the bimodule isomorphism $ {}_{R^s} R \otimes_R \undR \otimes_R R_{R^s}\cong {}_{R^s} \undR_{R^s} $. 
 \[  \beta_1 \ : \ R\lra \undR\otimes_R R \otimes_{R^s} R, \ \  \beta_1 (f ) = \undone \otimes x_1\otimes f - \undone \otimes 1 \otimes x_1 f, \ f\in R.
 \]
 Due to (\ref{eq_x_onetwo}),  
 \begin{equation*}
     \beta_1(f)= \undone \otimes (x_1\otimes 1 - 1\otimes x_1)f =  \undone \otimes (x_2\otimes 1 - 1\otimes x_2)f. 
 \end{equation*}
 To prove that $\beta_1$ is a well-defined bimodule map we check that $\ell_{x_i}(\beta_1(1))=r_{x_i}(\beta_1(1))$, where $\ell_{x}$, respectively $r_x$, denotes the left, respectively right, multiplication by $x$ in an $R$-bimodule: 
 \begin{align*}
 \ell_{x_1}(\beta(1))  & =   x_1 \undone \otimes x_1 \otimes 1 - x_1 \undone \otimes 1 \otimes x_1 = 
 \undone \otimes x_2x_1\otimes 1 - \undone \otimes x_2 \otimes x_1 \\
  & =  \undone \otimes ( 1 \otimes x_2x_1 - x_2\otimes x_1)  = \undone \otimes (1\otimes x_2 - x_2\otimes 1) x_1 \\ &=   \undone \otimes (1\otimes x_1 - x_1\otimes 1) x_1  = r_{x_1}(\beta(1)), \\
 \ell_{x_2}(\beta(1))  & =   x_2 \undone \otimes (x_2 \otimes 1 - 1 \otimes x_2) = 
 \undone \otimes (x_1x_2\otimes 1 - x_1 \otimes x_2) \\
  & =  \undone \otimes ( 1 \otimes x_1x_2 - x_1\otimes x_2)  = \undone \otimes 
  (1\otimes x_1 - x_1\otimes 1) x_2  =  r_{x_2}(\beta(1)). 
 \end{align*}
 We check the adjointness relation $(\alpha_1\otimes \id)\circ (\id \otimes \beta_1)=\id$, where $\id$ stands for the identity homomorphism of suitable bimodules: 
 \[  1\stackrel{\id\otimes \beta_1}{\longmapsto} 1\otimes \undone \otimes (x_1\otimes 1 - 1\otimes x_1) \stackrel{\alpha_1\otimes \id}{\longmapsto} 
 \partial(x_1)1 - \partial(1)x_1 = 1. 
 \] 
 To check the other adjointness relation
 $ (\id \otimes \alpha_1)\circ (\beta_1\otimes \id)=\id$ 
 we compute the corresponding endomorphism of the $(R,R^s)$-bimodule $\undR \otimes_R R_{R^s}$ (functor $F_-\circ F_{\uparrow}$ is given by tensoring with this bimodule): 
 \begin{align*}
 \undone \otimes 1  &\stackrel{\beta_1 \otimes \id}{\longmapsto}
 \undone \otimes (x_1\otimes 1-1\otimes x_1)\otimes \undone \otimes 1 \\ &\stackrel{\id \otimes \alpha_1}{\longmapsto} \undone\otimes x_1 \alpha_1(1\otimes \undone\otimes 1) - \undone \otimes 1 \alpha_1(x_1\otimes \undone\otimes 1) \\
 &\mspace{15mu}= \mspace{13mu}\undone\otimes x_1\partial(1) - \undone \otimes 1\partial(s(x_1))= 0 - \undone \otimes 1  (-1)= \undone \otimes 1. 
 \end{align*}

\end{proof}

Diagrams for maps $\alpha_1,\beta_1$ and their adjointness relations are depicted below. 
To shorten notations, we  write $F_{\downarrow-\uparrow}$ for $F_{\downarrow}\circ F_-\circ F_{\uparrow}$, etc. 
\begin{equation*}
\begin{split}
  \xy (0,1)*{
\tikzdiagc[scale=.6]{
\fill[blue!10]  (-1.25,1) to (-1.25,-1) to (-1,-1) to[out=75,in=180] (0,0) to (0,1)--cycle;
\fill[blue!10]  (0,1) to (0,0) to[out=0,in=105] (1,-1) to (1.25,-1)to (1.25,1) --cycle;
\draw[thick,Orange,densely dashed] (0,-1) to (0, 0); 
\draw[thick] (-1,-1) to[out=75,in=180] (0,0) to[out=0,in=105] (1,-1);
\draw[black,fill=black] (0,0) circle (1.5pt);
\draw[thick,-to] (-.65,-.3) -- (-.75,-.4);  
\draw[thick, dotted] (-1.25,1) to (1.25,1);\draw[thick, dotted] (-1.25,-1) to (1.25,-1);
\node at (2.9,-1) {\small \text{$F_{\downarrow -\uparrow}$}}; \node at (3.2,1) {\small \text{$\oneb_{R^s\dmod}$}}; 
\draw[-implies,double equal sign distance]  (2.9,-.4) to (2.9,.4);
\node at (3.6,0) {\small \text{$\alpha_1$}};
}} \endxy\mspace{50mu}
&
\mspace{50mu}
\xy (0,0)*{
\tikzdiagc[scale=.6]{
\fill[blue!10]  (-1,1) to[out=-75,in=180] (0,0) to[out=0,in=-105] (1,1)--cycle;
\draw[thick,Orange,densely dashed] (0,-.02) to (-.5,0) to[out=180,in=-90] (-1.5,1); 
\draw[thick] (-1,1) to[out=-75,in=180] (0,0) to[out=0,in=-105] (1,1);
\draw[black,fill=black] (0,0) circle (1.5pt);
\draw[thick,to-] (.65,.3) -- (.75,.4);  
\draw[thick, dotted] (-1.75,1) to (1.25,1);\draw[thick, dotted] (-1.75,-1) to (1.25,-1);
\node at (3,1) {\small \text{$F_{-\uparrow\downarrow}$}}; 
 \node at (3.2,-1) {\small \text{$\oneb_{R\dmod}$}};
\draw[-implies,double equal sign distance]  (2.9,-.4) to (2.9,.4);
\node at (3.6,0) {\small \text{$\beta_1$}};
}} \endxy
\\[1.5ex]
\xy (0,.8)*{
\tikzdiagc[scale=.6,xscale=.85]{
\fill[blue!10]  (-1,-1) to (-1.5,-1) to (-1.5,1) to (-.5,1) to (-.5,.5) to[out=180,in=90] (-1,-1)--cycle;
  \fill[blue!10]  (-.5,1) to (-.5,.5) to[out=0,in=180] (.5,-.5) to[out=0,in=-90] (1,1) --cycle;
\draw[thick,Orange,densely dashed] (.5,-.52) to (.25,-.5) to[out=180,in=-90] (-.5,.5); 
\draw[thick] (-1,-1) to[out=90,in=180] (-.5,.5) to[out=0,in=180] (.5,-.5) to[out=0,in=-90] (1,1);
\draw[black,fill=black] (-.5,.5) circle (1.5pt);
\draw[black,fill=black] (.5,-.5) circle (1.5pt);
\draw[thick,black,to-] (-1,-.35) -- (-1,-.3);  
\draw[thick, dotted] (-1.5,1) to (1.5,1);\draw[thick, dotted] (-1.5,-1) to (1.5,-1);
}} \endxy
=
\xy (0,0)*{
\tikzdiagc[scale=.6,xscale=.85]{
\fill[blue!10]  (-1.5,-1) to (0,-1) to (0,1) to (-1.5,1)--cycle;
\draw[thick] (0,-1) to (0,0);
\draw[thick,to-] (0,-.1) -- (0,1);  
\draw[thick, dotted] (-1.5,1) to (1.5,1);\draw[thick, dotted] (-1.5,-1) to (1.5,-1);
}} \endxy 
\mspace{40mu}
&
\mspace{50mu}
\xy (0,-1.7)*{
\tikzdiagc[scale=.6,xscale=.75,yscale=-1]{
\fill[blue!10]  (1,1) to (1.5,1) to (1.5,-1) to (.5,-1) to (.5,-.5) to[out=0,in=90] (1,1) --cycle;
\fill[blue!10]  (.5,-1) to (.5,-.5) to[out=180,in=0] (-.5,.5) to[out=180,in=90] (-1,-1)--cycle;
\draw[thick,Orange,densely dashed] (.5,-.5) to (.5, 1);
\draw[thick,Orange,densely dashed] (-.5,.48)  to (-.75,.5) to[out=180,in=90]  (-1.75, -1); 
\draw[thick] (-1,-1) to[out=90,in=180] (-.5,.5) to[out=0,in=180] (.5,-.5) to[out=0,in=-90] (1,1);
\draw[black,fill=black] (-.5, .5) circle (1.5pt);
\draw[black,fill=black] ( .5,-.5) circle (1.5pt);
\draw[thick,-to] (1,.15) -- (1,.1);  
\draw[thick, dotted] (-2,1) to (1.5,1);\draw[thick, dotted] (-2,-1) to (1.5,-1);
}} \endxy
=
\xy (0,0)*{
\tikzdiagc[scale=.6,xscale=-.75,yscale=-1]{
\fill[blue!10]  (-1.75,-1) to (0,-1) to (0,1) to (-1.75,1)--cycle;
\draw[thick,Orange,densely dashed] (1,-1) to (1, 1);
\draw[thick] (0,-1) to (0,0);
\draw[thick,to-] (0,-.1) -- (0,1);  
\draw[thick, dotted] (-1.75,1) to (1.5,1);\draw[thick, dotted] (-1.75,-1) to (1.5,-1);
}} \endxy
\\[-2ex]
\scriptstyle (\alpha_1\otimes \id_{F_{\downarrow}})(\id_{F_{\downarrow}}\otimes\beta_1)=\id_{F_{\downarrow}}
\mspace{50mu}
&
\mspace{40mu}
\scriptstyle (\id_{F_{-\uparrow}}\otimes \alpha_1)(\beta_1\otimes \id_{F_{-\uparrow}})=\id_{F_{-\uparrow}}
\end{split}
\end{equation*}

It is also  natural to define another ``cup'' morphism, with the dotted line entering the local minimum in the middle, see below,  as the composition of $\beta_1$ and the isomorphism (\ref{eq_iso_down}).  We call this cup \emph{balanced} and denote the morphism by $\widetilde{\beta}_1$.  Then the diagram representing $\beta_1$ can be rewritten as the composition of a balanced cup and a crossing, see below (the map $\psi\colon F_{-\uparrow}\to F_{\uparrow\minuss}$ is the crossing isomorphism in (\ref{eq_fig_crossings})). 

\begin{align*}
\xy (0,0)*{
\tikzdiagc[scale=.6]{
\fill[blue!10]  (-1,1) to[out=-75,in=180] (0,-.25) to[out=0,in=-105] (1,1)--cycle;
\draw[thick,Orange,densely dashed] (0,-.27) to (0,1); 
\draw[thick] (-1,1) to[out=-75,in=180] (0,-.25) to[out=0,in=-105] (1,1);
\draw[fill=black] (0,-.25) circle (1.5pt);
\draw[thick,to-] (.65,.085) -- (.75,.285);  
\draw[thick, dotted] (-1.5,1) to (1.25,1);\draw[thick, dotted] (-1.5,-1) to (1.25,-1);
}} \endxy
 := & 
\xy (0,0)*{
\tikzdiagc[scale=.6]{
\fill[blue!10]  (-1,1) to[out=-75,in=180] (0,-.25) to[out=0,in=-105] (1,1)--cycle;
\draw[thick,Orange,densely dashed] (0,-.27) to[out=180,in=-30] (-.8,-.1) to[out=135,in=-135] (-.75,.35) to[out=45,in=-90] (0,1); 
\draw[thick] (-1,1) to[out=-75,in=180] (0,-.25) to[out=0,in=-105] (1,1);
\draw[fill=black] (0,-.25) circle (1.5pt);
\draw[thick,to-] (.65,.085) -- (.75,.285);  
\draw[thick, dotted] (-1.5,1) to (1.25,1);\draw[thick, dotted] (-1.5,-1) to (1.25,-1);
\node at (2.4,1) {\small \text{$F_{\uparrow \minuss\downarrow}$}}; 
 \node at (2.4,-1) {\small \text{$\oneb_{R\dmod}$}};
\draw[-implies,double equal sign distance]  (2.4,-.4) to (2.4,.4);
\node at (4.0,0) {\small \text{$\widetilde{\beta}_1=\psi\circ\beta_1$}};
}} \endxy
\\[1ex] 
\xy (0,0)*{
\tikzdiagc[scale=.6]{
\fill[blue!10]  (-1,1) to[out=-75,in=180] (0,-.25) to[out=0,in=-105] (1,1)--cycle;
\draw[thick,Orange,densely dashed] (0,-.27) to (-.5,-.27) to[out=180,in=-90] (-1.5,1); 
\draw[thick] (-1,1) to[out=-75,in=180] (0,-.25) to[out=0,in=-105] (1,1);
\draw[fill=black] (0,-.25) circle (1.5pt);
\draw[thick,to-] (.65,.085) -- (.75,.285);  
\draw[thick, dotted] (-1.75,1) to (1.25,1);\draw[thick, dotted] (-1.75,-1) to (1.25,-1);
}} \endxy
= &
\xy (0,0)*{
\tikzdiagc[scale=.6]{
\fill[blue!10]  (-1,1) to[out=-75,in=180] (0,-.25) to[out=0,in=-105] (1,1)--cycle;
\draw[thick,Orange,densely dashed] (0,-.27) to[out=90,in=-90] (-1.5,1); \draw[thick] (-1,1) to[out=-75,in=180] (0,-.25) to[out=0,in=-105] (1,1);
\draw[fill=black] (0,-.25) circle (1.5pt);
\draw[thick,to-] (.65,.085) -- (.75,.285);  
\draw[thick, dotted] (-1.75,1) to (1.25,1);\draw[thick, dotted] (-1.75,-1) to (1.25,-1);
\node at (2.4,1) {\small \text{$F_{-\uparrow\downarrow}$}}; 
 \node at (2.4,-1) {\small \text{$\oneb_{R\dmod}$}};
\draw[-implies,double equal sign distance]  (2.4,-.4) to (2.4,.4);
\node at (4.3,0) {\small \text{$\beta_1=\psi^{-1}\circ\widetilde{\beta}_1$}};
}} \endxy
\end{align*}

The balanced cup morphism is 
\begin{equation*} \widetilde{\beta}_1 \colon R \lra R \otimes \undR^s \otimes R, \   \widetilde{\beta}_1 (f) = -x_2 \otimes \undone^s \otimes f - 1\otimes \undone^s \otimes x_1f = -x_1 \otimes \undone^s \otimes f - 1\otimes \undone^s \otimes x_2f. 
\end{equation*} 
As $\widetilde{\beta}_1$ is a bimodule map, $f\in R$ can also be placed on the far left in the formula. 

\begin{prop}\label{prop_adj_minus}
  Functors $F_{\uparrow\downarrow}=F_{\uparrow}\circ F_{\downarrow}$ and $F_{\uparrow \minuss\downarrow}=F_{\uparrow}\circ F_{\minuss}\circ F_{\downarrow}$ are biadjoint. 
\end{prop}

\begin{proof}
This follows from Proposition~\ref{prop_adj_two} and functor isomorphisms (\ref{eq_iso_down}) and (\ref{eq_iso_up}). 
\end{proof} 

The corresponding biadjointness maps are shown below. We fix these biadjointness maps $\alpha_2,\beta_2,\alpha_3,\beta_3$. 

\begin{equation}\label{eq_biadjoint_pics}
\begin{split}
  \xy (0,1)*{
\tikzdiagc[scale=.6]{
\fill[blue!10] (-1.5,-1) to[out=75,in=180] (0,.5) to[out=0,in=105] (1.5,-1)--cycle;
\fill[white] (-1,-1) to[out=75,in=180] (0,0) to[out=0,in=105] (1,-1)--cycle;
\draw[thick,Orange,densely dashed] (1.25,-1) to[out=90,in=-90] (0, 0); 
\draw[thick] (-1,-1) to[out=75,in=180] (0,0) to[out=0,in=105] (1,-1);
\draw[thick,to-] (-1.15,-.2) -- (-1.23,-.3);  
\draw[thick] (-1.5,-1) to[out=75,in=180] (0,.5) to[out=0,in=105] (1.5,-1);
\draw[fill=black] (0,0) circle (1.5pt);
\draw[thick,-to] (-.65,-.3) -- (-.75,-.4);  
\draw[thick, dotted] (-1.75,1) to (1.75,1);\draw[thick, dotted] (-1.75,-1) to (1.75,-1);
\node at (3.2,-1) {\small \text{$F_{\uparrow\downarrow}\circ F_{\uparrow\minuss\downarrow}$}}; \node at (3.2,1) {\small \text{$\id_{R}$}}; \draw[-implies,double equal sign distance]  (3.1,-.4) to (3.1,.4);
\node at (3.7,0) {\small \text{$\alpha_2$}};
}} \endxy\mspace{50mu}
&
\mspace{50mu}
  \xy (0,.25)*{
\tikzdiagc[scale=.6,xscale=-1,yscale=-1]{
\fill[blue!10] (-1.5,-1) to[out=75,in=180] (0,.5) to[out=0,in=105] (1.5,-1)--cycle;
\fill[white] (-1,-1) to[out=75,in=180] (0,0) to[out=0,in=105] (1,-1)--cycle;
\draw[thick,Orange,densely dashed] (1.25,-1) to[out=100,in=-55] (0, .5); 
\draw[thick] (-1,-1) to[out=75,in=180] (0,0) to[out=0,in=105] (1,-1);
\draw[thick,to-] (-1.15,-.2) -- (-1.23,-.3);  
\draw[thick] (-1.5,-1) to[out=75,in=180] (0,.5) to[out=0,in=105] (1.5,-1);
\draw[fill=black] (0,.5) circle (1.5pt);
\draw[thick,-to] (-.65,-.3) -- (-.75,-.4);  
\draw[thick, dotted] (-1.75,1) to (1.75,1);\draw[thick, dotted] (-1.75,-1) to (1.75,-1);
\node at (-3.2,-1) {\small \text{$F_{\uparrow\minuss\downarrow}\circ F_{\uparrow\downarrow}$}}; \node at (-3.2,1) {\small \text{$\id_{R}$}}; \draw[-implies,double equal sign distance]  (-3.1,.4) to (-3.1,-.4);
\node at (-3.7,0) {\small \text{$\beta_2$}};
}} \endxy\mspace{50mu}
\\[1.5ex]
  \xy (0,1)*{
\tikzdiagc[scale=.6]{
\fill[blue!10] (-1.5,-1) to[out=75,in=180] (0,.5) to[out=0,in=105] (1.5,-1)--cycle;
\fill[white] (-1,-1) to[out=75,in=180] (0,0) to[out=0,in=105] (1,-1)--cycle;
\draw[thick,Orange,densely dashed] (-1.25,-1) to[out=80,in=-90] (0, 0); \draw[thick] (-1,-1) to[out=75,in=180] (0,0) to[out=0,in=105] (1,-1);
\draw[thick,-to] (1.15,-.2) -- (1.23,-.3);  
\draw[thick] (-1.5,-1) to[out=75,in=180] (0,.5) to[out=0,in=105] (1.5,-1);
\draw[fill=black] (0,0) circle (1.5pt);
\draw[thick,to-] (.65,-.3) -- (.75,-.4);  
\draw[thick, dotted] (-1.75,1) to (1.75,1);\draw[thick, dotted] (-1.75,-1) to (1.75,-1);
\node at (3.2,-1) {\small \text{$F_{\uparrow\minuss\downarrow}\circ F_{\uparrow\downarrow}$}}; \node at (3.2,1) {\small \text{$\id_{R}$}}; 
\draw[-implies,double equal sign distance]  (3.1,-.4) to (3.1,.4);
\node at (3.7,0) {\small \text{$\alpha_3$}};
}} \endxy\mspace{50mu}
&
\mspace{50mu}
  \xy (0,.25)*{
\tikzdiagc[scale=.6,xscale=1,yscale=-1]{
\fill[blue!10] (-1.5,-1) to[out=75,in=180] (0,.5) to[out=0,in=105] (1.5,-1)--cycle;
\fill[white] (-1,-1) to[out=75,in=180] (0,0) to[out=0,in=105] (1,-1)--cycle;
\draw[thick,Orange,densely dashed] (1.25,-1) to[out=100,in=-55] (0, .5); 
\draw[thick] (-1,-1) to[out=75,in=180] (0,0) to[out=0,in=105] (1,-1);
\draw[thick,-to] (-1.15,-.2) -- (-1.23,-.3);  
\draw[thick] (-1.5,-1) to[out=75,in=180] (0,.5) to[out=0,in=105] (1.5,-1);
\draw[fill=black] (0,.5) circle (1.5pt);
\draw[thick,to-] (-.65,-.3) -- (-.75,-.4);  
\draw[thick, dotted] (-1.75,1) to (1.75,1);\draw[thick, dotted] (-1.75,-1) to (1.75,-1);
\node at (3.2,-1) {\small \text{$F_{\uparrow\downarrow}\circ F_{\uparrow\minuss\downarrow}$}}; \node at (3.2,1) {\small \text{$\id_{R}$}}; 
\draw[-implies,double equal sign distance]  (3.1,.4) to (3.1,-.4);
\node at (3.7,0) {\small \text{$\beta_3$}};
}} \endxy
\end{split}
\end{equation}

A quick computation shows that the two bimodule maps on the left of the diagram (\ref{eq_red_below}) below differ by a minus sign. We can then define a trivalent vertex, with a orange dashed line entering it from below, as in (\ref{eq_red_below})  on the right. This bimodule map is given by 
\begin{equation} \label{eq_red_blue_cup}\undR \lra  R\otimes_{R^s} R, \ \ \undone \lra 1\otimes x_2 - x_2\otimes 1  = x_1\otimes 1 - 1\otimes x_1. 
\end{equation} 
It may be interesting to compare our dashed line and vertex with dashed lines in Ellis-Lauda~\cite{EL}. 

\begin{equation}\label{eq_red_below}
\xy (0,0)*{
\tikzdiagc[scale=.6]{
\fill[blue!10]  (-1,1) to[out=-75,in=180] (0,0) to[out=0,in=-105] (1,1)--cycle;
\draw[thick,Orange,densely dashed] (0,0) to[out=90,in=45] (-.8,.5) to[out=-135,in=90] (0,-1); 
\draw[thick] (-1,1) to[out=-75,in=180] (0,0) to[out=0,in=-105] (1,1);
\draw[fill=black] (0,0) circle (1.5pt);
\draw[thick,to-] (.45,.1) -- (.55,.2);  
\draw[thick, dotted] (-1.25,1) to (1.25,1);\draw[thick, dotted] (-1.25,-1) to (1.25,-1);
}} \endxy
\ = - \ 
\xy (0,0)*{
\tikzdiagc[scale=.6]{
\fill[blue!10]  (-1,1) to[out=-75,in=180] (0,0) to[out=0,in=-105] (1,1)--cycle;
\draw[thick,Orange,densely dashed] (0,0) to[out=90,in=135] (.8,.5) to[out=-45,in=90] (0,-1); 
\draw[thick] (-1,1) to[out=-75,in=180] (0,0) to[out=0,in=-105] (1,1);
\draw[fill=black] (0,0) circle (1.5pt);
\draw[thick,to-] (.45,.1) -- (.55,.2);  
\draw[thick, dotted] (-1.25,1) to (1.25,1);\draw[thick, dotted] (-1.25,-1) to (1.25,-1);
}} \endxy
\mspace{80mu}
\xy (0,0)*{
\tikzdiagc[scale=.6]{
\fill[blue!10]  (-1,1) to[out=-75,in=180] (0,0) to[out=0,in=-105] (1,1)--cycle;
\draw[thick,Orange,densely dashed] (0,0) to (0,-1); 
\draw[thick] (-1,1) to[out=-75,in=180] (0,0) to[out=0,in=-105] (1,1);
\draw[fill=black] (0,0) circle (1.5pt);
\draw[thick,to-] (.65,.3) -- (.75,.4);  
\draw[thick, dotted] (-1.25,1) to (1.25,1);\draw[thick, dotted] (-1.25,-1) to (1.25,-1);
}} \endxy
\ :=\
\xy (0,0)*{
\tikzdiagc[scale=.6]{
\fill[blue!10]  (-1,1) to[out=-75,in=180] (0,0) to[out=0,in=-105] (1,1)--cycle;
\draw[thick,Orange,densely dashed] (0,0) to[out=90,in=45] (-.8,.5) to[out=-135,in=90] (0,-1); 
\draw[thick] (-1,1) to[out=-75,in=180] (0,0) to[out=0,in=-105] (1,1);
\draw[fill=black] (0,0) circle (1.5pt);
\draw[thick,to-] (.45,.1) -- (.55,.2);  
\draw[thick, dotted] (-1.25,1) to (1.25,1);\draw[thick, dotted] (-1.25,-1) to (1.25,-1);
}} \endxy
\end{equation}

Likewise, there's a sign in a similar relation given by reflecting these diagrams about a horizontal axis and reversing the shading of regions, see~\eqref{eq:refletriv} below left. One can then define the reflected trivalent vertex as in the figure in~\eqref{eq:refletriv} below, on the right. This bimodule map is $\partial'$, see formula (\ref{eq_d_prime}), 
\begin{equation*}
    R\lra \undR^s, \ \ f\longmapsto \undone^s \partial(f), \ \ f\in R. 
\end{equation*}

\begin{equation}\label{eq:refletriv}
\xy (0,0)*{
\tikzdiagc[scale=.6,yscale=-1]{
\fill[blue!10] (-1.25,-1) to (1.25,-1) to (1.25,1) to (-1.25,1)--cycle;
\fill[white]  (-1,1) to[out=-75,in=180] (0,0) to[out=0,in=-105] (1,1)--cycle;
\draw[thick,Orange,densely dashed] (0,0) to[out=90,in=45] (-.8,.5) to[out=-135,in=90] (0,-1); 
\draw[thick] (-1,1) to[out=-75,in=180] (0,0) to[out=0,in=-105] (1,1);
\draw[fill=black] (0,0) circle (1.5pt);
\draw[thick,to-] (.45,.1) -- (.55,.2);  
\draw[thick, dotted] (-1.25,1) to (1.25,1);\draw[thick, dotted] (-1.25,-1) to (1.25,-1);
}} \endxy
\ = - \ 
\xy (0,0)*{
\tikzdiagc[scale=.6,yscale=-1]{
\fill[blue!10] (-1.25,-1) to (1.25,-1) to (1.25,1) to (-1.25,1)--cycle;
\fill[white]  (-1,1) to[out=-75,in=180] (0,0) to[out=0,in=-105] (1,1)--cycle;
\draw[thick,Orange,densely dashed] (0,0) to[out=90,in=135] (.8,.5) to[out=-45,in=90] (0,-1); 
\draw[thick] (-1,1) to[out=-75,in=180] (0,0) to[out=0,in=-105] (1,1);
\draw[fill=black] (0,0) circle (1.5pt);
\draw[thick,to-] (.45,.1) -- (.55,.2);  
\draw[thick, dotted] (-1.25,1) to (1.25,1);\draw[thick, dotted] (-1.25,-1) to (1.25,-1);
}} \endxy
\mspace{80mu}
\xy (0,0)*{
\tikzdiagc[scale=.6,yscale=-1]{
\fill[blue!10] (-1.25,-1) to (1.25,-1) to (1.25,1) to (-1.25,1)--cycle;
\fill[white]  (-1,1) to[out=-75,in=180] (0,0) to[out=0,in=-105] (1,1)--cycle;
\draw[thick,Orange,densely dashed] (0,0) to (0,-1); 
\draw[thick] (-1,1) to[out=-75,in=180] (0,0) to[out=0,in=-105] (1,1);
\draw[fill=black] (0,0) circle (1.5pt);
\draw[thick,to-] (.65,.3) -- (.75,.4);  
\draw[thick, dotted] (-1.25,1) to (1.25,1);\draw[thick, dotted] (-1.25,-1) to (1.25,-1);
}} \endxy
\ :=\
\xy (0,0)*{
\tikzdiagc[scale=.6,yscale=-1]{
\fill[blue!10] (-1.25,-1) to (1.25,-1) to (1.25,1) to (-1.25,1)--cycle;
\fill[white]  (-1,1) to[out=-75,in=180] (0,0) to[out=0,in=-105] (1,1)--cycle;
\draw[thick,Orange,densely dashed] (0,0) to[out=90,in=45] (-.8,.5) to[out=-135,in=90] (0,-1); 
\draw[thick] (-1,1) to[out=-75,in=180] (0,0) to[out=0,in=-105] (1,1);
\draw[fill=black] (0,0) circle (1.5pt);
\draw[thick,to-] (.45,.1) -- (.55,.2);  
\draw[thick, dotted] (-1.25,1) to (1.25,1);\draw[thick, dotted] (-1.25,-1) to (1.25,-1);
}} \endxy
\end{equation}

Some other relations in this graphical calculus are shown below. 
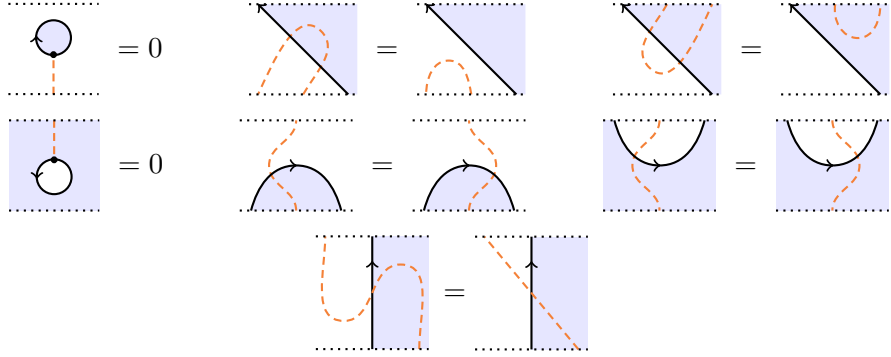
\begingroup\allowdisplaybreaks
\begin{gather*}
\xy (0,1)*{
\tikzdiagc[scale=.6]{
\draw[thick,Orange,densely dashed] (0,-.25) to (0,-1); 
\draw[fill=blue!10] (0,.25) circle (11pt);
\draw[fill=black] (0,-.125) circle (1.5pt);
\draw[thick,to-] (-.38,.3) -- (-.38,.2);  
\draw[thick, dotted] (-1,1) to (1,1);\draw[thick, dotted] (-1,-1) to (1,-1);
}} \endxy
\ =0
\mspace{53mu}
\xy (0,0)*{
\tikzdiagc[scale=.6]{
\fill[blue!10] (1.2,-1)to (1,-1) to (-1,1) to (1,1) to (1.2,1)--cycle;
\draw[thick,Orange,densely dashed] (0,-1) to[out=60,in=-45] (.4,.4) to[out=135,in=65] (-1,-1); 
\draw[thick,-to] (1,-1) to (-1,1);
\draw[thick, dotted] (-1.2,1) to (1.2,1);\draw[thick, dotted] (-1.2,-1) to (1.2,-1);
}} \endxy
\ =\
\xy (0,0)*{
\tikzdiagc[scale=.6]{
\fill[blue!10] (1.2,-1)to (1,-1) to (-1,1) to (1,1) to (1.2,1)--cycle;
\draw[thick,Orange,densely dashed] (0,-1) to[out=100,in=0] (-.5,-.25) to[out=180,in=90] (-1,-1); 
\draw[thick,-to] (1,-1) to (-1,1);
\draw[thick, dotted] (-1.2,1) to (1.2,1);\draw[thick, dotted] (-1.2,-1) to (1.2,-1);
}} \endxy
\mspace{53mu}
\xy (0,0)*{
\tikzdiagc[scale=.6]{
\fill[blue!10] (1.2,-1)to (1,-1) to (-1,1) to (1,1) to (1.2,1)--cycle;
\draw[thick,Orange,densely dashed] (0,1) to[out=240,in=135] (-.4,-.4) to[out=-45,in=245] (1,1); 
\draw[thick,-to] (1,-1) to (-1,1);
\draw[thick, dotted] (-1.2,1) to (1.2,1);\draw[thick, dotted] (-1.2,-1) to (1.2,-1);
}} \endxy
\ =\
\xy (0,0)*{
\tikzdiagc[scale=.6]{
\fill[blue!10] (1.2,-1)to (1,-1) to (-1,1) to (1,1) to (1.2,1)--cycle;
\draw[thick,Orange,densely dashed] (0,1) to[out=-100,in=180] (.5,.25) to[out=0,in=-90] (1,1); 
\draw[thick,-to] (1,-1) to (-1,1);
\draw[thick, dotted] (-1.2,1) to (1.2,1);\draw[thick, dotted] (-1.2,-1) to (1.2,-1);
}} \endxy
\\[1ex]
\xy (0,1)*{
\tikzdiagc[scale=.6,yscale=-1]{
\fill[blue!10] (-1,1) to (1,1) to (1,-1) to (-1,-1) --cycle;
\draw[thick,Orange,densely dashed] (0,-.25) to (0,-1); 
\draw[fill=white] (0,.25) circle (11pt);
\draw[fill=black] (0,-.125) circle (1.5pt);
\draw[thick,to-] (-.38,.3) -- (-.38,.2);  
\draw[thick, dotted] (-1,1) to (1,1);\draw[thick, dotted] (-1,-1) to (1,-1);
}} \endxy
\ =0
\mspace{47mu}
\xy (0,0)*{
\tikzdiagc[scale=.6,yscale=-1]{
\fill[blue!10]  (-1,1) to[out=-75,in=180] (0,0) to[out=0,in=-105] (1,1)--cycle;
\draw[thick,Orange,densely dashed] (0,1) to[out=-90,in=90] (-.6,0) to[out=-90,in=90] (0,-1); 
\draw[thick] (-1,1) to[out=-75,in=180] (0,0) to[out=0,in=-105] (1,1);
\draw[thick,-to] (0,0) -- (.01,0);  
\draw[thick, dotted] (-1.25,1) to (1.25,1);\draw[thick, dotted] (-1.25,-1) to (1.25,-1);
}} \endxy
\ =  \ 
\xy (0,0)*{
\tikzdiagc[scale=.6,yscale=-1]{
\fill[blue!10]  (-1,1) to[out=-75,in=180] (0,0) to[out=0,in=-105] (1,1)--cycle;
\draw[thick,Orange,densely dashed] (0,1) to[out=-90,in=90] (.6,0) to[out=-90,in=90] (0,-1); 
\draw[thick] (-1,1) to[out=-75,in=180] (0,0) to[out=0,in=-105] (1,1);
\draw[thick,-to] (0,0) -- (.01,0);  
\draw[thick, dotted] (-1.25,1) to (1.25,1);\draw[thick, dotted] (-1.25,-1) to (1.25,-1);
}} \endxy
\mspace{47mu}
\xy (0,0)*{
\tikzdiagc[scale=.6,yscale=1]{
\fill[blue!10 ](-1.25,1) to (1.25,1) to (1.25,-1) to (-1.25,-1)--cycle;
\fill[white]  (-1,1) to[out=-75,in=180] (0,0) to[out=0,in=-105] (1,1)--cycle;
\draw[thick,Orange,densely dashed] (0,1) to[out=-90,in=90] (-.6,0) to[out=-90,in=90] (0,-1); 
\draw[thick] (-1,1) to[out=-75,in=180] (0,0) to[out=0,in=-105] (1,1);
\draw[thick,-to] (.01,0) -- (.02,0);  
\draw[thick, dotted] (-1.25,1) to (1.25,1);\draw[thick, dotted] (-1.25,-1) to (1.25,-1);
}} \endxy
\ =\
\xy (0,0)*{
\tikzdiagc[scale=.6,yscale=1]{
\fill[blue!10 ](-1.25,1) to (1.25,1) to (1.25,-1) to (-1.25,-1)--cycle;
\fill[white]  (-1,1) to[out=-75,in=180] (0,0) to[out=0,in=-105] (1,1)--cycle;
\draw[thick,Orange,densely dashed] (0,1) to[out=-90,in=90] (.6,0) to[out=-90,in=90] (0,-1); 
\draw[thick] (-1,1) to[out=-75,in=180] (0,0) to[out=0,in=-105] (1,1);
\draw[thick,-to] (.01,0) -- (.02,0);  
\draw[thick, dotted] (-1.25,1) to (1.25,1);\draw[thick, dotted] (-1.25,-1) to (1.25,-1);
}} \endxy
\\[1ex]
\xy (0,1)*{
\tikzdiagc[scale=.5]{
\fill[blue!10] (1.5,-1.5) to (0,-1.5) to (0,1.5) to (1.5,1.5)--cycle;
\draw[thick,-to] (0,-1.5) to (0,.85); \draw[thick] (0,.80) to (0,1.5); 
\draw[thick,Orange,densely dashed] (1.25,-1.5) to[out=90,in=0] (.75,.75) to[out=180,in=0] (-.75,-.75) to[out=180,in=-90] (-1.25,1.5); 
\draw[thick, dotted] (-1.5,1.5) to (1.5,1.5);\draw[thick, dotted] (-1.5,-1.5) to (1.5,-1.5);
}} \endxy
=
\xy (0,0)*{
\tikzdiagc[scale=.5]{
\fill[blue!10] (1.5,-1.5) to (0,-1.5) to (0,1.5) to (1.5,1.5)--cycle;
\draw[thick,-to] (0,-1.5) to (0,.85); \draw[thick] (0,.80) to (0,1.5); 
\draw[thick,Orange,densely dashed] (1.25,-1.5) to (-1.25,1.5); 
\draw[thick, dotted] (-1.5,1.5) to (1.5,1.5);\draw[thick, dotted] (-1.5,-1.5) to (1.5,-1.5);
}} \endxy
\end{gather*}\endgroup

Recall that a vertical dotted orange line on a white (respectively, blue or shaded) background denotes the identity map of $R$-bimodule $\undR$ (respectively, of $\undR^s$-bimodule $\undR^s$), see below. 
\begin{equation}\label{eq:idBP_2}
\xy (0,0)*{
\tikzdiagc[scale=.5]{
\draw[very thick,Orange,densely dashed] (0,-1) -- (0, 1);
\draw[thick, dotted] (-1,1) to (1,1);
\draw[thick, dotted] (-1,-1) to (1,-1);
}} \endxy = \id_{\undR} , 
\mspace{40mu}
\xy (0,0)*{
\tikzdiagc[scale=.5]{
\fill[blue!10 ](-1,1) to (1,1) to (1,-1) to (-1,-1)--cycle;
\draw[very thick,Orange,densely dashed] (0,-1) -- (0, 1);
\draw[thick, dotted] (-1,1) to (1,1);
\draw[thick, dotted] (-1,-1) to (1,-1);
}} \endxy = \id_{\undR^s}.
\end{equation}
The isomorphism (\ref{eq_iso_two_undR_1}) between $\undR\otimes_R \undR$ and $R$ can be represented by orange dashed ``cup'' and ``cap'' maps, see below. These maps satisfy the following relations
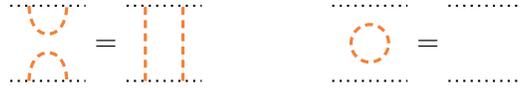
\begin{equation*} 
\xy (0,0)*{
\tikzdiagc[scale=.5]{
\draw[very thick,Orange,densely dashed] (-.5,1) to[out=-90,in=180] (0,.25) to[out=0,in=-90] (.5, 1);
\draw[very thick,Orange,densely dashed] (-.5,-1) to[out=90,in=180] (0,-.25) to[out=0,in=90] (.5,-1);
\draw[thick, dotted] (-1,1) to (1,1);
\draw[thick, dotted] (-1,-1) to (1,-1);
}}  \endxy = 
\xy (0,0)*{
\tikzdiagc[scale=.5]{
\draw[very thick,Orange,densely dashed] (-.5,-1) -- (-.5, 1);
\draw[very thick,Orange,densely dashed] (.5,-1) -- (.5, 1);
\draw[thick, dotted] (-1,1) to (1,1);
\draw[thick, dotted] (-1,-1) to (1,-1);
}} \endxy
\mspace{80mu}
\xy (0,0)*{
\tikzdiagc[scale=.5]{
\draw[very thick,Orange,densely dashed] (0,0) circle (.5);
\draw[thick, dotted] (-1,1) to (1,1);
\draw[thick, dotted] (-1,-1) to (1,-1);
}} \endxy
=
\xy (0,0)*{
\tikzdiagc[scale=.5]{
\draw[thick, dotted] (-1,1) to (1,1);
\draw[thick, dotted] (-1,-1) to (1,-1);
}} \endxy
\end{equation*} 
as well as the isotopy relations on the cup and the cap. 
Orange dashed cup and cap maps have degree $0$. Isomorphism  (\ref{eq_iso_two_undR_2}) is represented by oranged dashed cup and cap maps on a blue (shaded) background, with the following relations: 
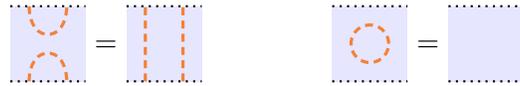
\begin{equation*}
\xy (0,0)*{
\tikzdiagc[scale=.5]{
\fill[blue!10 ](-1,1) to (1,1) to (1,-1) to (-1,-1)--cycle;
\draw[very thick,Orange,densely dashed] (-.5,1) to[out=-90,in=180] (0,.25) to[out=0,in=-90] (.5, 1);
\draw[very thick,Orange,densely dashed] (-.5,-1) to[out=90,in=180] (0,-.25) to[out=0,in=90] (.5,-1);
\draw[thick, dotted] (-1,1) to (1,1);
\draw[thick, dotted] (-1,-1) to (1,-1);
}} \endxy
=
\xy (0,0)*{
\tikzdiagc[scale=.5]{
\fill[blue!10 ](-1,1) to (1,1) to (1,-1) to (-1,-1)--cycle;
\draw[very thick,Orange,densely dashed] (-.5,-1) -- (-.5, 1);
\draw[very thick,Orange,densely dashed] (.5,-1) -- (.5, 1);
\draw[thick, dotted] (-1,1) to (1,1);
\draw[thick, dotted] (-1,-1) to (1,-1);
}} \endxy    
\mspace{80mu}
\xy (0,0)*{
\tikzdiagc[scale=.5]{
\fill[blue!10 ](-1,1) to (1,1) to (1,-1) to (-1,-1)--cycle;
\draw[very thick,Orange,densely dashed] (0,0) circle (.5);
\draw[thick, dotted] (-1,1) to (1,1);
\draw[thick, dotted] (-1,-1) to (1,-1);
}} \endxy
=
\xy (0,0)*{
\tikzdiagc[scale=.5]{
\fill[blue!10 ](-1,1) to (1,1) to (1,-1) to (-1,-1)--cycle;
\draw[thick, dotted] (-1,1) to (1,1);
\draw[thick, dotted] (-1,-1) to (1,-1);
}} \endxy
\end{equation*} 
and the isotopy relations. 

Likewise, isomorphisms (\ref{eq_iso_two_undR_3}) and (\ref{eq_iso_two_undR_4}) are represented by the orange dashed cup and cap diagrams in white-blue (or white-shaded) regions, as shown below for the isomorphisms (\ref{eq_iso_two_undR_3}), together with suitable relations on them, including isotopies.  

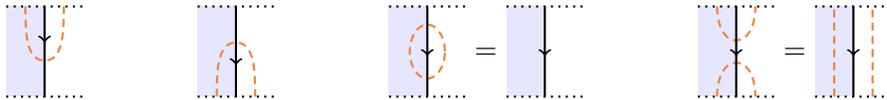
\begin{equation*}
\xy (0,0)*{
\tikzdiagc[scale=.6,xscale=.85]{
\fill[blue!10]  (-1,-1) to (0,-1) to (0,1) to (-1,1)--cycle;
\draw[thick,Orange,densely dashed] (-.5,1) to[out=-90,in=180] (0,-.2) to[out=0,in=-90] (.5,1); 
\draw[thick] (0,-1) to (0,1);
\draw[thick,to-] (0,.2) -- (0,.25);  
\draw[thick, dotted] (-1,1) to (1,1);\draw[thick, dotted] (-1,-1) to (1,-1);
}} \endxy
\mspace{70mu}
\xy (0,0)*{
\tikzdiagc[scale=.6,xscale=.85]{
\fill[blue!10]  (-1,-1) to (0,-1) to (0,1) to (-1,1)--cycle;
\draw[thick,Orange,densely dashed] (-.5,-1) to[out=90,in=180] (0,.2) to[out=0,in=90] (.5,-1); 
\draw[thick] (0,-1) to (0,1);
\draw[thick,to-] (0,-.3) -- (0,-.25);  
\draw[thick, dotted] (-1,1) to (1,1);\draw[thick, dotted] (-1,-1) to (1,-1);
}} \endxy
\mspace{70mu}
\xy (0,0)*{
\tikzdiagc[scale=.6,xscale=.85]{
\fill[blue!10]  (-1,-1) to (0,-1) to (0,1) to (-1,1)--cycle;
\draw[thick, Orange, densely dashed] (0,0) ellipse (13 pt and 17 pt);
\draw[thick] (0,-1) to (0,1);
\draw[thick,to-] (0,-.1) -- (0,0);  
\draw[thick, dotted] (-1,1) to (1,1);\draw[thick, dotted] (-1,-1) to (1,-1);
}} \endxy
=
\xy (0,0)*{
\tikzdiagc[scale=.6,xscale=.85]{
\fill[blue!10]  (-1,-1) to (0,-1) to (0,1) to (-1,1)--cycle;
\draw[thick] (0,-1) to (0,1);
\draw[thick,to-] (0,-.1) -- (0,0);  
\draw[thick, dotted] (-1,1) to (1,1);\draw[thick, dotted] (-1,-1) to (1,-1);
}} \endxy
\mspace{70mu}
\xy (0,0)*{
\tikzdiagc[scale=.6,xscale=.85]{
\fill[blue!10]  (-1,-1) to (0,-1) to (0,1) to (-1,1)--cycle;
\draw[thick,Orange,densely dashed] (-.5,1) to[out=-90,in=180] (0,.25) to[out=0,in=-90] (.5,1); 
\draw[thick,Orange,densely dashed] (-.5,-1) to[out=90,in=180] (0,-.25) to[out=0,in=90] (.5,-1); 
\draw[thick] (0,-1) to (0,1);
\draw[thick,to-] (0,-.1) -- (0,0);  
\draw[thick, dotted] (-1,1) to (1,1);\draw[thick, dotted] (-1,-1) to (1,-1);
}} \endxy
=
\xy (0,0)*{
\tikzdiagc[scale=.6,xscale=.85]{
\fill[blue!10]  (-1,-1) to (0,-1) to (0,1) to (-1,1)--cycle;
\draw[thick,Orange,densely dashed] (-.5,-1) to (-.5,1); 
\draw[thick,Orange,densely dashed] ( .5,-1) to ( .5,1); 
\draw[thick] (0,-1) to (0,1);
\draw[thick,to-] (0,-.1) -- (0,0);  
\draw[thick, dotted] (-1,1) to (1,1);\draw[thick, dotted] (-1,-1) to (1,-1);
}} \endxy
\end{equation*}

All eight possible orange dashed cup and cap maps have degree $0$.


\subsection{Bimodules \texorpdfstring{$B$}{B} and \texorpdfstring{$\undB$}{B} and their tensor products }

Define graded $R$-bimodules
\begin{equation}\label{eq_bimodules_B}
B\ := \ R\otimes_{R^s} R \{-1\}, \  \ \undB \ := \ R\otimes_{R^s} \undR^s \otimes_{R^s} R\{-1\}=R\undotimes_s R\{-1\}. 
\end{equation} 
We use a shorthand and denote $M\otimes_{R^s} \undR^s \otimes_{R^s}N$ (respectively, its element $m\otimes \undone^s \otimes n$) by $M\undotimes_{s} N$ (respectively, by $m\undotimes_s n$). Likewise, $M\otimes_R \undR \otimes_R N$ (and its element $m\otimes \undone \otimes n$) can be denoted $M\undotimes N$ (and by  $m\undotimes n$). 

Endofunctors $F_{\uparrow\downarrow}$ and $F_{\uparrow - \downarrow}$ of the category $R\dmod$ of graded $R$-modules are given by tensoring with bimodules $B\{1\}$ and $\undB\{1\}$, respectivley. Natural transformations $\alpha_2,\beta_2,\alpha_3, \beta_3$ can then be rewritten as bimodule maps, denoted the same (the tensor products are over $R$ and $f,g\in R$): 
\begin{align*}
  \alpha_2 & :  B\otimes \undB \lra R, \ \alpha_2(1\otimes_s f \otimes g \undotimes_s 1) =   \partial(s(fg)) ,  \\
  \beta_2 & :  R\lra \undB\otimes B, \  \beta_2(f) =   -(x_1 \undotimes_s 1) \otimes (1 \otimes_s f) - (1\undotimes_s 1)\otimes (1 \otimes_s x_2f),  \\
  \alpha_3 & :  \undB \otimes B \lra R, \ 
  \alpha_3(1\undotimes_s 1 \otimes f \otimes_s g) = \partial(f)g,  \\
  \beta_3 & :  R\lra B\otimes \undB, \  \beta_3(f) =  -(x_1 \otimes_s 1) \otimes (1 \undotimes_s f) - (1\otimes_s 1)\otimes (1 \undotimes_s x_2f) . 
\end{align*}
All four maps have zero degree: 
$\deg(\alpha_2)=\deg(\beta_2)=\deg(\alpha_3)=\deg(\beta_3)=0$.

We fix graded bimodule isomorphisms 
\begin{equation}\label{eq_blue_red} 
    \undR \otimes_R B \ \cong \ \undB \ \cong \   B\otimes_R \undR
\end{equation}
given by 
\begin{equation*}
    \undone \otimes f\otimes g \ \longmapsto \ s(f) \otimes \undone^s \ \otimes g, \ \ 
    f\otimes \undone^s \otimes g \ \longmapsto \ f\otimes s(g) \otimes \undone,  \  \ f,g\in R.  
\end{equation*}

We depict the identity maps of $B$  by a blue line 
\begin{equation}\label{eq:idBP}
\xy (0,0)*{
\tikzdiagc[scale=.5]{
  \draw[ultra thick,blue] (0,-1) -- (0, 1);
\draw[thick, dotted] (-1,1) to (1,1);
\draw[thick, dotted] (-1,-1) to (1,-1);
}} \endxy = \id_B
\end{equation}

Bimodule maps (\ref{eq_map_mult}) and (\ref{eq_red_blue_cup}) are given by the following diagrams (we write $f\otimes_s g$ for $f\otimes_{R^s}g$ and $f\otimes g = 1\otimes fg = fg\otimes 1$ for $f\otimes_R g$): 
\begin{align}\label{eq:enddot}
\xy (0,1)*{
\tikzdiagc[scale=.5]{
  \draw[ultra thick,blue] (0,-1) -- (0, 0)node[pos=1, tikzdot]{};
\draw[thick, dotted] (-1,1) to (1,1);
\draw[thick, dotted] (-1,-1) to (1,-1);
}} \endxy\ 
&\colon B\stackrel{m}{\lra} R , \mspace{20mu}
f\otimes_s g\mapsto fg ,
\\[1ex] \label{eq:startdot}
\xy (0,1)*{
\tikzdiagc[scale=.5]{
\draw[very thick, densely dashed,Orange] (0,-1) -- (0,0);
\draw[ultra thick,blue] (0,0) -- (0, 1)node[pos=0, tikzdot]{};
\draw[thick, dotted] (-1,1) to (1,1);
\draw[thick, dotted] (-1,-1) to (1,-1);
}} \endxy\ 
&\colon \undR \stackrel{\Delta}{\lra} B , \mspace{20mu}
\undone \mapsto 1\otimes_s x_2 - x_2\otimes_s 1= 1\otimes_s x_1 - x_1\otimes_s 1.
\end{align}
Due to our definition (\ref{eq_bimodules_B}) of the graded bimodule $B$, both of these maps have degree $1$. 
The maps in equations~\eqref{eq:enddot} and~\eqref{eq:startdot} fit into a short exact sequence 
\begin{equation} \label{exact_seq_1} 
0 \lra 
\undR\{1\} \xrightarrow{
\xy (0,0)*{
\tikzdiagc[scale=.35]{
\draw[very thick, densely dashed,Orange] (0,-1) -- (0,0);
\draw[ultra thick,blue] (0,0) -- (0, 1)node[pos=0, tikzdot]{};
\draw[thick, dotted] (-1,1) to (1,1);
\draw[thick, dotted] (-1,-1) to (1,-1);
}} \endxy
} B \xrightarrow{
\xy (0,0)*{
\tikzdiagc[scale=.35]{
  \draw[ultra thick,blue] (0,-1) -- (0, 0)node[pos=1, tikzdot]{};
\draw[thick, dotted] (-1,1) to (1,1);
\draw[thick, dotted] (-1,-1) to (1,-1);
}} \endxy
}R\{-1\} \lra 0 ,
\end{equation} 
where we shifted the gradings of the left and right terms to make the differential grading-preserving.  

Tensoring this sequence with $\undR$ gives another exact sequence, since $\undR$ is a free left and right $R$-module:  
\begin{equation}\label{eq:SESBuB}
0 \lra R\{1\} \xrightarrow{ 
\xy (0,0)*{
\tikzdiagc[scale=.35]{
\draw[very thick,densely dashed,Orange] (-.5,0) to[out=-90,in=180] (-.1,-.75) to[out=0,in=-90] (.5,1);
\draw[ultra thick,blue] (-.5,0) -- (-.5, 1)node[pos=0, tikzdot]{};
\draw[thick, dotted] (-1,1) to (1,1);
\draw[thick, dotted] (-1,-1) to (1,-1);
}} \endxy
} B\otimes_R \undR \xrightarrow{ 
\xy (0,0)*{
\tikzdiagc[scale=.35]{
\draw[ultra thick,blue] (-.5,-1) -- (-.5, 0)node[pos=1, tikzdot]{};
\draw[very thick, densely dashed,Orange] (.5,-1) -- (.5,1);
\draw[thick, dotted] (-1,1) to (1,1);
\draw[thick, dotted] (-1,-1) to (1,-1);
}} \endxy
} \undR\{-1\} \lra 0 .
\end{equation}
The middle term in the second sequence is isomorphic to $\undB$. Here and later we fix the isomorphism $B\otimes_R \undR\cong \undB$ given by \begin{equation*}
\xy (0,0)*{
\tikzdiagc[scale=.6]{
\fill[blue!10] (.5,-1) to (-.5,-1) to (-.5,1) to (.5,1)--cycle;
\draw[thick,Orange,densely dashed] (1,-1) to[out=90,in=0] (.5,0) to[out=180,in=-90] (0,1); 
\draw[thick,-to] (-.5,-1) to (-.5,.1); \draw[thick] (-.5,0) to (-.5,1); 
\draw[thick] (.5,-1) to (.5,-.5); \draw[thick,to-] (.5,-.55) to (.5,1); 
\draw[thick, dotted] (-1,1) to (1,1);\draw[thick, dotted] (-1,-1) to (1,-1);
}} \endxy
\end{equation*}

Exactness of sequences (\ref{exact_seq_1}) and (\ref{eq:SESBuB})  implies relations
\[
\xy (0,0)*{
\tikzdiagc[scale=.5]{
\draw[very thick, densely dashed,Orange] (0,-1) -- (0,0);
\draw[ultra thick,blue] (0,-.25) -- (0, .55)node[pos=0, tikzdot]{}node[pos=1, tikzdot]{};
\draw[thick, dotted] (-1,1) to (1,1);
\draw[thick, dotted] (-1,-1) to (1,-1);
}} \endxy
=
0
\mspace{80mu}
\xy (0,0)*{
\tikzdiagc[scale=.5]{
\draw[very thick,densely dashed,Orange] (-.5,0) to[out=-90,in=180] (-.1,-.75) to[out=0,in=-90] (.5,1);
\draw[ultra thick,blue] (-.5,-.2) -- (-.5, .55)node[pos=0, tikzdot]{}node[pos=1, tikzdot]{};
\draw[thick, dotted] (-1,1) to (1,1);
\draw[thick, dotted] (-1,-1) to (1,-1);
}} \endxy
=
0 .
\]
 Note that the two relations are equivalent, due to the isotopy relation on red cups and caps.

\begin{lem}
The following are $(R,R)$-bimodule maps:
\begin{align}
\xy (0,1)*{
\tikzdiagc[scale=.5]{
\draw[ultra thick,blue] (-1,-1)-- (1, 1); 
\draw[very thick,Orange,densely dashed] (1,-1)-- (-1, 1);  
\draw[thick, dotted] (-1.25,1) to (1.25,1);\draw[thick, dotted] (-1.25,-1) to (1.25,-1);
}} \endxy
&\colon 1\otimes_s 1\otimes \undone \mapsto \undone\otimes 1\otimes_s 1 ,
\\
\xy (0,1)*{
\tikzdiagc[scale=.5,xscale=-1]{
\draw[ultra thick,blue] (-1,-1)-- (1, 1); 
\draw[very thick,Orange,densely dashed] (1,-1)-- (-1, 1);  
\draw[thick, dotted] (-1.25,1) to (1.25,1);\draw[thick, dotted] (-1.25,-1) to (1.25,-1);
}} \endxy
&\colon \undone \otimes 1 \otimes_s 1 \mapsto 1\otimes_s 1\otimes\undone ,
\\  \label{eq:splitB}
\xy (0,1)*{
\tikzdiagc[scale=.5]{
   \draw[ultra thick,blue] (0,0)-- (0, -1); \draw[ultra thick,blue] (-1,1) -- (0,0); \draw[ultra thick,blue] (1,1) -- (0,0);
 \draw[thick, dotted] (-1.25,1) to (1.25,1);\draw[thick, dotted] (-1.25,-1) to (1.25,-1);
}} \endxy
&\colon 1\otimes_s 1\mapsto 1\otimes_s 1\otimes_s 1 ,
\\ \label{eq:mergeB}
\xy (0,1)*{
\tikzdiagc[scale=.5]{
\draw[very thick,Orange,densely dashed] (0,0) to[out=160,in=-90] (-1, 1);  
   \draw[ultra thick,blue] (0,0)-- (0,1); \draw[ultra thick,blue] (-1,-1) -- (0,0); \draw[ultra thick,blue] (1,-1) -- (0,0);
 \draw[thick, dotted] (-1.25,1) to (1.25,1);\draw[thick, dotted] (-1.25,-1) to (1.25,-1);
}} \endxy
&\colon 1\otimes_s f\otimes_s 1 \mapsto \undone\otimes \partial(f) \otimes_s 1 .
\end{align}
\end{lem}

\emph{Proof} is straightforward. $\square$

These maps have degrees $0,0,-1,-1$, respectively. 

The shortcuts below will be useful in the sequel. 
\[
\xy (0,0)*{
\tikzdiagc[scale=.45]{
\draw[very thick,Orange,densely dashed] (.125,0) to (.125,1.5);  
\draw[ultra thick,blue] (-.75,-1) to[out=80,in=180] (.125,0) to[out=0,in=100] (1,-1);
\draw[thick, dotted] (-1.25,1.5) to (1.75,1.5);\draw[thick, dotted] (-1.25,-1) to (1.75,-1);
}} \endxy
:=
\xy (0,0)*{
\tikzdiagc[scale=.45]{
\draw[very thick,Orange,densely dashed] (0,0) to[out=150,in=-90]  (-.5,1.5);  
\draw[ultra thick,blue] (0,0)-- (0,.75)node[pos=1, tikzdot]{}; \draw[ultra thick,blue] (-1,-1) -- (0,0); \draw[ultra thick,blue] (1,-1) -- (0,0);
\draw[thick, dotted] (-1.25,1.5) to (1.75,1.5);\draw[thick, dotted] (-1.25,-1) to (1.75,-1);
}} \endxy
\mspace{65mu}
\xy (0,0)*{
\tikzdiagc[scale=.45,yscale=-1]{
\draw[very thick,Orange,densely dashed] (.125,0) to (.125,1.5);  
\draw[ultra thick,blue] (-.75,-1) to[out=80,in=180] (.125,0) to[out=0,in=100] (1,-1);
\draw[thick, dotted] (-1.25,1.5) to (1.75,1.5);\draw[thick, dotted] (-1.25,-1) to (1.75,-1);
}} \endxy
:=
\xy (0,0)*{
\tikzdiagc[scale=.45,yscale=-1]{
\draw[very thick,Orange,densely dashed] (0,1.5) to (0,.65);  
\draw[ultra thick,blue] (0,0)-- (0,.65)node[pos=1, tikzdot]{}; \draw[ultra thick,blue] (-1,-1) -- (0,0); \draw[ultra thick,blue] (1,-1) -- (0,0);
\draw[thick, dotted] (-1.25,1.5) to (1.75,1.5);\draw[thick, dotted] (-1.25,-1) to (1.75,-1);
}} \endxy
\mspace{65mu}
\xy (0,0)*{
\tikzdiagc[scale=.45]{
\draw[very thick,Orange,densely dashed] (0,0) to[out=30,in=-90]  (.75, 1.5);  
\draw[ultra thick,blue] (0,0)-- (0,1.5); \draw[ultra thick,blue] (-1,-1) -- (0,0); \draw[ultra thick,blue] (1,-1) -- (0,0);
\draw[thick, dotted] (-1.25,1.5) to (1.75,1.5);\draw[thick, dotted] (-1.25,-1) to (1.75,-1);
}} \endxy
:=
\xy (0,0)*{
\tikzdiagc[scale=.45]{
\draw[very thick,Orange,densely dashed] (0,0) to[out=135,in=-90] (-.5,.5)  to[out=90,in=-90] (.75, 1.5);  
\draw[ultra thick,blue] (0,0)-- (0,1.5); \draw[ultra thick,blue] (-1,-1) -- (0,0); \draw[ultra thick,blue] (1,-1) -- (0,0);
\draw[thick, dotted] (-1.25,1.5) to (1.75,1.5);\draw[thick, dotted] (-1.25,-1) to (1.75,-1);
}} \endxy
\]

\begin{lem}
The bimodule maps in equations~\eqref{eq:idBP} to~\eqref{eq:mergeB} satisfy the relations below.
\begingroup\allowdisplaybreaks
\begin{gather*}
\xy (0,1)*{
\tikzdiagc[scale=.5]{
\draw[ultra thick,blue] (-.75,-1.5) to[out=90,in=-90] (.6,0) to[out=90,in=-90] (-.75,1.5); 
\draw[very thick,Orange,densely dashed] (.75,-1.5) to[out=90,in=-90] (-.6,0) to[out=90,in=-90] (.75, 1.5); 
\draw[thick, dotted] (-1.25,1.5) to (1.25,1.5);\draw[thick, dotted] (-1.25,-1.5) to (1.25,-1.5);
}} \endxy
 = 
\xy (0,0)*{
\tikzdiagc[scale=.5]{
\draw[ultra thick,blue] (-.75,-1.5) to (-.75,1.5); 
\draw[very thick,Orange,densely dashed] (.75,-1.5) to (.75, 1.5); 
\draw[thick, dotted] (-1.25,1.5) to (1.25,1.5);\draw[thick, dotted] (-1.25,-1.5) to (1.25,-1.5);
}} \endxy
\mspace{80mu}
\xy (0,0)*{
\tikzdiagc[scale=.5,xscale=-1]{
\draw[ultra thick,blue] (-.75,-1.5) to[out=90,in=-90] (.6,0) to[out=90,in=-90] (-.75,1.5); 
\draw[very thick,Orange,densely dashed] (.75,-1.5) to[out=90,in=-90] (-.6,0) to[out=90,in=-90] (.75, 1.5); 
\draw[thick, dotted] (-1.25,1.5) to (1.25,1.5);\draw[thick, dotted] (-1.25,-1.5) to (1.25,-1.5);
}} \endxy
 = 
\xy (0,0)*{
\tikzdiagc[scale=.5,xscale=-1]{
\draw[ultra thick,blue] (-.75,-1.5) to (-.75,1.5); 
\draw[very thick,Orange,densely dashed] (.75,-1.5) to (.75, 1.5); 
\draw[thick, dotted] (-1.25,1.5) to (1.25,1.5);\draw[thick, dotted] (-1.25,-1.5) to (1.25,-1.5);
}} \endxy
\\[1ex] 
\xy (0,1)*{
\tikzdiagc[scale=.5]{
\draw[ultra thick,blue] (-1,-1)-- (.5, .5)node[pos=1, tikzdot]{}; 
\draw[very thick,Orange,densely dashed] (1,-1)-- (-1, 1);  
\draw[thick, dotted] (-1.25,1) to (1.25,1);\draw[thick, dotted] (-1.25,-1) to (1.25,-1);
}} \endxy
=
\xy (0,0)*{
\tikzdiagc[scale=.5]{
\draw[ultra thick,blue] (-1,-1)-- (-.5, -.5)node[pos=1, tikzdot]{}; 
\draw[very thick,Orange,densely dashed] (1,-1)-- (-1, 1);  
\draw[thick, dotted] (-1.25,1) to (1.25,1);\draw[thick, dotted] (-1.25,-1) to (1.25,-1);
}} \endxy
\mspace{80mu}
\xy (0,0)*{
\tikzdiagc[scale=.5,xscale=-1]{
\draw[ultra thick,blue] (-1,-1)-- (.5, .5)node[pos=1, tikzdot]{}; 
\draw[very thick,Orange,densely dashed] (1,-1)-- (-1, 1);  
\draw[thick, dotted] (-1.25,1) to (1.25,1);\draw[thick, dotted] (-1.25,-1) to (1.25,-1);
}} \endxy
=
\xy (0,0)*{
\tikzdiagc[scale=.5,xscale=-1]{
\draw[ultra thick,blue] (-1,-1)-- (-.5, -.5)node[pos=1, tikzdot]{}; 
\draw[very thick,Orange,densely dashed] (1,-1)-- (-1, 1);  
\draw[thick, dotted] (-1.25,1) to (1.25,1);\draw[thick, dotted] (-1.25,-1) to (1.25,-1);
}} \endxy
\\[1ex]
\xy (0,1)*{
\tikzdiagc[scale=.5]{
\draw[very thick,densely dashed,Orange] (-.5,0) to[out=-90,in=180] (-.1,-.75) to[out=0,in=-90] (.5,.25) to[out=90,in=-90] (-.5,1.5);
\draw[ultra thick,blue] (-.5,0) -- (-.5, .25)node[pos=0, tikzdot]{};
\draw[ultra thick,blue] (-.5,.25) to[out=90,in=-90] (.5,1.5);
\draw[thick, dotted] (-1,1.5) to (1,1.5);
\draw[thick, dotted] (-1,-1) to (1,-1);
}} \endxy
= -
\xy (0,0)*{
\tikzdiagc[scale=.5,xscale=-1]{
\draw[very thick,densely dashed,Orange] (-.5,0) to[out=-90,in=180] (-.1,-.75) to[out=0,in=-90] (.5,.25) to[out=90,in=-90] (.5,1.5);
\draw[ultra thick,blue] (-.5,0) -- (-.5, .25)node[pos=0, tikzdot]{};
\draw[ultra thick,blue] (-.5,.25) to[out=90,in=-90] (-.5,1.5);
\draw[thick, dotted] (-1,1.5) to (1,1.5);
\draw[thick, dotted] (-1,-1) to (1,-1);
}} \endxy
\\[1ex] 
\xy (0,1)*{
\tikzdiagc[scale=.5]{
    \draw[ultra thick,blue,] (0,-1)-- (0,0);
    \draw[ultra thick,blue] (0,0) -- (-.5,.5)node[pos=1, tikzdot]{};
    \draw[ultra thick,blue] (0,0) to[out=45,in=-90] (.5,1.25);
\draw[thick, dotted] (-1.25,1.25) to (1.25,1.25);\draw[thick, dotted] (-1.25,-1) to (1.25,-1);
}} \endxy
=
\xy (0,0)*{
\tikzdiagc[scale=.5]{
    \draw[ultra thick,blue] (0,-1) -- (0,1.25);
\draw[thick, dotted] (-1.25,1.25) to (1.25,1.25);\draw[thick, dotted] (-1.25,-1) to (1.25,-1);
}} \endxy
=
\xy (0,0)*{
\tikzdiagc[scale=.5,xscale=-1]{
    \draw[ultra thick,blue,] (0,-1)-- (0,0);
    \draw[ultra thick,blue] (0,0) -- (-.5,.5)node[pos=1, tikzdot]{};
    \draw[ultra thick,blue] (0,0) to[out=45,in=-90] (.5,1.25);
\draw[thick, dotted] (-1.25,1.25) to (1.25,1.25);\draw[thick, dotted] (-1.25,-1) to (1.25,-1);
}} \endxy
\\[1ex]
\xy (0,1)*{
\tikzdiagc[scale=.5]{
\draw[thick,Orange,densely dashed] (-.5,-.5) to[out=-135,in=0] (-1,-.75) to[out=180,in=180] (-.75,.5) to[out=0,in=135] (0,0); 
\draw[ultra thick,blue,] (0,1)-- (0,0);
\draw[ultra thick,blue] (0,0) -- (-.5,-.5)node[pos=1, tikzdot]{};
\draw[ultra thick,blue] (0,0) to[out=-45,in=90] (.5,-1.25);
\draw[thick, dotted] (-1.25,1) to (1.25,1);\draw[thick, dotted] (-1.25,-1.25) to (1.25,-1.25);
}} \endxy
=
\xy (0,0)*{
\tikzdiagc[scale=.5]{
    \draw[ultra thick,blue] (0,-1.25) -- (0,1);
\draw[thick, dotted] (-1.25,1) to (1.25,1);\draw[thick, dotted] (-1.25,-1.25) to (1.25,-1.25);
}} \endxy
=\ -\  
\xy (0,0)*{
\tikzdiagc[scale=.5,xscale=-1]{
\draw[thick,Orange,densely dashed] (-.5,-.5) to[out=-135,in=0] (-1,-.75) to[out=180,in=180] (-.75,.5) to[out=0,in=135] (0,0); 
\draw[ultra thick,blue,] (0,1)-- (0,0);
\draw[ultra thick,blue] (0,0) -- (-.5,-.5)node[pos=1, tikzdot]{};
\draw[ultra thick,blue] (0,0) to[out=-45,in=90] (.5,-1.25);
\draw[thick, dotted] (-1.25,1) to (1.25,1);\draw[thick, dotted] (-1.25,-1.25) to (1.25,-1.25);
}} \endxy
\mspace{70mu}
\xy (0,0)*{
\tikzdiagc[scale=.5]{
\draw[thick,Orange,densely dashed] (.5,-.5) to[out=-45,in=0] (.5,-1) to[out=180,in=180] (-.5,.4) to[out=0,in=135] (0,0); 
\draw[ultra thick,blue,] (0,1)-- (0,0);
\draw[ultra thick,blue] (0,0) -- (.5,-.5)node[pos=1, tikzdot]{};
\draw[ultra thick,blue] (0,0) to[out=-135,in=90] (-.5,-1.25);
\draw[thick, dotted] (-1.25,1) to (1.25,1);\draw[thick, dotted] (-1.25,-1.25) to (1.25,-1.25);
}} \endxy
=
-\ 
\xy (0,0)*{
\tikzdiagc[scale=.5]{
    \draw[ultra thick,blue] (0,-1.25) -- (0,1);
\draw[thick, dotted] (-1.25,1) to (1.25,1);\draw[thick, dotted] (-1.25,-1.25) to (1.25,-1.25);
}} \endxy
\\[1ex]
\xy (0,1)*{
\tikzdiagc[scale=.5]{
\draw[ultra thick,blue,] (0,-1)-- (0,0);
\draw[ultra thick,blue] (0,0) to[out=45,in=-120] (1,1.25);
\draw[ultra thick,blue] (0,0) -- (-.5,.5);
\draw[ultra thick,blue] (-.5,.5) to[out=45,in=-120] (0,1.25);
\draw[ultra thick,blue] (-.5,.5) to[out=135,in=-70] (-1,1.25);
\draw[thick, dotted] (-1.25,1.25) to (1.25,1.25);\draw[thick, dotted] (-1.25,-1) to (1.25,-1);
}} \endxy
= 
\xy (0,0)*{
\tikzdiagc[scale=.5,xscale=-1]{
\draw[ultra thick,blue,] (0,-1)-- (0,0);
\draw[ultra thick,blue] (0,0) to[out=45,in=-120] (1,1.25);
\draw[ultra thick,blue] (0,0) -- (-.5,.5);
\draw[ultra thick,blue] (-.5,.5) to[out=45,in=-120] (0,1.25);
\draw[ultra thick,blue] (-.5,.5) to[out=135,in=-70] (-1,1.25);
\draw[thick, dotted] (-1.25,1.25) to (1.25,1.25);\draw[thick, dotted] (-1.25,-1) to (1.25,-1);
}} \endxy
\mspace{50mu}
\xy (0,0)*{
\tikzdiagc[scale=.5,yscale=-1]{
\draw[thick,Orange,densely dashed] (0,0) to[out=-135,in=0] (-.6,-.35)  to[out=180,in=-135] (-.5,.5); 
\draw[ultra thick,blue,] (0,-1)-- (0,0);
\draw[ultra thick,blue] (0,0) to[out=45,in=-120] (1,1.25);
\draw[ultra thick,blue] (0,0) -- (-.5,.5);
\draw[ultra thick,blue] (-.5,.5) to[out=45,in=-120] (0,1.25);
\draw[ultra thick,blue] (-.5,.5) to[out=135,in=-70] (-1,1.25);
\draw[thick, dotted] (-1.25,1.25) to (1.25,1.25);\draw[thick, dotted] (-1.25,-1) to (1.25,-1);
}} \endxy
=
\xy (0,0)*{
\tikzdiagc[scale=.5,yscale=-1,xscale=-1]{
\draw[thick,Orange,densely dashed] (0,0) to[out=-135,in=0] (-.6,-.35)  to[out=180,in=-135] (-.5,.5); 
\draw[ultra thick,blue,] (0,-1)-- (0,0);
\draw[ultra thick,blue] (0,0) to[out=45,in=-120] (1,1.25);
\draw[ultra thick,blue] (0,0) -- (-.5,.5);
\draw[ultra thick,blue] (-.5,.5) to[out=45,in=-120] (0,1.25);
\draw[ultra thick,blue] (-.5,.5) to[out=135,in=-70] (-1,1.25);
\draw[thick, dotted] (-1.25,1.25) to (1.25,1.25);\draw[thick, dotted] (-1.25,-1) to (1.25,-1);
}} \endxy
\mspace{50mu}
\xy (0,0)*{
\tikzdiagc[scale=.5]{
\draw[very thick,Orange,densely dashed] (-1,.5) to[out=20,in=-90] (-.25,1.25);
\draw[ultra thick,blue,] (0,-1)-- (0,-.5);
\draw[ultra thick,blue,] (0,-.5)-- (-1,.5);
\draw[ultra thick,blue] (-1.5,-1) to[out=90,in=-135] (-1,.5);
\draw[ultra thick,blue] (0,-.5) to[out=45,in=-90] (.5,1.25);
\draw[ultra thick,blue,] (-1,.5) -- (-1,1.25);
\draw[thick, dotted] (-1.25,1.25) to (1.25,1.25);\draw[thick, dotted] (-1.25,-1) to (1.25,-1);
}} \endxy
=
\xy (0,0)*{
\tikzdiagc[scale=.5,xscale=-1]{
\draw[very thick,Orange,densely dashed] (-1,.5) to[out=20,in=-90] (-.25,1.25);
\draw[ultra thick,blue,] (0,-1)-- (0,-.5);
\draw[ultra thick,blue,] (0,-.5)-- (-1,.5);
\draw[ultra thick,blue] (-1.5,-1) to[out=90,in=-135] (-1,.5);
\draw[ultra thick,blue] (0,-.5) to[out=45,in=-90] (.5,1.25);
\draw[ultra thick,blue,] (-1,.5) -- (-1,1.25);
\draw[thick, dotted] (-1.25,1.25) to (1.25,1.25);\draw[thick, dotted] (-1.25,-1) to (1.25,-1);
}} \endxy
\end{gather*}
\endgroup

\end{lem}

A proof is given by a straighforward computation. $\square$

Dashed orange lines in the present paper are similar to dashed blue lines in Ellis-Lauda's categorification of odd quantum $sl(2)$, see~\cite{EL}. (Compare adjointness relations (3.13), (3.14) in that paper with the  adjointness in Proposition~\ref{prop_adj_minus}.)

The difference of the present diagrammatical calculus of blue lines (for $B$) and dashed red lines (for $\undR$) from the earlier calculus in Section~\ref{subsec-biadjointness} is that blue regions (for the category $R^s\dmod$) are now hidden inside blue lines and graphs. Thickening these lines and graphs recovers the earlier diagrammatics, see equation~\ref{eq_convert} below. 
\begin{equation}\label{eq_convert}
  \addtolength{\tabcolsep}{5.5pt}
  \begin{tabular}{ccccc}
\xy (0,1)*{
\tikzdiagc[scale=.45]{
  \draw[ultra thick,blue] (0,-1) -- (0, 0)node[pos=1, tikzdot]{};
\draw[thick, dotted] (-1,1) to (1,1);
\draw[thick, dotted] (-1,-1) to (1,-1);
}} \endxy
&
\xy (0,1)*{
\tikzdiagc[scale=.45]{
\draw[very thick, densely dashed,Orange] (0,-1) -- (0,0);
\draw[ultra thick,blue] (0,0) -- (0, 1)node[pos=0, tikzdot]{};
\draw[thick, dotted] (-1,1) to (1,1);
\draw[thick, dotted] (-1,-1) to (1,-1);
}} \endxy\ 
&
\xy (0,1)*{
\tikzdiagc[scale=.45]{
\draw[ultra thick,blue] (-1,-1)-- (1, 1); 
\draw[very thick,Orange,densely dashed] (1,-1)-- (-1, 1);  
\draw[thick, dotted] (-1.25,1) to (1.25,1);\draw[thick, dotted] (-1.25,-1) to (1.25,-1);
}} \endxy
&
\xy (0,1)*{
\tikzdiagc[scale=.45]{
   \draw[ultra thick,blue] (0,0)-- (0, -1); \draw[ultra thick,blue] (-1,1) -- (0,0); \draw[ultra thick,blue] (1,1) -- (0,0);
 \draw[thick, dotted] (-1.25,1) to (1.25,1);\draw[thick, dotted] (-1.25,-1) to (1.25,-1);
}} \endxy
&
\xy (0,1)*{
\tikzdiagc[scale=.45]{
\draw[very thick,Orange,densely dashed] (0,0) to[out=160,in=-90] (-1, 1);  
   \draw[ultra thick,blue] (0,0)-- (0,1); \draw[ultra thick,blue] (-1,-1) -- (0,0); \draw[ultra thick,blue] (1,-1) -- (0,0);
 \draw[thick, dotted] (-1.25,1) to (1.25,1);\draw[thick, dotted] (-1.25,-1) to (1.25,-1);
}} \endxy
\\[3ex]
  \xy (0,0)*{
\tikzdiagc[scale=.5]{
\draw[implies-,double equal sign distance]  (2,-.4) to (2,.4);
}} \endxy  
  &
  \xy (0,0)*{
\tikzdiagc[scale=.5]{
\draw[implies-,double equal sign distance]  (2,-.4) to (2,.4);
}} \endxy  
  &
     \xy (0,0)*{
\tikzdiagc[scale=.5]{
\draw[implies-,double equal sign distance]  (2,-.4) to (2,.4);
}} \endxy  
  &
  \xy (0,0)*{
\tikzdiagc[scale=.5]{
\draw[implies-,double equal sign distance]  (2,-.4) to (2,.4);
}} \endxy  
  &
  \xy (0,0)*{
\tikzdiagc[scale=.5]{
\draw[implies-,double equal sign distance]  (2,-.4) to (2,.4);
     }} \endxy
     \\[3ex]
\xy (0,0)*{
\tikzdiagc[scale=.6,xscale=-1,yscale=-1]{
\fill[blue!10] (-.5,1) to[out=-90,in=180] (0,-.25) to[out=0,in=-90] (.5,1)--cycle;
\draw[thick] (-.5,1) to[out=-90,in=180] (0,-.25) to[out=0,in=-90] (.5,1);
\draw[thick,to-] (.49,.3) -- (.49,.4);  
\draw[thick, dotted] (-1.25,1) to (1.25,1);\draw[thick, dotted] (-1.25,-1) to (1.25,-1);
  }} \endxy
  &
  \xy (0,0)*{
\tikzdiagc[scale=.6]{
\fill[blue!10] (-.5,1) to[out=-90,in=180] (0,-.25) to[out=0,in=-90] (.5,1)--cycle;
\draw[thick,Orange,densely dashed] (0,-.99) to (0,-.25); 
\draw[thick] (-.5,1) to[out=-90,in=180] (0,-.25) to[out=0,in=-90] (.5,1);
\draw[thick,to-] (.49,.3) -- (.49,.4);  
\draw[fill=black] (0,-.25) circle (1.5pt);
\draw[thick, dotted] (-1.25,1) to (1.25,1);\draw[thick, dotted] (-1.25,-1) to (1.25,-1);
}} \endxy  
    &
\xy (0,0)*{
\tikzdiagc[scale=.6]{
\fill[blue!10] (-1,-1) --(.25,1) -- (1,1) to (-.25,-1)--cycle;
\draw[thick,Orange,densely dashed] (-1,1)-- (1,-1);  
\draw[thick] (-1,-1) -- (.25,1);
\draw[thick,-to] (-.7,-.5) -- (-.675,-.45);  
\draw[thick] (-.25,-1) -- (1,1);
\draw[thick,-to] (.01,-.6) -- (-.02,-.65);  
\draw[thick, dotted] (-1.25,1) to (1.25,1);\draw[thick, dotted] (-1.25,-1) to (1.25,-1);
}} \endxy
    &
\xy (0,0)*{
\tikzdiagc[scale=.6,yscale=1]{
\fill[blue!10] (-.3,-1) to[out=90,in=-90] (-1,1) to (1,1) to[out=-90,in=90] (.3,-1)--cycle;
\fill[white]  (-.5,1) to[out=-85,in=180] (0,.25) to[out=0,in=-95] (.5,1)--cycle;
\draw[thick] (-.5,1) to[out=-85,in=180] (0,.25) to[out=0,in=-95] (.5,1);
\draw[thick,-to] (.08,.25) -- (.09,.25);  
\draw[thick] (-.3,-1) to[out=90,in=-90] (-1,1);
\draw[thick,-to] (-.64,0) -- (-.66,.05); 
\draw[thick] ( .3,-1) to[out=90,in=-90] (1,1);
\draw[thick,to-] (.58,-.15) -- (.70,.1); 
\draw[thick, dotted] (-1.25,1) to (1.25,1);\draw[thick, dotted] (-1.25,-1) to (1.25,-1);
}} \endxy
    & 
\xy (0,0)*{
\tikzdiagc[scale=.6,xscale=-1,yscale=-1]{
\fill[blue!10] (-.3,-1) to[out=90,in=-90] (-1,1) to (1,1) to[out=-90,in=90] (.3,-1)--cycle;
\fill[white]  (-.5,1) to[out=-85,in=180] (0,.25) to[out=0,in=-95] (.5,1)--cycle;
\draw[thick,Orange,densely dashed] (0,.25) to[out=-90,in=120] (1,-1); 
\draw[thick] (-.5,1) to[out=-85,in=180] (0,.25) to[out=0,in=-95] (.5,1);
\draw[thick,-to] (.45,.7) -- (.48,.8);  
\draw[thick] (-.3,-1) to[out=90,in=-90] (-1,1);
\draw[thick,-to] (-.64,0) -- (-.66,.05); 
\draw[thick] ( .3,-1) to[out=90,in=-90] (1,1);
\draw[thick,to-] (.58,-.15) -- (.70,.1); 
\draw[fill=black] (0,.25) circle (1.5pt);
\draw[thick, dotted] (-1.25,1) to (1.25,1);\draw[thick, dotted] (-1.25,-1) to (1.25,-1);
}} \endxy
\end{tabular}
\addtolength{\tabcolsep}{1pt}
\end{equation}

\begin{prop}\label{prop:decBB}
The following equality holds 

\begin{equation}\label{eq:idBBdec}
\xy (0,0)*{
\tikzdiagc[scale=.45]{
\draw[ultra thick,blue] (-.5, -1.5) -- (-.5,1);  
\draw[ultra thick,blue] (.75,-1.5) -- (.75,1);
\draw[thick, dotted] (-1.25,1) to (1.5,1);\draw[thick, dotted] (-1.25,-1.5) to (1.5,-1.5);
}} \endxy
=
\xy (0,0)*{
\tikzdiagc[scale=.45]{
\draw[ultra thick,blue] (0,0)-- (0, -1.5); \draw[ultra thick,blue] (-1,1) -- (0,0); \draw[ultra thick,blue] (1,1) -- (0,0);
\draw[ultra thick,blue] (1,-1.5) -- (1,-.5)node[pos=1, tikzdot]{};
\draw[thick, dotted] (-1.25,1) to (1.5,1);\draw[thick, dotted] (-1.25,-1.5) to (1.5,-1.5);
}} \endxy
\ - \ 
\xy (0,0)*{
\tikzdiagc[scale=.45]{
\draw[very thick,Orange,densely dashed] (0,0) to[out=30,in=-90]  (1, .85);  
\draw[ultra thick,blue] (0,0)-- (0,1.5); \draw[ultra thick,blue] (-1,-1) -- (0,0); \draw[ultra thick,blue] (1,-1) -- (0,0);
\draw[ultra thick,blue] (1,1.5) -- (1,.85)node[pos=1, tikzdot]{};
\draw[thick, dotted] (-1.25,1.5) to (1.75,1.5);\draw[thick, dotted] (-1.25,-1) to (1.75,-1);
}} \endxy
\end{equation}
Moreover, both terms on the right hand side are orthogonal idempotents.
\end{prop}

\begin{proof}
A direct computation shows that the two terms on the right hand side are orthogonal idempotents.
To show that their sum is the identity of $B\otimes_s B$ note that for $f\in R$ the element $P(f):=f-x_2\partial f$ is in $R^s$ since $\partial P(f)=0$. This allows writing $f=P(f)+x_2\partial(f)$ (a similar argument appears in \cite[\S 2.2]{EKh}). 
We then compute
\begin{align*}
\Biggl(\ 
\xy (0,0)*{
\tikzdiagc[scale=.4]{
\draw[ultra thick,blue] (0,0)-- (0, -1.5); \draw[ultra thick,blue] (-1,1) -- (0,0); \draw[ultra thick,blue] (1,1) -- (0,0);
\draw[ultra thick,blue] (1,-1.5) -- (1,-.5)node[pos=1, tikzdot]{};
\draw[thick, dotted] (-1.25,1) to (1.5,1);\draw[thick, dotted] (-1.25,-1.5) to (1.5,-1.5);
}} \endxy
\ - \ 
\xy (0,0)*{
\tikzdiagc[scale=.4]{
\draw[very thick,Orange,densely dashed] (0,0) to[out=30,in=-90]  (1, .85);  
\draw[ultra thick,blue] (0,0)-- (0,1.5); \draw[ultra thick,blue] (-1,-1) -- (0,0); \draw[ultra thick,blue] (1,-1) -- (0,0);
\draw[ultra thick,blue] (1,1.5) -- (1,.85)node[pos=1, tikzdot]{};
\draw[thick, dotted] (-1.25,1.5) to (1.75,1.5);\draw[thick, dotted] (-1.25,-1) to (1.75,-1);
}} \endxy\ 
\Biggr)&
(1\otimes_s f \otimes_s 1)
\\
  &= 1\otimes_s 1 \otimes_s f - 1\otimes_s 1 \otimes_s x_2\partial f + 1\otimes_s x_2\partial f \otimes_s f
\\
&= 1\otimes_s 1 \otimes_s P(f) + 1\otimes_s x_2\partial f \otimes_s 1
\\
&= 1\otimes_s P(f) \otimes_s 1 + 1\otimes_s x_2\partial f \otimes_s 1
= 1\otimes_s f \otimes_s 1 ,
\end{align*}
as claimed. 
\end{proof}

\begin{cor}\label{cor_dir_sum} 
There are direct sum decompositions
\begin{eqnarray}\label{eq_dirs1}
B\otimes B & \cong  & B\{-1\} \oplus \undB\{1\} \ \cong \ \undB\otimes \undB,   \\ \label{eq_dirs2}
B\otimes\undB & \cong & B\{1\}\oplus \undB\{-1\} \ \cong \ \undB\otimes B . 
\end{eqnarray}
\end{cor}

\begin{proof}
The first idempotent on the right hand side of (\ref{eq:idBBdec}) is a composition of degree $0$ maps 
\begin{equation*}
B\otimes B \xra{\xy (0,0)*{
\tikzdiagc[scale=.4]{
\draw[ultra thick,blue] (0,1)-- (0, -1); 
\draw[ultra thick,blue] (1,-1) -- (1,0)node[pos=1, tikzdot]{};
\draw[thick, dotted] (-.5,1) to (1.5,1);\draw[thick, dotted] (-.5,-1) to (1.5,-1);
}} \endxy} B\{-1\}\xra{
\xy (0,0)*{
\tikzdiagc[scale=.4]{
\draw[ultra thick,blue] (0,0)-- (0, -1); \draw[ultra thick,blue] (-1,1) -- (0,0); \draw[ultra thick,blue] (1,1) -- (0,0);
\draw[thick, dotted] (-1.25,1) to (1.5,1);\draw[thick, dotted] (-1.25,-1) to (1.5,-1);
}} \endxy
} B\otimes B.
\end{equation*}

Composing in the opposite direction gives the identity map of $B\{-1\}$, 
so that this idempotent is a projection onto a copy of $B\{-1\}$. Grading shift is present due to the degree of (\ref{eq:enddot}) being one. 

Likewise, the second idempotent is a composition 
\begin{equation*}
B\otimes B \xra{
  \xy (0,0)*{
\tikzdiagc[scale=.4]{
\draw[very thick,Orange,densely dashed] (0,0) to[out=20,in=-100]  (1,1);  
\draw[ultra thick,blue] (0,0)-- (0,1); \draw[ultra thick,blue] (-1,-1) -- (0,0); \draw[ultra thick,blue] (1,-1) -- (0,0);
\draw[thick, dotted] (-1.25,1) to (1.25,1);\draw[thick, dotted] (-1.25,-1) to (1.25,-1);
}} \endxy} \undB\{1\}\xra{
- \ 
\xy (0,0)*{
\tikzdiagc[scale=.4]{
\draw[very thick,Orange,densely dashed] (1,-1) to (1,.25);  
\draw[ultra thick,blue] (0,-1)-- (0,1);
\draw[ultra thick,blue] (1,.25) -- (1,1)node[pos=0, tikzdot]{};
\draw[thick, dotted] (-.75,1) to (1.5,1);\draw[thick, dotted] (-.75,-1) to (1.5,-1);
}} \endxy
} B\otimes B,
\end{equation*}
with the composition in the opposite direction equal $\id_{\undB\{1\}}$. Thus, it's a projection onto a graded bimodule isomorphic to $\undB\{1\}$. 
We obtain  a direct sum decomposition $B\otimes B\cong \undB\{-1\}\oplus B\{1\}$  in (\ref{eq_dirs1}) 
Tensoring with $\undR$ on the right and on the left 
gives the remaining direct sum decompositions.
\end{proof}

\begin{rem}
The identity in~\eqref{eq:idBBdec} can be expressed in equivalent ways, which result in different presentations of the maps realising the isomorphisms in~\eqref{eq_dirs1} and~\eqref{eq_dirs2}. 
For example, it equals its reflection around a vertical axis:
\begin{equation}\label{prop:decBB-alternative}
\xy (0,0)*{
\tikzdiagc[scale=.45]{
\draw[ultra thick,blue] (-.5, -1.5) -- (-.5,1);  
\draw[ultra thick,blue] (.75,-1.5) -- (.75,1);
\draw[thick, dotted] (-1.25,1) to (1.5,1);\draw[thick, dotted] (-1.25,-1.5) to (1.5,-1.5);
}} \endxy
=
\xy (0,0)*{
\tikzdiagc[scale=.45,xscale=-1]{
\draw[ultra thick,blue] (0,0)-- (0, -1.5); \draw[ultra thick,blue] (-1,1) -- (0,0); \draw[ultra thick,blue] (1,1) -- (0,0);
\draw[ultra thick,blue] (1,-1.5) -- (1,-.5)node[pos=1, tikzdot]{};
\draw[thick, dotted] (-1.25,1) to (1.5,1);\draw[thick, dotted] (-1.25,-1.5) to (1.5,-1.5);
}} \endxy
\ - \ 
\xy (0,0)*{
\tikzdiagc[scale=.45,xscale=-1]{
\draw[very thick,Orange,densely dashed] (0,0) to[out=30,in=-90]  (1, .85);  
\draw[ultra thick,blue] (0,0)-- (0,1.5); \draw[ultra thick,blue] (-1,-1) -- (0,0); \draw[ultra thick,blue] (1,-1) -- (0,0);
\draw[ultra thick,blue] (1,1.5) -- (1,.85)node[pos=1, tikzdot]{};
\draw[thick, dotted] (-1.25,1.5) to (1.75,1.5);\draw[thick, dotted] (-1.25,-1) to (1.75,-1);
}} \endxy
\end{equation}

To see that this equation holds, we need the two relations below. 
\begin{gather*}
\xy (0,1)*{
\tikzdiagc[scale=.45,xscale=1]{
\draw[very thick,Orange,densely dashed] (0,0) to[out=30,in=-90]  (1, .85);  
\draw[ultra thick,blue] (0,0)-- (0,1.5); \draw[ultra thick,blue] (-1,-1) -- (0,0); \draw[ultra thick,blue] (1,-1) -- (0,0);
\draw[ultra thick,blue] (1,1.5) -- (1,.85)node[pos=1, tikzdot]{};
\draw[thick, dotted] (-1.25,1.5) to (1.75,1.5);\draw[thick, dotted] (-1.25,-1) to (1.75,-1);
}} \endxy
\ + \ 
\xy (0,0)*{
\tikzdiagc[scale=.45,xscale=-1]{
\draw[very thick,Orange,densely dashed] (0,0) to[out=30,in=-90]  (1, .85);  
\draw[ultra thick,blue] (0,0)-- (0,1.5); \draw[ultra thick,blue] (-1,-1) -- (0,0); \draw[ultra thick,blue] (1,-1) -- (0,0);
\draw[ultra thick,blue] (1,1.5) -- (1,.85)node[pos=1, tikzdot]{};
\draw[thick, dotted] (-1.25,1.5) to (1.75,1.5);\draw[thick, dotted] (-1.25,-1) to (1.75,-1);
}} \endxy
\ =\ 
\xy (0,0)*{
\tikzdiagc[scale=.45]{
\draw[very thick,Orange,densely dashed] (.125,-.125) to (.125,.625);  
\draw[ultra thick,blue] (-.75,-1) to[out=80,in=180] (.125,-.125) to[out=0,in=100] (1,-1);
\draw[ultra thick,blue] (-.75,1.5) to[out=-80,in=180] (.125,.625) to[out=0,in=-100] (1,1.5);
\draw[thick, dotted] (-1.25,1.5) to (1.75,1.5);\draw[thick, dotted] (-1.25,-1) to (1.75,-1);
}} \endxy
\\[1ex]
\xy (0,1)*{
\tikzdiagc[scale=.45]{
\draw[ultra thick,blue] (-.5, -1.5) -- (-.5,-.25)node[pos=1, tikzdot]{};  
\draw[ultra thick,blue] (.75,-1.5) -- (.75,1);
\draw[thick, dotted] (-1.25,1) to (1.5,1);\draw[thick, dotted] (-1.25,-1.5) to (1.5,-1.5);
}} \endxy
=
\xy (0,0)*{
\tikzdiagc[scale=.45,xscale=-1]{
\draw[ultra thick,blue] (-.5, -1.5) -- (-.5,-.25)node[pos=1, tikzdot]{};  
\draw[ultra thick,blue] (.75,-1.5) -- (.75,1);
\draw[thick, dotted] (-1.25,1) to (1.5,1);\draw[thick, dotted] (-1.25,-1.5) to (1.5,-1.5);
}} \endxy
\ - \ 
\xy (0,0)*{
\tikzdiagc[scale=.45]{
\draw[very thick,Orange,densely dashed] (.125,0) to (.125,.85);  
\draw[ultra thick,blue] (-.75,-1) to[out=80,in=180] (.125,0) to[out=0,in=100] (1,-1);
\draw[ultra thick,blue] (.125,1.5) -- (.125,.85)node[pos=1, tikzdot]{};
\draw[thick, dotted] (-1.25,1.5) to (1.75,1.5);\draw[thick, dotted] (-1.25,-1) to (1.75,-1);
}} \endxy
\end{gather*}
The first one can be proved by direct computation, and the second is obtained from~\eqref{eq:idBBdec} by postcomposing all terms with the map $m$ at the appropriate place.
 Combining these two relations with \eqref{eq:idBBdec} gives~\eqref{prop:decBB-alternative}.
 \end{rem}

The direct sum decompositions in Corollary~\ref{cor_dir_sum} are not canonical. Specializing to $B\otimes B$, there is a canonical 
short exact sequence below  (up to a choice of signs for the maps) 
with the inclusion given by map (\ref{eq:splitB}) 
\begin{equation*}
  0 \lra B\{-1\} \xra{
\xy (0,0)*{
\tikzdiagc[scale=.4]{
\draw[ultra thick,blue] (0,0)-- (0, -1); \draw[ultra thick,blue] (-1,1) -- (0,0); \draw[ultra thick,blue] (1,1) -- (0,0);
\draw[thick, dotted] (-1.25,1) to (1.25,1);\draw[thick, dotted] (-1.25,-1) to (1.25,-1);
}} \endxy
  }
  B\otimes B
  \xra{
\xy (0,0)*{
  \tikzdiagc[scale=.4]{
\draw[very thick,Orange,densely dashed] (0,0) to[out=20,in=-100]  (1,1);  
\draw[ultra thick,blue] (0,0)-- (0,1); \draw[ultra thick,blue] (-1,-1) -- (0,0); \draw[ultra thick,blue] (1,-1) -- (0,0);
\draw[thick, dotted] (-1.25,1) to (1.25,1);\draw[thick, dotted] (-1.25,-1) to (1.25,-1);    
}} \endxy
  } \undB\{1\}  \lra 0 .
\end{equation*}

This sequence splits, but a splitting is  non-unique, due to the  existence of a non-trivial degree $2$ bimodule map $\undB\lra B$, see below on the left. Via  adjointness, it comes from a degree two homomorphism $R\lra B\otimes B$, shown below on the right 

\begin{equation*}
-\ 
\xy (0,0)*{
\tikzdiagc[scale=.6]{
\draw[very thick, densely dashed,Orange] (-.5,-.25) to[out=-90,in=180] (0,-.75) to[out=0,in=-90] (.5,-.25);
\draw[very thick, densely dashed,Orange] (1,.25) to[out=90,in=180] (1.5,.75) to[out=0,in=90] (2,-1.25) ; 
\draw[ultra thick,blue] (-.5,-.25) -- (-.5, 1)node[pos=0, tikzdot]{};
\draw[ultra thick,blue] ( .5,-.25) -- (.5,-.24)node[pos=0, tikzdot]{};
\draw[ultra thick,blue] ( .5,-.25) to[out=90,in=180] (1,.25) to[out=0,in=90] (1.5,-.25) to (1.5,-1.25);
\draw[blue,fill=blue] (1,.25) circle (2.5pt);
\draw[thick, dotted] (-1,1) to (2.5,1);
\draw[thick, dotted] (-1,-1.25) to (2.5,-1.25);
}} \endxy
\mspace{70mu}
\xy (0,0)*{
\tikzdiagc[scale=.6]{
\draw[very thick, densely dashed,Orange] (-.5,0) to[out=-90,in=180] (0,-.75) to[out=0,in=-90] (.5,0);
\draw[ultra thick,blue] (-.5,0) -- (-.5, 1)node[pos=0, tikzdot]{};
\draw[ultra thick,blue] ( .5,0) -- ( .5, 1)node[pos=0, tikzdot]{};
\draw[thick, dotted] (-1,1) to (1,1);
\draw[thick, dotted] (-1,-1.25) to (1,-1.25);
}} \endxy
\end{equation*}
(the minus sign is added to match our definition of the corresponding adjointness morphism). 
A particular direct sum decomposition of $B\otimes B$ is given by the following maps, as in the proof of Corollary~\ref{cor_dir_sum}.  
\begin{equation*}
\xy (0,0)*{
\tikzdiagc[scale=.6]{
  \node at (-5.5,0) {\small \text{$B\{-1\}$}};
\node at ( 0,0) {\small \text{$B\otimes B$}};
\node at ( 5.35,0) {\small \text{$\undB\{1\}$}};
\draw[->] (-4.5,.25) to[out=20,in=160] (-1,.25);
\draw[<-] (-4.5,-.25) to[out=-20,in=-160] (-1,-.25);
\draw[<-] ( 4.5,.25) to[out=160,in=20] ( 1,.25);
\draw[->] ( 4.5,-.25) to[out=-160,in=-20] ( 1,-.25);
\begin{scope}[shift={(-2.8,1.7)},scale=.75]
\draw[ultra thick,blue] (0,0)-- (0, -1); \draw[ultra thick,blue] (-1,1) -- (0,0); \draw[ultra thick,blue] (1,1) -- (0,0);
\draw[thick, dotted] (-1.25,1) to (1.25,1);\draw[thick, dotted] (-1.25,-1) to (1.25,-1);
\end{scope}
\begin{scope}[shift={(-3.2,-1.7)},scale=.75]
\draw[ultra thick,blue] (0,1)-- (0, -1); 
\draw[ultra thick,blue] (1,-1) -- (1,0)node[pos=1, tikzdot]{};
\draw[thick, dotted] (-.5,1) to (1.5,1);\draw[thick, dotted] (-.5,-1) to (1.5,-1);
\end{scope}
\begin{scope}[shift={(2.7,1.7)},scale=.75]
\draw[very thick,Orange,densely dashed] (0,0) to[out=20,in=-100]  (1,1);  
\draw[ultra thick,blue] (0,0)-- (0,1); \draw[ultra thick,blue] (-1,-1) -- (0,0); \draw[ultra thick,blue] (1,-1) -- (0,0);
\draw[thick, dotted] (-1.25,1) to (1.25,1);\draw[thick, dotted] (-1.25,-1) to (1.25,-1);
\end{scope}
\begin{scope}[shift={(2.6,-1.7)},scale=.75]
\node at (-1.2,0) {\text{$-$}};
\draw[very thick,Orange,densely dashed] (1,-1) to (1,.25);  
\draw[ultra thick,blue] (0,-1)-- (0,1);
\draw[ultra thick,blue] (1,.25) -- (1,1)node[pos=0, tikzdot]{};
\draw[thick, dotted] (-.5,1) to (1.5,1);\draw[thick, dotted] (-.5,-1) to (1.5,-1);
\end{scope}
}} \endxy  
\end{equation*}

%
%

\section{Oriented calculus for products of 
generating bimodules}

Our diagrammatics so far explicitly includes bimodules $B$ (blue lines) and $\undR$  (dashed orange lines). Bimodule $\undB$ and the maps that go through it appear implicitly through a combination of diagrammatics for $B$  and for $\undR$. It's natural to extend this diagrammatics, by  depicting the identity map of $B$, respectively $\undB$, by a vertical blue line oriented up, respectively down, see below. Then the biadjointness maps  (\ref{eq_biadjoint_pics}) can be compactly depicted by oriented cups and caps, with the usual isotopy relations on these cup and caps.

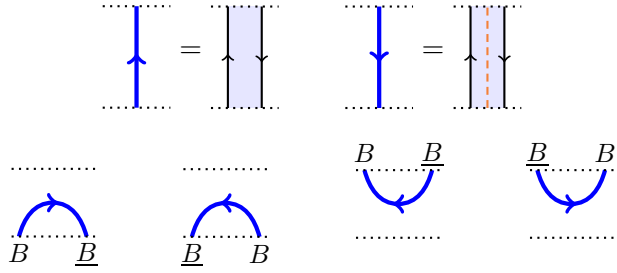
\begin{gather*}
\xy (0,0)*{
\tikzdiagc[scale=.45]{
\draw[ultra thick,blue,-to] (0,-1.5) to (0,.1); \draw[ultra thick,blue] (0,0) to (0,1.5); 
\draw[thick, dotted] (-1,1.5) to (1,1.5);\draw[thick, dotted] (-1,-1.5) to (1,-1.5);
}} \endxy  
=
\xy (0,-1)*{
\tikzdiagc[scale=.45]{
\fill[blue!10] (.5,-1.5) to (-.5,-1.5) to (-.5,1.5) to (.5,1.5)--cycle;
\draw[thick,-to] (-.5,-1.5) to (-.5,.1); \draw[thick] (-.5,0) to (-.5,1.5); 
\draw[thick] (.5,-1.5) to (.5,-.1); \draw[thick,to-] (.5,-.15) to (.5,1.5); 
\draw[thick, dotted] (-1,1.5) to (1,1.5);\draw[thick, dotted] (-1,-1.5) to (1,-1.5);
}} \endxy
\mspace{40mu}
\xy (0,-1)*{
\tikzdiagc[scale=.45]{
\draw[ultra thick,blue] (0,-1.5) to (0,0); \draw[ultra thick,blue,to-] (0,-.15) to (0,1.5); 
\draw[thick, dotted] (-1,1.5) to (1,1.5);\draw[thick, dotted] (-1,-1.5) to (1,-1.5);
}} \endxy  
=
\xy (0,-1)*{
\tikzdiagc[scale=.45]{
\fill[blue!10] (.5,-1.5) to (-.5,-1.5) to (-.5,1.5) to (.5,1.5)--cycle;
\draw[thick,-to] (-.5,-1.5) to (-.5,.1); \draw[thick] (-.5,0) to (-.5,1.5); 
\draw[thick,Orange,densely dashed] (0,-1.5)-- (0, 1.5);  
\draw[thick] (.5,-1.5) to (.5,-.1); \draw[thick,to-] (.5,-.15) to (.5,1.5); 
\draw[thick, dotted] (-1,1.5) to (1,1.5);\draw[thick, dotted] (-1,-1.5) to (1,-1.5);
}} \endxy
\\[1ex] 
\xy (0,-2.4)*{
\tikzdiagc[scale=.45,xscale=-1,yscale=-1]{
\draw[ultra thick,blue] (-1,1) to[out=-75,in=180] (0,0) to[out=0,in=-105] (1,1);
\draw[ultra thick,blue,to-] (-.15,0) -- (-.1,0);  
\draw[thick, dotted] (-1.25,1) to (1.25,1);\draw[thick, dotted] (-1.25,-1) to (1.25,-1);
\node at (1,1.5) {\small \text{$B$}};
\node at (-1,1.55) {\small \text{$\undB$}};
}} \endxy
\mspace{40mu}
\xy (0,-3.4)*{
\tikzdiagc[scale=.45,xscale=1,yscale=-1]{
\draw[ultra thick,blue] (-1,1) to[out=-75,in=180] (0,0) to[out=0,in=-105] (1,1);
\draw[ultra thick,blue,to-] (-.15,0) -- (-.1,0);  
\draw[thick, dotted] (-1.25,1) to (1.25,1);\draw[thick, dotted] (-1.25,-1) to (1.25,-1);
\node at (1,1.5) {\small \text{$B$}};
\node at (-1,1.55) {\small \text{$\undB$}};
}} \endxy
\mspace{40mu}
\xy (0,1.4)*{
\tikzdiagc[scale=.45,xscale=1,yscale=1]{
\draw[ultra thick,blue] (-1,1) to[out=-75,in=180] (0,0) to[out=0,in=-105] (1,1);
\draw[ultra thick,blue,to-] (-.15,0) -- (-.1,0);  
\draw[thick, dotted] (-1.25,1) to (1.25,1);\draw[thick, dotted] (-1.25,-1) to (1.25,-1);
\node at (1,1.45) {\small \text{$\undB$}};
\node at (-1,1.5) {\small \text{$B$}};
}} \endxy
\mspace{40mu}
\xy (0,1.4)*{
\tikzdiagc[scale=.45,xscale=-1,yscale=1]{
\draw[ultra thick,blue] (-1,1) to[out=-75,in=180] (0,0) to[out=0,in=-105] (1,1);
\draw[ultra thick,blue,to-] (-.15,0) -- (-.1,0);  
\draw[thick, dotted] (-1.25,1) to (1.25,1);\draw[thick, dotted] (-1.25,-1) to (1.25,-1);
\node at (1,1.45) {\small \text{$\undB$}};
\node at (-1,1.5) {\small \text{$B$}};
}} \endxy
\end{gather*}

A closed circle, either clockwise or counterclockwise oriented, evaluates to $0$, see below. 

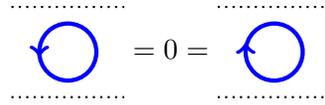
\begin{equation*}
\xy (0,0)*{
\tikzdiagc[scale=.6,xscale=1,yscale=1]{
\draw[ultra thick, blue] (0,0) circle (18pt); 
\draw[ultra thick,blue,to-] (-.63,-.1) -- (-.63,0);  
\draw[thick, dotted] (-1.25,1) to (1.25,1);\draw[thick, dotted] (-1.25,-1) to (1.25,-1);
}} \endxy
=
0
=
\xy (0,0)*{
\tikzdiagc[scale=.6,xscale=1,yscale=1]{
\draw[ultra thick, blue] (0,0) circle (18pt); 
\draw[ultra thick,blue,to-] (-.63,.2) -- (-.63,.1);  
\draw[thick, dotted] (-1.25,1) to (1.25,1);\draw[thick, dotted] (-1.25,-1) to (1.25,-1);
}} \endxy
\end{equation*}

There are additional generating maps and isotopy relations on them (some relations are depicted below, together with the maps).

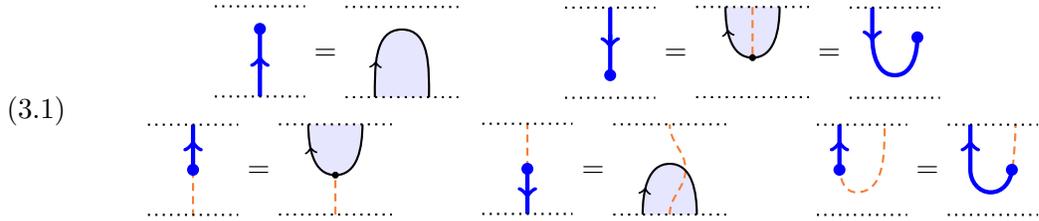
\begin{equation}\label{fig_five_diag}
\begin{split}
\xy (0,1)*{
\tikzdiagc[scale=.6]{
  \draw[ultra thick,blue] (0,-.4) -- (0, .5)node[pos=1, tikzdot]{};
  \draw[ultra thick,blue,-to] (0,-1) to (0,-.1); 
\draw[thick, dotted] (-1,1) to (1,1);\draw[thick, dotted] (-1,-1) to (1,-1);
}} \endxy  
=
\xy (0,0)*{
\tikzdiagc[scale=.6,xscale=1,yscale=1]{
\fill[blue!10]  (-.6,-1) to[out=90,in=180] (0,.5) to[out=0,in=90] (.6,-1)--cycle;
\draw[thick] (-.6,-1) to[out=90,in=180] (0,.5) to[out=0,in=90] (.6,-1);
\draw[thick,-to] (-.585,-.3) -- (-.580,-.2);  
\draw[thick, dotted] (-1.25,1) to (1.25,1);\draw[thick, dotted] (-1.25,-1) to (1.25,-1);
}} \endxy
\mspace{65mu}
\xy (0,0)*{
\tikzdiagc[scale=.6,yscale=-1]{
  \draw[ultra thick,blue] (0,-.4) -- (0, .5)node[pos=1, tikzdot]{};
  \draw[ultra thick,blue,-to] (0,-1) to (0,-.1); 
\draw[thick, dotted] (-1,1) to (1,1);\draw[thick, dotted] (-1,-1) to (1,-1);
}} \endxy  
=
\xy (0,0)*{
\tikzdiagc[scale=.6,xscale=1,yscale=-1]{
\fill[blue!10]  (-.6,-1) to[out=90,in=180] (0,.125) to[out=0,in=90] (.6,-1)--cycle;
\draw[thick,Orange,densely dashed] (0,-1) to (0,.125); 
\draw[thick] (-.6,-1) to[out=90,in=180] (0,.125) to[out=0,in=90] (.6,-1);
\draw[thick,to-] (-.551,-.5) -- (-.54,-.4);  
\draw[fill=black] (0,.125) circle (1.5pt);
\draw[thick, dotted] (-1.25,1) to (1.25,1);\draw[thick, dotted] (-1.25,-1) to (1.25,-1);
}} \endxy
=
\xy (0,0)*{
\tikzdiagc[scale=.6,yscale=-1]{
  \draw[ultra thick,blue] (0,-.35) to[out=90,in=180] (.5,.5) to[out=0,in=90] (1,-.35);
\draw[ultra thick, blue] (1,-.35)--(1,-.34)node[pos=1, tikzdot]{};
  \draw[ultra thick,blue,-to] (0,-1) to (0,-.25); 
\draw[thick, dotted] (-.5,1) to (1.5,1);\draw[thick, dotted] (-.5,-1) to (1.5,-1);
}} \endxy  
\mspace{61mu}
\\[1ex]
\xy (0,1)*{
\tikzdiagc[scale=.6]{
\draw[thick,Orange,densely dashed] (0,0) to (0,-1); 
  \draw[ultra thick,blue] (0,.6) -- (0,1);
  \draw[ultra thick,blue,-to] (0,0) -- (0,.65)node[pos=0, tikzdot]{}; 
\draw[thick, dotted] (-1,1) to (1,1);\draw[thick, dotted] (-1,-1) to (1,-1);
}} \endxy  
=
\xy (0,0)*{
\tikzdiagc[scale=.6,xscale=1,yscale=-1]{
\fill[blue!10]  (-.6,-1) to[out=90,in=180] (0,.125) to[out=0,in=90] (.6,-1)--cycle;
\draw[thick,Orange,densely dashed] (0,1) to (0,.125); 
\draw[thick] (-.6,-1) to[out=90,in=180] (0,.125) to[out=0,in=90] (.6,-1);
\draw[thick,to-] (-.551,-.5) -- (-.54,-.4);  
\draw[fill=black] (0,.125) circle (1.5pt);
\draw[thick, dotted] (-1.25,1) to (1.25,1);\draw[thick, dotted] (-1.25,-1) to (1.25,-1);
}} \endxy
\mspace{55mu}
\xy (0,0)*{
\tikzdiagc[scale=.6,yscale=-1]{
\draw[thick,Orange,densely dashed] (0,0) to (0,-1); 
  \draw[ultra thick,blue] (0,.6) -- (0,1);
  \draw[ultra thick,blue,-to] (0,0) -- (0,.65)node[pos=0, tikzdot]{}; 
\draw[thick, dotted] (-1,1) to (1,1);\draw[thick, dotted] (-1,-1) to (1,-1);
}} \endxy  
=
\xy (0,0)*{
\tikzdiagc[scale=.6,xscale=1,yscale=1]{
\fill[blue!10]  (-.6,-1) to[out=90,in=180] (0,.125) to[out=0,in=90] (.6,-1)--cycle;
\draw[thick,Orange,densely dashed] (0,-1) to[out=90,in=-90] (.35,0) to[out=90,in=-90] (0,1); 
\draw[thick] (-.6,-1) to[out=90,in=180] (0,.125) to[out=0,in=90] (.6,-1);
\draw[thick,-to] (-.551,-.5) -- (-.54,-.4);  
\draw[thick, dotted] (-1.25,1) to (1.25,1);\draw[thick, dotted] (-1.25,-1) to (1.25,-1);
}} \endxy
\mspace{55mu}
\xy (0,0)*{
\tikzdiagc[scale=.6]{
\draw[thick,Orange,densely dashed] (0,.1) to[out=-90,in=180] (.375,-.5) to[out=0,in=-90] (1,1); 
  \draw[ultra thick,blue] (0,.6) -- (0,1);
  \draw[ultra thick,blue,-to] (0,0) -- (0,.65)node[pos=0, tikzdot]{}; 
\draw[thick, dotted] (-.5,1) to (1.5,1);\draw[thick, dotted] (-.5,-1) to (1.5,-1);
}} \endxy  
=
\xy (0,0)*{
\tikzdiagc[scale=.6]{
\draw[thick,Orange,densely dashed] (.94,.05) to[out=85,in=-90] (1,1);   
  \draw[ultra thick,blue] (0,.6) -- (0,1);
  \draw[ultra thick,blue,-to] (0,0) -- (0,.65);
\draw[ultra thick,blue] (0,0) to[out=-90,in=180] (.45,-.5) to[out=0,in=-95] (.92,0); 
\draw[ultra thick,blue] (.92,-.05)--(.92,0)node[pos=1, tikzdot]{}; 
\draw[thick, dotted] (-.5,1) to (1.5,1);\draw[thick, dotted] (-.5,-1) to (1.5,-1);
}} \endxy  
\end{split}
\end{equation}

Isomorphisms (\ref{eq_blue_red}) can be depicted by a trivalent vertex where a dashed orange line enters the point of orientation reversal of a blue line, see below, together with the corresponding relations. There are 8 such trivalent vertices, with some relations on them also shown below and other relations are obtained by suitable symmetries (horizontal and vertical reflection and orientation reversal). 

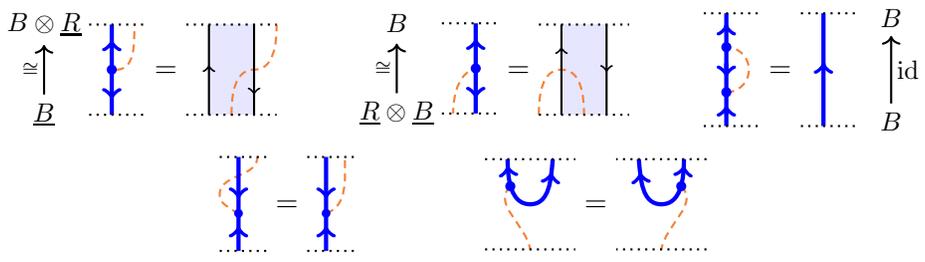
\begin{equation*}
\xy (0,0)*{
\tikzdiagc[scale=.6]{
\draw[thick,Orange,densely dashed] (.1,0) to[out=0,in=-90] (.5,1); 
\draw[ultra thick,blue] (0,.6) -- (0,1);
\draw[ultra thick,blue,-to] (0,0) -- (0,.65); 
\draw[blue,fill=blue] (0,0) circle (2.5pt);
\draw[ultra thick,blue] (0,-.6) -- (0,-1);
\draw[ultra thick,blue,-to] (0,0) -- (0,-.65); 
\draw[thick, dotted] (-.5,1) to (.75,1);\draw[thick, dotted] (-.5,-1) to (.75,-1);
\node at (-1.5,-1) {\small \text{$\undB$}};\node at (-1.5, 1) {\small \text{$B\otimes\undR$}};
\draw[->] (-1.5,-.55) to (-1.5,.55);\node at (-1.8,0) {\tiny \text{$\cong$}};
}} \endxy  
=
\xy (0,0)*{
\tikzdiagc[scale=.6]{
\fill[blue!10] (.5,-1) to (-.5,-1) to (-.5,1) to (.5,1)--cycle;
\draw[thick,Orange,densely dashed] (0,-1) to[out=90,in=180] (.5,0) to[out=0,in=-90] (1,1); 
\draw[thick,-to] (-.5,-1) to (-.5,.1); \draw[thick] (-.5,0) to (-.5,1); 
\draw[thick] (.5,-1) to (.5,-.5); \draw[thick,to-] (.5,-.55) to (.5,1); 
\draw[thick, dotted] (-1,1) to (1,1);\draw[thick, dotted] (-1,-1) to (1,-1);
}} \endxy
\mspace{45mu}
\xy (0,0)*{
\tikzdiagc[scale=.6,yscale=-1,xscale=1]{
\draw[thick,Orange,densely dashed] (.1,0) to[out=180,in=-90] (-.5,1); 
\draw[ultra thick,blue] (0,.6) -- (0,1);
\draw[ultra thick,blue,-to] (0,0) -- (0,.65); 
\draw[blue,fill=blue] (0,0) circle (2.5pt);
\draw[ultra thick,blue] (0,-.6) -- (0,-1);
\draw[ultra thick,blue,-to] (0,0) -- (0,-.65); 
\draw[thick, dotted] (-.75,1) to (.5,1);\draw[thick, dotted] (-.75,-1) to (.5,-1);
\node at (-1.75,-1) {\small \text{$B$}};\node at (-1.75, 1) {\small \text{$\undR\otimes\undB$}};
\draw[<-] (-1.75,-.55) to (-1.75,.55);\node at (-2.05,0) {\tiny \text{$\cong$}};
}} \endxy
=
\xy (0,0)*{
\tikzdiagc[scale=.6]{
\fill[blue!10] (.5,-1) to (-.5,-1) to (-.5,1) to (.5,1)--cycle;
\draw[thick,Orange,densely dashed] (-1,-1) to[out=90,in=180] (-.5,0) to[out=0,in=90] (0,-1); 
\draw[thick,-to] (-.5,-1) to (-.5,.5); \draw[thick] (-.5,.4) to (-.5,1); 
\draw[thick] (.5,-1) to (.5,0); \draw[thick,to-] (.5,-.05) to (.5,1); 
\draw[thick, dotted] (-1,1) to (1,1);
\draw[thick, dotted] (-1,-1) to (1,-1);
}} \endxy
\mspace{45mu}
\xy (0,0)*{
\tikzdiagc[scale=.6]{
\draw[thick,Orange,densely dashed] (0,-.5) to[out=0,in=-90] (.5,0) to[out=90,in=0] (0,.5); 
\draw[ultra thick,blue] (0,-1.25) -- (0,1.25);
\draw[blue,fill=blue] (0,-.5) circle (2.5pt);
\draw[blue,fill=blue] (0,.5) circle (2.5pt);
\draw[ultra thick,blue,-to] (0,.99) -- (0,1);
\draw[ultra thick,blue,-to] (0,.14) -- (0,-.15);
\draw[ultra thick,blue,-to] (0,-.99) -- (0,-.8);
\draw[thick, dotted] (-.5,1.25) to (.75,1.25);\draw[thick, dotted] (-.5,-1.25) to (.75,-1.25);
}} \endxy
=
\xy (0,0)*{
\tikzdiagc[scale=.6]{
\draw[ultra thick,blue] (0,-1.25) -- (0,1.25);
\draw[ultra thick,blue,-to] (0,.14) -- (0,.15);
\draw[thick, dotted] (-.5,1.25) to (.75,1.25);\draw[thick, dotted] (-.5,-1.25) to (.75,-1.25);
\node at (1.5,-1.15) {\small \text{$B$}};\node at (1.5, 1.15) {\small \text{$B$}};
\draw[->] (1.5,-.75) to (1.5,.75);\node at (1.875,0) {\small \text{$\id$}};
}} \endxy
\end{equation*}

\begin{equation*}
\xy (0,0)*{
\tikzdiagc[scale=.5]{
\draw[thick,Orange,densely dashed] (0,-.285) to[out=135,in=-90] (-.5,.25) to[out=90,in=-90] (.5,1.25); 
\draw[ultra thick,blue] (0,-1.25) -- (0,1.25);
\draw[blue,fill=blue] (0,-.25) circle (2.5pt);
\draw[ultra thick,blue,-to] (0,.2) -- (0,.15);
\draw[ultra thick,blue,-to] (0,-.79) -- (0,-.6);
\draw[thick, dotted] (-.5,1.25) to (.75,1.25);\draw[thick, dotted] (-.5,-1.25) to (.75,-1.25);
}} \endxy
=
\xy (0,0)*{
\tikzdiagc[scale=.5]{
\draw[thick,Orange,densely dashed] (0,-.25) to[out=20,in=-90] (.5,1.25); 
\draw[ultra thick,blue] (0,-1.25) -- (0,1.25);
\draw[blue,fill=blue] (0,-.25) circle (2.5pt);
\draw[ultra thick,blue,-to] (0,.2) -- (0,.15);
\draw[ultra thick,blue,-to] (0,-.79) -- (0,-.6);
\draw[thick, dotted] (-.5,1.25) to (.75,1.25);\draw[thick, dotted] (-.5,-1.25) to (.75,-1.25);
}} \endxy
\mspace{80mu}
\xy (0,0)*{
\tikzdiagc[scale=.6]{
\draw[thick,Orange,densely dashed] (-.44,.4) to[out=225,in=100] (0,-1);   
\draw[ultra thick,blue] (-.5,1) to[out=-90,in=180] (0,0) to[out=0,in=-95] (.5,1); 
\draw[blue,fill=blue] (-.44,.4) circle (2.5pt);
\draw[ultra thick,blue,-to] (.46,.6) -- (.47,.7);
\draw[ultra thick,blue,-to] (-.49,.75) -- (-.50,.85);
\draw[thick, dotted] (-1,1) to (1,1);\draw[thick, dotted] (-1,-1) to (1,-1);
}} \endxy  
=
\xy (0,0)*{
\tikzdiagc[scale=.6,xscale=-1]{
\draw[thick,Orange,densely dashed] (-.44,.4) to[out=225,in=100] (0,-1);   
\draw[ultra thick,blue] (-.5,1) to[out=-90,in=180] (0,0) to[out=0,in=-95] (.5,1); 
\draw[blue,fill=blue] (-.44,.4) circle (2.5pt);
\draw[ultra thick,blue,-to] (.46,.6) -- (.47,.7);
\draw[ultra thick,blue,-to] (-.49,.75) -- (-.50,.85);
\draw[thick, dotted] (-1,1) to (1,1);\draw[thick, dotted] (-1,-1) to (1,-1);
}} \endxy 
\end{equation*}

Blue lines in the top row of (\ref{eq_convert}) now acquire upward orientation, see below. The rightmost diagram is an exception; dashed orange line is hidden at the cost of orienting the left bottom leg down. 
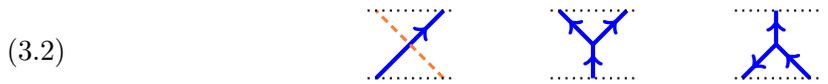
\begin{equation}\label{eq_three_small}
\xy (0,1)*{
\tikzdiagc[scale=.45]{
\draw[ultra thick,blue] (-1,-1)-- (1, 1); 
\draw[ultra thick,blue,-to] (.4,.4) -- (.5,.5);
\draw[very thick,Orange,densely dashed] (1,-1)-- (-1, 1);  
\draw[thick, dotted] (-1.25,1) to (1.25,1);\draw[thick, dotted] (-1.25,-1) to (1.25,-1);
}} \endxy
\mspace{60mu}
\xy (0,1)*{
\tikzdiagc[scale=.45]{
\draw[ultra thick,blue] (0,0)-- (0, -1); \draw[ultra thick,blue] (-1,1) -- (0,0); \draw[ultra thick,blue] (1,1) -- (0,0);
\draw[ultra thick,blue,-to] (0,-.4) -- (0,-.3);
\draw[ultra thick,blue,to-] (.7,.7) -- (.6,.6);
\draw[ultra thick,blue,to-] (-.7,.7) -- (-.6,.6);
\draw[thick, dotted] (-1.25,1) to (1.25,1);\draw[thick, dotted] (-1.25,-1) to (1.25,-1);
}} \endxy
\mspace{60mu}
\xy (0,1)*{
\tikzdiagc[scale=.45,xscale=-1]{
\draw[ultra thick,blue] (0,0)-- (0,1); \draw[ultra thick,blue] (-1,-1) -- (0,0); \draw[ultra thick,blue] (1,-1) -- (0,0);
\draw[ultra thick,blue,-to] (0,.6) -- (0,.7);
\draw[ultra thick,blue,-to] (-.4,-.4) -- (-.3,-.3);
\draw[ultra thick,blue,to-] (.7,-.7) -- (.6,-.6);
\draw[thick, dotted] (-1.25,1) to (1.25,1);\draw[thick, dotted] (-1.25,-1) to (1.25,-1);
}} \endxy
\end{equation}

 Composing trivalent vertices with cups and caps results in rotated trivalent vertices, see below (where top left diagram is the rightmost diagram in (\ref{eq_three_small}).

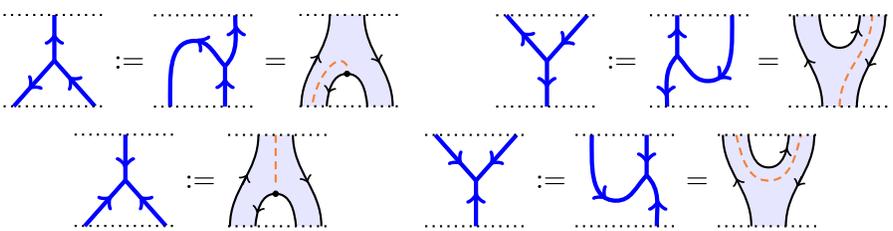
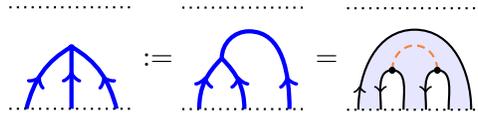
\begingroup\allowdisplaybreaks
\begin{gather*}
  \xy (0,1)*{
\tikzdiagc[scale=.54]{
\draw[ultra thick,blue] (0,0)-- (0,1.125);
\draw[ultra thick,blue] (-1,-1.125) -- (0,0); \draw[ultra thick,blue] (1,-1.125) -- (0,0);
\draw[ultra thick,blue,-to] (0,.54) -- (0,.66);
\draw[ultra thick,blue,to-] (-.63,-.7) -- (-.53,-.6);
\draw[ultra thick,blue,to-] (.35,-.4) -- (.45,-.5);
\draw[thick, dotted] (-1.25,1.125) to (1.25,1.125);\draw[thick, dotted] (-1.25,-1.125) to (1.25,-1.125);
}} \endxy
:=
\xy (0,0)*{
\tikzdiagc[scale=.54]{
\draw[ultra thick,blue] (0,0)-- (0, -1);
\draw[ultra thick,blue] (0,0) to[out=135,in=0] (-.75,.62) to[out=180,in=90] (-1.35,-1);
\draw[ultra thick,blue] (0,0) to[out=45,in=-90] (.25,1.25);
\draw[ultra thick,blue,-to] (0,-.4) -- (0,-.3);
\draw[ultra thick,blue,to-] (.275,1.0) -- (.275,.9);
\draw[ultra thick,blue,to-] (-.7,.66) -- (-.6,.6);
\draw[thick, dotted] (-1.75,1.25) to (.75,1.25);
\draw[thick, dotted] (-1.75,-1) to (.75,-1);
}} \endxy
=
\xy (0,0)*{
\tikzdiagc[scale=.54,xscale=-1,yscale=-1]{
\fill[blue!10] (-.425,-1.125) to[out=90,in=-90] (-1.125,1.125) to (1.125,1.125) to[out=-90,in=90] (.425,-1.125)--cycle;
\fill[white]  (-.5,1.125) to[out=-85,in=180] (0,.325) to[out=0,in=-95] (.5,1.125)--cycle;
\draw[thick,Orange,densely dashed] (0,.125) to[out=-60,in=-90] (.85,1.135); 
\draw[thick] (-.5,1.125) to[out=-85,in=180] (0,.325) to[out=0,in=-95] (.5,1.125);
\draw[thick,-to] (.43,.7) -- (.46,.8);  
\draw[thick] (-.425,-1.125) to[out=90,in=-90] (-1.125,1.125);
\draw[thick,-to] (-.77,0) -- (-.79,.05); 
\draw[thick] ( .425,-1.125) to[out=90,in=-90] (1.125,1.125);
\draw[thick,to-] (.7,-.15) -- (.82,.1); 
\draw[fill=black] (0,.325) circle (1.5pt);
\draw[thick, dotted] (-1.25,1.125) to (1.25,1.125);\draw[thick, dotted] (-1.25,-1.125) to (1.25,-1.125);
}} \endxy
\mspace{60mu}
\xy (0,0)*{
\tikzdiagc[scale=.54,yscale=-1]{
\draw[ultra thick,blue] (0,0)-- (0,1.125);
\draw[ultra thick,blue] (-1,-1.125) -- (0,0); \draw[ultra thick,blue] (1,-1.125) -- (0,0);
\draw[ultra thick,blue,-to] (0,.54) -- (0,.66);
\draw[ultra thick,blue,to-] (-.63,-.7) -- (-.53,-.6);
\draw[ultra thick,blue,to-] (.35,-.4) -- (.45,-.5);
\draw[thick, dotted] (-1.25,1.125) to (1.25,1.125);\draw[thick, dotted] (-1.25,-1.125) to (1.25,-1.125);
}} \endxy
:=
\xy (0,0)*{
\tikzdiagc[scale=.54,xscale=-1,yscale=-1]{
\draw[ultra thick,blue] (0,0)-- (0, -1);
\draw[ultra thick,blue] (0,0) to[out=135,in=0] (-.75,.62) to[out=180,in=90] (-1.35,-1);
\draw[ultra thick,blue] (0,0) to[out=45,in=-90] (.25,1.25);
\draw[ultra thick,blue,-to] (0,-.6) -- (0,-.7);
\draw[ultra thick,blue,to-] (.275,1.0) -- (.275,.9);
\draw[ultra thick,blue,-to] (-1.1,.48) -- (-1.0,.58);
\draw[thick, dotted] (-1.75,1.25) to (.75,1.25);
\draw[thick, dotted] (-1.75,-1) to (.75,-1);
}} \endxy
=
\xy (0,0)*{
\tikzdiagc[scale=.54,xscale=1,yscale=1]{
\fill[blue!10] (-.425,-1.125) to[out=90,in=-90] (-1.125,1.125) to (1.125,1.125) to[out=-90,in=90] (.425,-1.125)--cycle;
\fill[white]  (-.5,1.125) to[out=-85,in=180] (0,.325) to[out=0,in=-95] (.5,1.125)--cycle;
\draw[thick,Orange,densely dashed] (0,-1.125) to[out=90,in=-90] (.8,1.125); 
\draw[thick] (-.5,1.125) to[out=-85,in=180] (0,.325) to[out=0,in=-95] (.5,1.125);
\draw[thick,-to] (.43,.7) -- (.46,.8);  
\draw[thick] (-.425,-1.125) to[out=90,in=-90] (-1.125,1.125);
\draw[thick,-to] (-.77,0) -- (-.79,.05); 
\draw[thick] ( .425,-1.125) to[out=90,in=-90] (1.125,1.125);
\draw[thick,to-] (.7,-.15) -- (.82,.1); 
\draw[thick, dotted] (-1.25,1.125) to (1.25,1.125);\draw[thick, dotted] (-1.25,-1.125) to (1.25,-1.125);
}} \endxy
\\[1ex]
\xy (0,1)*{
\tikzdiagc[scale=.54]{
\draw[ultra thick,blue] (0,0)-- (0,1.125);
\draw[ultra thick,blue] (-1,-1.125) -- (0,0); \draw[ultra thick,blue] (1,-1.125) -- (0,0);
\draw[ultra thick,blue,to-] (0,.30) -- (0,.32);
\draw[ultra thick,blue,to-] (-.35,-.4) -- (-.45,-.5);
\draw[ultra thick,blue,to-] (.35,-.4) -- (.45,-.5);
\draw[thick, dotted] (-1.25,1.125) to (1.25,1.125);\draw[thick, dotted] (-1.25,-1.125) to (1.25,-1.125);
}} \endxy
:=
\xy (0,0)*{
\tikzdiagc[scale=.54,xscale=-1,yscale=-1]{
\fill[blue!10] (-.425,-1.125) to[out=90,in=-90] (-1.125,1.125) to (1.125,1.125) to[out=-90,in=90] (.425,-1.125)--cycle;
\fill[white]  (-.5,1.125) to[out=-85,in=180] (0,.325) to[out=0,in=-95] (.5,1.125)--cycle;
\draw[thick,Orange,densely dashed]  (0,-1.125) to (0,.125); 
\draw[thick] (-.5,1.125) to[out=-85,in=180] (0,.325) to[out=0,in=-95] (.5,1.125);
\draw[thick,-to] (.43,.7) -- (.46,.8);  
\draw[thick] (-.425,-1.125) to[out=90,in=-90] (-1.125,1.125);
\draw[thick,-to] (-.77,0) -- (-.79,.05); 
\draw[thick] ( .425,-1.125) to[out=90,in=-90] (1.125,1.125);
\draw[thick,to-] (.7,-.15) -- (.82,.1); 
\draw[fill=black] (0,.325) circle (1.5pt);
\draw[thick, dotted] (-1.25,1.125) to (1.25,1.125);\draw[thick, dotted] (-1.25,-1.125) to (1.25,-1.125);
}} \endxy
\mspace{60mu}
\xy (0,0)*{
\tikzdiagc[scale=.54,yscale=-1]{
\draw[ultra thick,blue] (0,0)-- (0,1.125);
\draw[ultra thick,blue] (-1,-1.125) -- (0,0); \draw[ultra thick,blue] (1,-1.125) -- (0,0);
\draw[ultra thick,blue,to-] (0,.30) -- (0,.32);
\draw[ultra thick,blue,to-] (-.35,-.4) -- (-.45,-.5);
\draw[ultra thick,blue,to-] (.35,-.4) -- (.45,-.5);
\draw[thick, dotted] (-1.25,1.125) to (1.25,1.125);\draw[thick, dotted] (-1.25,-1.125) to (1.25,-1.125);
}} \endxy
:=
\xy (0,0)*{
\tikzdiagc[scale=.54,xscale=1,yscale=-1]{
\draw[ultra thick,blue] (0,0)-- (0, -1);
\draw[ultra thick,blue] (0,0) to[out=135,in=0] (-.75,.62) to[out=180,in=90] (-1.35,-1);
\draw[ultra thick,blue] (0,0) to[out=45,in=-90] (.25,1.25);
\draw[ultra thick,blue,to-] (0,-.3) -- (0,-.4);
\draw[ultra thick,blue,-to] (.273,.7) -- (.273,.6);
\draw[ultra thick,blue,-to] (-1.1,.48) -- (-1.0,.58);
\draw[thick, dotted] (-1.75,1.25) to (.75,1.25);
\draw[thick, dotted] (-1.75,-1) to (.75,-1);
}} \endxy
=
\xy (0,0)*{
\tikzdiagc[scale=.54,xscale=1,yscale=1]{
\fill[blue!10] (-.425,-1.125) to[out=90,in=-90] (-1.125,1.125) to (1.125,1.125) to[out=-90,in=90] (.425,-1.125)--cycle;
\fill[white]  (-.5,1.125) to[out=-85,in=180] (0,.325) to[out=0,in=-95] (.5,1.125)--cycle;
\draw[thick,Orange,densely dashed] (-.8,1.125) to[out=-75,in=180] (0,0) to[out=0,in=-105] (.8,1.125); 
\draw[thick] (-.5,1.125) to[out=-85,in=180] (0,.325) to[out=0,in=-95] (.5,1.125);
\draw[thick,-to] (.43,.7) -- (.46,.8);  
\draw[thick] (-.425,-1.125) to[out=90,in=-90] (-1.125,1.125);
\draw[thick,-to] (-.77,0) -- (-.79,.05); 
\draw[thick] ( .425,-1.125) to[out=90,in=-90] (1.125,1.125);
\draw[thick,to-] (.7,-.15) -- (.82,.1); 
\draw[thick, dotted] (-1.25,1.125) to (1.25,1.125);\draw[thick, dotted] (-1.25,-1.125) to (1.25,-1.125);
}} \endxy
\\[1ex]
\xy (0,1)*{
\tikzdiagc[scale=.6]{
\draw[ultra thick,blue] (0,-1.125)-- (0,.29);
\draw[ultra thick,blue] (-1,-1.125) to[out=75,in=-150] (0,.25);
\draw[ultra thick,blue] ( 1,-1.125) to[out=105,in=-30] (0,.25);
\draw[blue,fill=blue] (0,.23) circle (1.5pt);
\draw[ultra thick,blue,to-] (0,-.30) -- (0,-.32);
\draw[ultra thick,blue,to-] (-.70,-.4) -- (-.75,-.5);
\draw[ultra thick,blue,to-] ( .70,-.4) -- ( .75,-.5);
\draw[thick, dotted] (-1.375,1.125) to (1.375,1.125);\draw[thick, dotted] (-1.375,-1.125) to (1.375,-1.125);
}} \endxy
:=
\xy (0,0)*{
\tikzdiagc[scale=.6,yscale=1]{
\draw[ultra thick,blue] (-.5,-1.125) to[out=90,in=-135] (0,0);
\draw[ultra thick,blue] (.5,-1.125) to[out=90,in=-45] (0,0);
\draw[ultra thick,blue] (0,0) to[out=90,in=180] (.6,.6) to[out=0,in=90] (1.5,-1.125);
\draw[ultra thick,blue,to-] (1.44,-.4) -- (1.46,-.5);
\draw[ultra thick,blue,to-] (-.36,-.4) -- (-.38,-.5);
\draw[ultra thick,blue,to-] ( .36,-.4) -- ( .38,-.5);
\draw[thick, dotted] (-.875,1.125) to (1.875,1.125);\draw[thick, dotted] (-.875,-1.125) to (1.875,-1.125);
}} \endxy
=
\xy (0,0)*{
\tikzdiagc[scale=.6,xscale=1,yscale=1]{
\fill[blue!10] (-1.25,-1.125) to[out=90,in=180] (0,.65) to[out=0,in=90] (1.25,-1.125)--cycle;
\fill[white]  ( .75,-1.125) to[out=90,in=0] (.5,-.25) to[out=180,in=90] (.25,-1.125)--cycle;
\fill[white] (-.75,-1.125) to[out=90,in=180] (-.5,-.25) to[out=0,in=90] (-.25,-1.125) --cycle;
\draw[thick,Orange,densely dashed] (-.5,-.25) to[out=90,in=180] (0,.3) to[out=0,in=90] (.5,-.25); 
\draw[thick] (-1.25,-1.125) to[out=90,in=180] (0,.65) to[out=0,in=90] (1.25,-1.125);
\draw[thick] (-.75,-1.125) to[out=90,in=180] (-.5,-.25) to[out=0,in=90] (-.25,-1.125);
\draw[thick] ( .75,-1.125) to[out=90,in=0] (.5,-.25) to[out=180,in=90] (.25,-1.125);
\draw[fill=black] (-.5,-.25) circle (1.5pt);
\draw[fill=black] ( .5,-.25) circle (1.5pt);
\draw[thick,-to] (-1.22,-.65) -- (-1.22,-.6); 
\draw[thick,to-] (-.76,-.75) -- (-.76,-.7); 
\draw[thick,to-] (.24,-.75) -- (.24,-.7); 
\draw[thick, dotted] (-1.5,1.125) to (1.5,1.125);\draw[thick, dotted] (-1.5,-1.125) to (1.5,-1.125);
}} \endxy
\end{gather*}\endgroup

At a trivalent vertex, the three edges either all oriented into the vertex, or one edge is oriented in and two edges out. The number of ``out'' oriented edges at each vertex is even.  

The degrees of various maps are summarized the table below. 
\begin{center}
\renewcommand{\arraystretch}{2}
  \begin{tabular}{|c|c|c|c|c|c|c|}
\hline
  map  &
$\xy (0,0)*{
\tikzdiagc[scale=.35,xscale=-1,yscale=-1]{
\draw[ultra thick,blue,to-] (-1,1) to[out=-75,in=180] (0,-.25) to[out=0,in=-105] (1,1);
}} \endxy$
  &
$\xy (0,0)*{
\tikzdiagc[scale=.35,xscale=1,yscale=-1]{
\draw[ultra thick,blue,to-] (-1,1) to[out=-75,in=180] (0,-.25) to[out=0,in=-105] (1,1);
}} \endxy$
  &
$\xy (0,0)*{
\tikzdiagc[scale=.35,xscale=-1,yscale=1]{
\draw[ultra thick,blue,to-] (-1,1) to[out=-75,in=180] (0,-.25) to[out=0,in=-105] (1,1);
}} \endxy$
  &
$\xy (0,0)*{
\tikzdiagc[scale=.35,xscale=1,yscale=1]{
\draw[ultra thick,blue,to-] (-1,1) to[out=-75,in=180] (0,-.25) to[out=0,in=-105] (1,1);
}} \endxy$
  &
$\xy (0,0)*{
\tikzdiagc[scale=.35]{
  \draw[ultra thick,blue] (0,-.4) -- (0, .25)node[pos=1, tikzdot]{};
  \draw[ultra thick,blue,-to] (0,-1) to (0,-.3); 
\draw[thick, dotted] (-1,1) to (1,1);\draw[thick, dotted] (-1,-1) to (1,-1);
}} \endxy$ 
  &
$\xy (0,0)*{
\tikzdiagc[scale=.35]{
\draw[very thick,Orange,densely dashed] (0,-1) to (0,.125); 
  \draw[ultra thick,blue,-to] (0,0) -- (0,1)node[pos=0, tikzdot]{}; 
\draw[thick, dotted] (-1,1) to (1,1);\draw[thick, dotted] (-1,-1) to (1,-1);
}} \endxy$  
  \\[1ex]
\hline
degree  & 0& 0 & 0 &0 &1 & 1  \\
\hline
\hline
  map  &
    $\xy (0,0)*{
\tikzdiagc[scale=.35,xscale=-1]{
\draw[ultra thick,blue] (-1,-1)-- (1, 1); 
\draw[ultra thick,blue,-to] (.6,.6) -- (.7,.7);
\draw[very thick,Orange,densely dashed] (1,-1)-- (-1, 1);  
}} \endxy$
  &
    $\xy (0,0)*{
\tikzdiagc[scale=.35]{
\draw[ultra thick,blue] (-1,-1)-- (1, 1); 
\draw[ultra thick,blue,-to] (.6,.6) -- (.7,.7);
\draw[very thick,Orange,densely dashed] (1,-1)-- (-1, 1);  
}} \endxy$
  &
$\xy (0,0)*{
\tikzdiagc[scale=.35,yscale=-1]{
    \draw[ultra thick,blue,] (0,0)-- (0,1);
    \draw[ultra thick,blue,to-] (-1,-1) -- (0,0);
    \draw[ultra thick,blue,to-] (1,-1) -- (0,0);
\draw[ultra thick,blue,to-] (0,.2) -- (0,.3);
}} \endxy$
  &
$\xy (0,0)*{
\tikzdiagc[scale=.35]{
\draw[ultra thick,blue] (0,0)-- (0,1); \draw[ultra thick,blue] (-1,-1) -- (0,0); \draw[ultra thick,blue] (1,-1) -- (0,0);
\draw[ultra thick,blue,-to] (0,.6) -- (0,.7);
\draw[ultra thick,blue,-to] (-.6,-.6) -- (-.7,-.7);
\draw[ultra thick,blue,-to] ( .4,-.4) -- ( .3,-.3);
}} \endxy$
  &
$\xy (0,0)*{
\tikzdiagc[scale=.35,xscale=1,yscale=1]{
\draw[very thick,Orange,densely dashed] (-1,1) to[out=-75,in=180] (0,-.125) to[out=0,in=-105] (1,1);
}} \endxy$
  &
$\xy (0,0)*{
\tikzdiagc[scale=.35,xscale=1,yscale=-1]{
\draw[very thick,Orange,densely dashed] (-1,1) to[out=-75,in=180] (0,-.125) to[out=0,in=-105] (1,1);
}} \endxy$
  \\[1ex]
\hline
degree  & 0& 0& -1& -1& 0& 0 \\
\hline
\end{tabular}
\end{center}

It's also convenient to introduce a crossings of a downward-oriented blue line with dashed orange line, a degree $0$ map defined as shown below. 

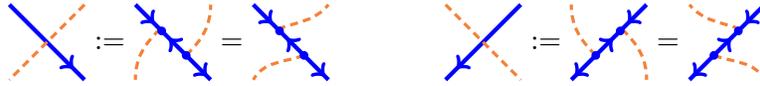
\begin{equation*}
\xy (0,0)*{
\tikzdiagc[scale=.5,xscale=1,yscale=-1]{
\draw[ultra thick,blue] (-1,-1)-- (1, 1); 
\draw[ultra thick,blue,-to] (.6,.6) -- (.7,.7);
\draw[very thick,Orange,densely dashed] (1,-1)-- (-1, 1);  
}} \endxy
  :=
  \xy (0,0)*{
\tikzdiagc[scale=.5,xscale=1,yscale=-1]{
\draw[very thick,Orange,densely dashed] (.35,.35) to[out=-45,in=90] (1,-1);  
\draw[very thick,Orange,densely dashed] (-.35,-.35) to[out=135,in=-90] (-1,1);  
\draw[ultra thick,blue] (-1,-1)-- (1, 1); 
\draw[ultra thick,blue,to-] (-.5,-.5) -- (-.6,-.6);
\draw[ultra thick,blue,-to] (.8,.8) -- (.9,.9);
\draw[ultra thick,blue,-to] (0,0) -- (-.1,-.1);
\draw[blue,fill=blue] (-.3,-.3) circle (2.5pt);
\draw[blue,fill=blue] (.35,.35) circle (2.5pt);
}} \endxy
=
  \xy (0,0)*{
\tikzdiagc[scale=.5,xscale=1,yscale=-1]{
\draw[very thick,Orange,densely dashed] (-.35,-.35) to[out=-20,in=130] (1,-1);  
\draw[very thick,Orange,densely dashed] ( .35, .35) to[out=160,in=-60] (-1,1);  
\draw[ultra thick,blue] (-1,-1)-- (1, 1); 
\draw[ultra thick,blue,to-] (-.5,-.5) -- (-.6,-.6);
\draw[ultra thick,blue,-to] (.8,.8) -- (.9,.9);
\draw[ultra thick,blue,-to] (0,0) -- (-.1,-.1);
\draw[blue,fill=blue] (-.3,-.3) circle (2.5pt);
\draw[blue,fill=blue] (.35,.35) circle (2.5pt);
}} \endxy
\mspace{70mu}
\xy (0,0)*{
\tikzdiagc[scale=.5,xscale=-1,yscale=-1]{
\draw[ultra thick,blue] (-1,-1)-- (1, 1); 
\draw[ultra thick,blue,-to] (.6,.6) -- (.7,.7);
\draw[very thick,Orange,densely dashed] (1,-1)-- (-1, 1);  
}} \endxy
  :=
  \xy (0,0)*{
\tikzdiagc[scale=.5,xscale=-1,yscale=-1]{
\draw[very thick,Orange,densely dashed] (.35,.35) to[out=-45,in=90] (1,-1);  
\draw[very thick,Orange,densely dashed] (-.35,-.35) to[out=135,in=-90] (-1,1);  
\draw[ultra thick,blue] (-1,-1)-- (1, 1); 
\draw[ultra thick,blue,to-] (-.5,-.5) -- (-.6,-.6);
\draw[ultra thick,blue,-to] (.8,.8) -- (.9,.9);
\draw[ultra thick,blue,-to] (0,0) -- (-.1,-.1);
\draw[blue,fill=blue] (-.3,-.3) circle (2.5pt);
\draw[blue,fill=blue] (.35,.35) circle (2.5pt);
}} \endxy
=
\xy (0,0)*{
\tikzdiagc[scale=.5,xscale=-1,yscale=-1]{
\draw[very thick,Orange,densely dashed] (-.35,-.35) to[out=-20,in=130] (1,-1);  
\draw[very thick,Orange,densely dashed] ( .35, .35) to[out=160,in=-60] (-1,1);  
\draw[ultra thick,blue] (-1,-1)-- (1, 1); 
\draw[ultra thick,blue,to-] (-.5,-.5) -- (-.6,-.6);
\draw[ultra thick,blue,-to] (.8,.8) -- (.9,.9);
\draw[ultra thick,blue,-to] (0,0) -- (-.1,-.1);
\draw[blue,fill=blue] (-.3,-.3) circle (2.5pt);
\draw[blue,fill=blue] (.35,.35) circle (2.5pt);
}} \endxy
\end{equation*}

Also, the following relations hold. 
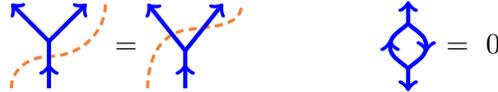
\begin{equation*}
\xy (0,0)*{
\tikzdiagc[scale=.5,yscale=-1]{
\draw[very thick,Orange,densely dashed] (-1,1.25) to[out=-90,in=180] (0,.4) to[out=0,in=90] (1.5,-1);  
\draw[ultra thick,blue,] (0,0)-- (0,1.25);
\draw[ultra thick,blue,to-] (-1,-1) -- (0,0);
\draw[ultra thick,blue,to-] (1,-1) -- (0,0);
\draw[ultra thick,blue,to-] (0,.6) -- (0,.7);
}} \endxy
=
\xy (0,0)*{
\tikzdiagc[scale=.5,yscale=-1]{
\draw[very thick,Orange,densely dashed] (-1,1.25) to[out=-90,in=170] (0,-.3) to[out=-10,in=90] (1.5,-1);  
\draw[ultra thick,blue,] (0,.25)-- (0,1.25);
\draw[ultra thick,blue,to-] (-1,-1) -- (0,.25);
\draw[ultra thick,blue,to-] (1,-1) -- (0,.25);
\draw[ultra thick,blue,to-] (0,.6) -- (0,.7);
}} \endxy
\mspace{90mu}
\xy (0,0)*{
\tikzdiagc[scale=.4,yscale=1]{
\draw[ultra thick,blue,to-] (0,-1.5) -- (0,-.75);
\draw[ultra thick,blue,-to] (0,.75)-- (0,1.5);
\draw[ultra thick,blue] (0,-.75) to[out=45,in=-90] (.6,0) to[out=90,in=-45] (0,.75);
\draw[ultra thick,blue] (0,-.75) to[out=135,in=-90] (-.6,0) to[out=90,in=-135] (0,.75);
\draw[ultra thick,blue,-to] (.6,-.15) -- (.6,-.2);
\draw[ultra thick,blue,-to] (-.6,.15) -- (-.6,.2);
}} \endxy
\ =\ 0
\end{equation*}

Similar to the decomposition of $B\otimes B$ and using oriented lines, we obtain the following direct sum decomposition  
\begin{equation*}
    B\otimes \undB \ \cong \ B\{1\} \oplus \undB\{-1\}
\end{equation*}
from Corollary~\ref{cor_dir_sum} diagrammatically (note the minus sign in one of the maps). 

\begin{equation}\label{eq_dir_sum}
\xy (0,0)*{
\tikzdiagc[scale=.5]{
  \node at (-5.4,0) {\small \text{$B\{1\}$}};
\node at ( 0,0) {\small \text{$B\otimes\undB$}};
\node at ( 5.50,0) {\small \text{$\undB\{-1\}$}};
\draw[->] (-4.5,.25) to[out=20,in=160] (-1,.25);
\draw[<-] (-4.5,-.25) to[out=-20,in=-160] (-1,-.25);
\draw[<-] ( 4.5,.25) to[out=160,in=20] ( 1,.25);
\draw[->] ( 4.5,-.25) to[out=-160,in=-20] ( 1,-.25);
\begin{scope}[scale=.95,shift={(-3,2)}]
\node at (-1.8,0) {\text{$-$}};
  \draw[ultra thick,blue] (0,-1)-- (0,-.2); \draw[ultra thick,blue] (-1,1) -- (-.55,.35); \draw[ultra thick,blue] (1,1) -- (0,-.2);
\draw[very thick,Orange,densely dashed] (0,-.2) to (-.55,.35);   
 \draw[blue,fill=blue] (0,-.2) circle (2.5pt);
 \draw[blue,fill=blue] (-.55,.35) circle (2.5pt);
 \draw[ultra thick,blue,-to] (0,-.5) -- (0,-.4);
\draw[ultra thick,blue,-to] (.525,.45) -- (.405,.3);
\draw[ultra thick,blue,to-] (-.865,.8) -- (-.790,.7);
 \draw[thick, dotted] (-1.25,1) to (1.25,1);\draw[thick, dotted] (-1.25,-1) to (1.25,-1);
\end{scope}
\begin{scope}[scale=.95,shift={(-3,-2)}]
\draw[ultra thick,blue] (0,0)-- (0,1); \draw[ultra thick,blue] (-1,-1) -- (0,0); \draw[ultra thick,blue] (1,-1) -- (0,0);
\draw[ultra thick,blue,-to] (0,.6) -- (0,.7);
\draw[ultra thick,blue,-to] (-.4,-.4) -- (-.3,-.3);
\draw[ultra thick,blue,to-] (.7,-.7) -- (.6,-.6);
\draw[thick, dotted] (-1.25,1) to (1.25,1);\draw[thick, dotted] (-1.25,-1) to (1.25,-1);
\end{scope}
\begin{scope}[scale=.95,shift={(3,2)}]
\draw[ultra thick,blue] (-.5,-1)-- (-.5,0)node[pos=1, tikzdot]{};
\draw[ultra thick,blue] (.5,-1) -- (.5,1);
\draw[ultra thick,blue,-to] (-.5,-.4) -- (-.5,-.3);
\draw[ultra thick,blue,-to] (.5,-.1) -- (.5,-.3);
\draw[thick, dotted] (-1.25,1) to (1.25,1);\draw[thick, dotted] (-1.25,-1) to (1.25,-1);
\end{scope}
\begin{scope}[scale=.95,shift={(3,-2)}]
\draw[ultra thick,blue] (0,0)-- (0, -1); \draw[ultra thick,blue] (-1,1) -- (0,0); \draw[ultra thick,blue] (1,1) -- (0,0);
\draw[ultra thick,blue,-to] (0,-.6) -- (0,-.7);
\draw[ultra thick,blue,-to] (.4,.4) -- (.3,.3);
\draw[ultra thick,blue,to-] (-.7,.7) -- (-.6,.6);
\draw[thick, dotted] (-1.25,1) to (1.25,1);\draw[thick, dotted] (-1.25,-1) to (1.25,-1);
\end{scope}
}} \endxy  
\end{equation}
 There's flexibility in choosing some arrows in a direct sum decomposition of $B\otimes\undB$, as in the earlier discussion about the equivalent case of decomposing $B\otimes B$.

\begin{rem}
Let $\sbim$ be the monoidal category of 2-variable odd Soergel bimodules generated by bimodules $B,\undR$ and their grading shifts (see more details about $\sbim$ in Section~\ref{sec_groth}).  Bimodule $B$ has an antiinvolution $\phi$ given by $\phi(x\otimes y)=y\otimes x$. Bimodule $\undR$ has an antiinvolution $\phi$ given by $\phi(x\undone y)=y\undone x$. 
Antiinvolutions $\phi$ extend to an involutive antiequivalence $\phi:\sbim \lra \sbim^{\mathrm{op}}$ of the category $\sbim$ that takes $B$ to $B$ and $\undB$ to $\undB$. 
In our graphical description of $\sbim$ some generating maps are invariant under the  reflection in a vertical line, such as in (\ref{eq:enddot}). Our diagrammatical notations for several maps break reflectional symmetry, requiring adding  a minus sign to the reflected diagram, including in (\ref{eq:startdot}). These signs later propagate in formulas: for instance, observe the absence of signs in the direct sum decomposition given by (\ref{eq-no-signs}) and the presence of a single minus sign in the decomposition (\ref{eq_dir_sum}).  
\end{rem}

%
%

\section{Odd Rouquier complexes and invertibility}
Consider the following two-term complexes of graded $B$-modules, where $\undB$ and $B$ terms are 
placed in cohomological degree $0$. 
The differential is given by maps (\ref{eq:enddot}) and the map $\underline{\Delta}$ obtained from (\ref{eq:startdot}) by tensoring with $\undone$, respectively. 
\begin{align*}
\eR & := 0\lra B\xrightarrow{\ m\ } R\{-1\} \lra 0, \\ 
\eR' & := 0\lra R\{1\}\xrightarrow{ \underline{\Delta} 
} \undB \lra 0.
\end{align*}

Complexes $\eR,\eR'$ can be viewed as odd analogues of the Rouquier complexes. 

\begin{thm}
There are homotopy equivalences of complexes of graded $B$-modules 
\( \eR\otimes_R\eR'\cong_h R\) and \( \eR'\otimes_R\eR\cong_h R\).
\end{thm}
Here $R$ denotes the identity $R$-bimodule, viewed as a complex concentrated in homological degree $0$. 

\begin{proof}
The complex $\eR\otimes_R\eR'$ is given by forming a commutative square of bimodules below  

\begin{equation*}
\xy (0,0)*{
\tikzdiagc[scale=.7]{
\node at (0,0) {\text{$B\otimes R\{1\}$}};
\node at (6,0) {\text{$R\otimes R$}};
\node at (0,-3) {\text{$B\otimes\undB$}};
\node at (6,-3) {\text{$R\otimes \undB\{-1\}$}};
\draw[thick,-to] (1.5,0) to (4.25,0);
\draw[thick,-to] (1.5,-3) to (4.25,-3);
\draw[thick,-to] (0,-.5) to (0,-2.5);
\draw[thick,-to] (6,-.5) to (6,-2.5);
\node at (3,-.45) {\small \text{$d_0$}};
\node at (3,-3.45) {\small \text{$d_3$}};
\node at (.5,-1.5) {\small \text{$d_1$}};
\node at (5.5,-1.5) {\small \text{$d_2$}};
\begin{scope}[shift={(2.8,.80)},scale=.6]
\draw[ultra thick,blue] (0,-.4) -- (0, .25)node[pos=1, tikzdot]{};
\draw[ultra thick,blue,-to] (0,-1) to (0,-.3); 
\draw[thick, dotted] (-1,1) to (1,1);\draw[thick, dotted] (-1,-1) to (1,-1);
\end{scope}
\begin{scope}[shift={(2.8,-2.20)},scale=.6]
\draw[ultra thick,blue] (-.5,-.4) -- (-.5, .25)node[pos=1, tikzdot]{};
\draw[ultra thick,blue,-to] (-.5,-1) to (-.5,-.3); 
\draw[ultra thick,blue] (.5,-1) -- (.5, 1);
\draw[ultra thick,blue,-to] (.5,-.2) to (.5,-.3); 
\draw[thick, dotted] (-1,1) to (1,1);\draw[thick, dotted] (-1,-1) to (1,-1);
\end{scope}
\begin{scope}[shift={(-.80,-1.5)},scale=.6]
\draw[ultra thick,blue] (.5,.4) -- (.5, -.025)node[pos=1, tikzdot]{};
\draw[ultra thick,blue,-to] (.5,1) to (.5,.3); 
\draw[ultra thick,blue] (-.5,-1) -- (-.5, 1);
\draw[ultra thick,blue,-to] (-.5,.2) to (-.5,.3); 
\draw[thick, dotted] (-1,1) to (1,1);\draw[thick, dotted] (-1,-1) to (1,-1);
\end{scope}
\begin{scope}[shift={(6.80,-1.5)},scale=.6,yscale=-1]
\draw[ultra thick,blue] (0,-.4) -- (0, .25)node[pos=1, tikzdot]{};
\draw[ultra thick,blue,-to] (0,-1) to (0,-.3); 
\draw[thick, dotted] (-1,1) to (1,1);\draw[thick, dotted] (-1,-1) to (1,-1);
\end{scope}
}}\endxy
\end{equation*}  
then adding a minus sign to  the map $d_2$ and collapsing the square into the complex $(C,\widehat{d})$ below. 

\begin{equation*}
\xy (0,0)*{
\tikzdiagc[scale=.6]{
\node at (-2,0) {\text{$C\colon$}};
\node at (0,0) {\text{$B\{1\}$}};
\node at (7,0) {\text{$R\oplus (B\otimes\undB)$}};
\node at (14,0) {\text{$\undB\{-1\}$}};
\draw[thick,-to] (1,0) to (5,0);
\draw[thick,-to] (9,0) to (12.80,0);
\node at (3,1) {\small \text{$\widehat{d}_{-1}=
    \begin{pmatrix}d_0\\d_1\end{pmatrix}
    $}};
\node at (10.85,1) {\small \text{$\widehat{d}_{0}=(-d_2,d_3)$}};
}}\endxy
\end{equation*}
Introduce maps $h_0,h_3$ and $\jmath$ between terms in the above commutative square, as shown below. 

\begin{equation*}
\xy (0,0)*{
\tikzdiagc[scale=.8]{
\node at (0,0) {\text{$B\{1\}$}};
\node at (8,0) {\text{$R$}};
\node at (0,-5) {\text{$B\otimes\undB$}};
\node at (8,-5) {\text{$\undB\{-1\}$}};
\draw[thick,-to] (1,0) to (7.25,0);
\draw[thick,-to] (1,-5) to (7.0,-5);
\draw[thick,-to] (0,-.5) to (0,-4.5);
\draw[thick,-to] (8,-.5) to (8,-4.5);
\node at (3,.75) {\small \text{$d_0=$}};
\node at (3.8,-4.2) {\small \text{$d_3=$}};
\node at (.65,-1.45) {\small \text{$d_1=$}};
\node at (8.7,-2.5) {\small \text{$d_2=$}};
\node at (8.6,-3.5) {\small \text{$(-)$}};
\begin{scope}[shift={(4.425,.825)},scale=.6]
\draw[ultra thick,blue] (0,-.4) -- (0, .25)node[pos=1, tikzdot]{};
\draw[ultra thick,blue,-to] (0,-1) to (0,-.3); 
\draw[thick, dotted] (-1,1) to (1,1);\draw[thick, dotted] (-1,-1) to (1,-1);
\end{scope}
\begin{scope}[shift={(5.0,-4.2)},scale=.6]
\draw[ultra thick,blue] (-.5,-.4) -- (-.5, .25)node[pos=1, tikzdot]{};
\draw[ultra thick,blue,-to] (-.5,-1) to (-.5,-.3); 
\draw[ultra thick,blue] (.5,-1) -- (.5, 1);
\draw[ultra thick,blue,-to] (.5,-.2) to (.5,-.3); 
\draw[thick, dotted] (-1,1) to (1,1);\draw[thick, dotted] (-1,-1) to (1,-1);
\end{scope}
\begin{scope}[shift={(1.95,-1.5)},scale=.6]
\draw[very thick,Orange,densely dashed] (.5,-.150) to (.5,-1); 
\draw[ultra thick,blue] (.5,.4) -- (.5, -.025)node[pos=1, tikzdot]{};
\draw[ultra thick,blue,-to] (.5,1) to (.5,.3); 
\draw[ultra thick,blue] (-.5,-1) -- (-.5, 1);
\draw[ultra thick,blue,-to] (-.5,.2) to (-.5,.3); 
\draw[thick, dotted] (-1,1) to (1,1);\draw[thick, dotted] (-1,-1) to (1,-1);
\end{scope}
\begin{scope}[shift={(9.9,-2.5)},scale=.6,yscale=-1]
\draw[ultra thick,blue] (0,-.4) -- (0, .25)node[pos=1, tikzdot]{};
\draw[ultra thick,blue,-to] (0,-1) to (0,-.3); 
\draw[thick, dotted] (-1,1) to (1,1);\draw[thick, dotted] (-1,-1) to (1,-1);
\end{scope}
\draw[densely dashed,-to] (-.5,-4.45) to[out=110,in=-110] (-.5,-.5);
\draw[densely dashed,-to] (7.5,-5.5) to[out=-170,in=-10] (.5,-5.5);
\draw[densely dashed,-to] (7.5,-.5) to (.8,-4.5);
\node at (-3.4,-2.5) {\small \text{$h_0=$}};
\node at ( 3,-6.75) {\small \text{$h_3=$}};
\begin{scope}[shift={(-2,-2.5)},scale=.6]
\draw[ultra thick,blue] (0,0)-- (0,1); \draw[ultra thick,blue] (-1,-1) -- (0,0); \draw[ultra thick,blue] (1,-1) -- (0,0);
\draw[ultra thick,blue,-to] (0,.6) -- (0,.7);
\draw[ultra thick,blue,-to] (-.4,-.4) -- (-.3,-.3);
\draw[ultra thick,blue,to-] (.7,-.7) -- (.6,-.6);
\draw[thick, dotted] (-1.25,1) to (1.25,1);\draw[thick, dotted] (-1.25,-1) to (1.25,-1);
\end{scope}
\begin{scope}[shift={(4.5,-6.75)},scale=.6,xscale=-1,yscale=-1]
\draw[ultra thick,blue] (0,0)-- (0,1); \draw[ultra thick,blue] (-1,-1) -- (0,0); \draw[ultra thick,blue] (1,-1) -- (0,0);
\draw[ultra thick,blue,-to] (0,.6) -- (0,.7);
\draw[ultra thick,blue,-to] (-.4,-.4) -- (-.3,-.3);
\draw[ultra thick,blue,to-] (.7,-.7) -- (.6,-.6);
\draw[thick, dotted] (-1.25,1) to (1.25,1);\draw[thick, dotted] (-1.25,-1) to (1.25,-1);
\end{scope}
\begin{scope}[shift={(5.5,-2.1)},scale=.5]
\fill[white]  (-3.5,1.2) to (1.35,1.2) to (1.35,-1.02) to (-3.5,-1.02)--cycle;
\node at (-2.5,0) {\small \text{$j=$}};
\draw[ultra thick,blue,to-] (-1,1) to[out=-75,in=180] (0,-.25) to[out=0,in=-105] (1,1);
\draw[thick, dotted] (-1.25,1) to (1.25,1);\draw[thick, dotted] (-1.25,-1) to (1.25,-1);
\end{scope}
}}\endxy
\end{equation*}

The following relations hold 
\begin{equation*}
   h_0 d_1 = \id_{B\{1\}}, \ \ 
   d_2  =  d_3 \jmath, \ \ 
   d_3 h_3 = \id_{\undB\{-1\}}, \ \ h_0h_3 = 0. 
\end{equation*}
Thus, $d_1$ is split injective, with a section $h_0$. Likewise, $d_3$ is split surjective, with $h_3$ as a section. 

We would like to check that $B\otimes \undB=\mathrm{im}(d_1)\oplus \mathrm{im}(h_3)$. 
 Consider the map $d_3'$ given by  
 \begin{equation*}
d_3' =  
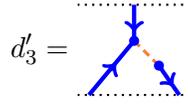
\xy (0,0)*{
\tikzdiagc[scale=.6,yscale=-1,xscale=-1]{
  \draw[ultra thick,blue] (0,-1)-- (0,-.2); \draw[ultra thick,blue] (-1,1) -- (-.55,.35); \draw[ultra thick,blue] (1,1) -- (0,-.2);
\draw[very thick,Orange,densely dashed] (0,-.2) to (-.55,.35);   
 \draw[blue,fill=blue] (0,-.2) circle (2.5pt);
 \draw[blue,fill=blue] (-.55,.35) circle (2.5pt);
 \draw[ultra thick,blue,-to] (0,-.5) -- (0,-.4);
\draw[ultra thick,blue,-to] (.525,.45) -- (.405,.3);
\draw[ultra thick,blue,to-] (-.865,.8) -- (-.790,.7);
 \draw[thick, dotted] (-1.25,1) to (1.25,1);\draw[thick, dotted] (-1.25,-1) to (1.25,-1);
}} \endxy
\end{equation*}

  Map $d_3'$ is a rotation  of the top left diagram in (\ref{eq_dir_sum}). 
 Then 
 \begin{equation}\label{eq-no-signs}
 d_3' d_1=0, \ \ h_0h_3=0,  \ \  d_3'h_3=\id_{\undB\{-1\}}, \ \
 d_1h_0+h_3 d_3' = \id_{B\otimes \undB}.
 \end{equation} 
  Pairs of maps $(d_1,h_0)$ and $(d'_3,h_3)$ give a direct sum decomposition $B\otimes \undB\cong B\{1\}\oplus \undB\{-1\}$.

Complex $C$ above splits into the direct sum of three subcomplexes: 
\begin{eqnarray*}
 0 \lra & B\{1\} & \stackrel{d_0+d_1}{\lra} B\{1\}\lra \ 0 , \\
        & 0  & \, \lra \ \ R \ \ \ \  \lra \ 0 , \\
        & 0  &  \lra h_3(\undB\{-1\}) \ \stackrel{d_3}{\lra} \ \undB\{-1\} \lra 0, 
\end{eqnarray*}
where the middle complex consists of pairs $(a,\jmath(a)), a\in R$. The first and third complexes are contractible, while the middle complex is the identity bimodule $R$. Consequently, \( \eR\otimes_R\eR'\cong_h R\). A similar computation, changing the order of terms in tensor products and reflecting all map diagrams about horizontal axes, shows that \( \eR'\otimes_R\eR\cong_h R\).
\end{proof}

\begin{cor}
Functors of tensoring with bimodule complexes $\eR$ and $\eR'$ are mutually-invertible functors in the homotopy category of complexes of graded $R$-modules. 
\end{cor}

\begin{prop} After removing contractible summands, 
complex $\eR^n$ for $n>0$ simplifies to the $(n+1)$-term complex, nontrivial in cohomological degrees from $0$ to $n$, with the head 
\begin{equation*}
  \cdots 
\xra{\xy (0,0)*{
\tikzdiagc[scale=.35]{
\draw[ultra thick,blue] (-.5,-1) -- (-.5, -.15)node[pos=1, tikzdot]{};
\draw[very thick, densely dashed,Orange] (.5,-1) -- (.5,.15);
\draw[ultra thick,blue] (.5,.15) -- (.5, 1)node[pos=0, tikzdot]{};
\draw[thick, dotted] (-1,1) to (1,1);
\draw[thick, dotted] (-1,-1) to (1,-1);
}} \endxy}
B\{5-n\}
\xra{\xy (0,0)*{
\tikzdiagc[scale=.35]{
\draw[ultra thick,blue] (-.5,-1) -- (-.5,-.45)node[pos=1, tikzdot]{};
\draw[very thick,densely dashed,Orange] (-.5,.45) to[out=-90,in=180] (-.1,-.1) to[out=0,in=-90] (.5,1);
\draw[ultra thick,blue] (-.5,.45) -- (-.5, 1)node[pos=0, tikzdot]{};
\draw[thick, dotted] (-1,1) to (1,1);
\draw[thick, dotted] (-1,-1) to (1,-1);
}} \endxy}
\undB\{3-n\}
\xra{\xy (0,0)*{
\tikzdiagc[scale=.35]{
\draw[ultra thick,blue] (-.5,-1) -- (-.5, -.15)node[pos=1, tikzdot]{};
\draw[very thick, densely dashed,Orange] (.5,-1) -- (.5,.15);
\draw[ultra thick,blue] (.5,.15) -- (.5, 1)node[pos=0, tikzdot]{};
\draw[thick, dotted] (-1,1) to (1,1);
\draw[thick, dotted] (-1,-1) to (1,-1);
}} \endxy}
B\{1-n\}
\xra{\xy (0,0)*{
\tikzdiagc[scale=.35]{
\draw[ultra thick,blue] (0,-1) -- (0,0)node[pos=1, tikzdot]{};
\draw[thick, dotted] (-1,1) to (1,1);
\draw[thick, dotted] (-1,-1) to (1,-1);
}} \endxy}
R\{-n\}\lra 0   
\end{equation*}
and the tail 
\begin{equation*}
0  \lra B\{n-1\}
\xra{\xy (0,0)*{
\tikzdiagc[scale=.35]{
\draw[ultra thick,blue] (-.5,-1) -- (-.5,-.45)node[pos=1, tikzdot]{};
\draw[very thick,densely dashed,Orange] (-.5,.45) to[out=-90,in=180] (-.1,-.1) to[out=0,in=-90] (.5,1);
\draw[ultra thick,blue] (-.5,.45) -- (-.5, 1)node[pos=0, tikzdot]{};
\draw[thick, dotted] (-1,1) to (1,1);
\draw[thick, dotted] (-1,-1) to (1,-1);
}} \endxy}
\undB\{n-3\}
\xra{\xy (0,0)*{
\tikzdiagc[scale=.35]{
\draw[ultra thick,blue] (-.5,-1) -- (-.5, -.15)node[pos=1, tikzdot]{};
\draw[very thick, densely dashed,Orange] (.5,-1) -- (.5,.15);
\draw[ultra thick,blue] (.5,.15) -- (.5, 1)node[pos=0, tikzdot]{};
\draw[thick, dotted] (-1,1) to (1,1);
\draw[thick, dotted] (-1,-1) to (1,-1);
}} \endxy}
B\{n-5\}
\xra{\xy (0,0)*{
\tikzdiagc[scale=.35]{
\draw[ultra thick,blue] (-.5,-1) -- (-.5,-.45)node[pos=1, tikzdot]{};
\draw[very thick,densely dashed,Orange] (-.5,.45) to[out=-90,in=180] (-.1,-.1) to[out=0,in=-90] (.5,1);
\draw[ultra thick,blue] (-.5,.45) -- (-.5, 1)node[pos=0, tikzdot]{};
\draw[thick, dotted] (-1,1) to (1,1);
\draw[thick, dotted] (-1,-1) to (1,-1);
}} \endxy}
\undB\{n-7\}
\xra{\xy (0,0)*{
\tikzdiagc[scale=.35]{
\draw[ultra thick,blue] (-.5,-1) -- (-.5, -.15)node[pos=1, tikzdot]{};
\draw[very thick, densely dashed,Orange] (.5,-1) -- (.5,.15);
\draw[ultra thick,blue] (.5,.15) -- (.5, 1)node[pos=0, tikzdot]{};
\draw[thick, dotted] (-1,1) to (1,1);
\draw[thick, dotted] (-1,-1) to (1,-1);
}} \endxy}
\cdots
\end{equation*}
for odd $n$ and 
\begin{equation*}
  0  \lra
\undB\{n-1\}
\xra{\xy (0,0)*{
\tikzdiagc[scale=.35]{
\draw[ultra thick,blue] (-.5,-1) -- (-.5, -.15)node[pos=1, tikzdot]{};
\draw[very thick, densely dashed,Orange] (.5,-1) -- (.5,.15);
\draw[ultra thick,blue] (.5,.15) -- (.5, 1)node[pos=0, tikzdot]{};
\draw[thick, dotted] (-1,1) to (1,1);
\draw[thick, dotted] (-1,-1) to (1,-1);
}} \endxy}
  B\{n-3\}
\xra{\xy (0,0)*{
\tikzdiagc[scale=.35]{
\draw[ultra thick,blue] (-.5,-1) -- (-.5,-.45)node[pos=1, tikzdot]{};
\draw[very thick,densely dashed,Orange] (-.5,.45) to[out=-90,in=180] (-.1,-.1) to[out=0,in=-90] (.5,1);
\draw[ultra thick,blue] (-.5,.45) -- (-.5, 1)node[pos=0, tikzdot]{};
\draw[thick, dotted] (-1,1) to (1,1);
\draw[thick, dotted] (-1,-1) to (1,-1);
}} \endxy}
\undB\{n-5\}
\xra{\xy (0,0)*{
\tikzdiagc[scale=.35]{
\draw[ultra thick,blue] (-.5,-1) -- (-.5, -.15)node[pos=1, tikzdot]{};
\draw[very thick, densely dashed,Orange] (.5,-1) -- (.5,.15);
\draw[ultra thick,blue] (.5,.15) -- (.5, 1)node[pos=0, tikzdot]{};
\draw[thick, dotted] (-1,1) to (1,1);
\draw[thick, dotted] (-1,-1) to (1,-1);
}} \endxy}
B\{n-7\}
\xra{\xy (0,0)*{
\tikzdiagc[scale=.35]{
\draw[ultra thick,blue] (-.5,-1) -- (-.5,-.45)node[pos=1, tikzdot]{};
\draw[very thick,densely dashed,Orange] (-.5,.45) to[out=-90,in=180] (-.1,-.1) to[out=0,in=-90] (.5,1);
\draw[ultra thick,blue] (-.5,.45) -- (-.5, 1)node[pos=0, tikzdot]{};
\draw[thick, dotted] (-1,1) to (1,1);
\draw[thick, dotted] (-1,-1) to (1,-1);
}} \endxy}
\cdots
\end{equation*}
for even $n$.
Under the differential maps, $1\otimes_s 1$ and $1\otimes_s\undone$ are sent to $1\otimes_s\undone x_1 + x_2\otimes_s\undone\in \undB$ and to $1\otimes_sx_2-x_2\otimes_s1\in B$, respectively. 
\end{prop}

The proposition can be proved by induction on $n$ and a direct computation using Gauss elimination. 
\qed

The case of $\eR'^n$ is similar:
\begin{prop}
After removing contractible summands,  complex $\eR'^n$ for $n>0$ reduces to the following $(n+1)$-term complex that lives in cohomological degrees from $-n$ to $0$:
\begin{equation*}
  0 \lra R\{n\}
\xra{\xy (0,0)*{
\tikzdiagc[scale=.35]{
\draw[very thick,densely dashed,Orange] (-.5,0) to[out=-90,in=180] (-.1,-.75) to[out=0,in=-90] (.5,1);
\draw[ultra thick,blue] (-.5,0) -- (-.5, 1)node[pos=0, tikzdot]{};
\draw[thick, dotted] (-1,1) to (1,1);
\draw[thick, dotted] (-1,-1) to (1,-1);
}} \endxy}
\undB\{n-1\}
\xra{\xy (0,0)*{
\tikzdiagc[scale=.35]{
\draw[ultra thick,blue] (-.5,-1) -- (-.5, -.15)node[pos=1, tikzdot]{};
\draw[very thick, densely dashed,Orange] (.5,-1) -- (.5,.15);
\draw[ultra thick,blue] (.5,.15) -- (.5, 1)node[pos=0, tikzdot]{};
\draw[thick, dotted] (-1,1) to (1,1);
\draw[thick, dotted] (-1,-1) to (1,-1);
}} \endxy}
B\{n-3\}
\xra{\xy (0,0)*{
\tikzdiagc[scale=.35]{
\draw[ultra thick,blue] (-.5,-1) -- (-.5,-.45)node[pos=1, tikzdot]{};
\draw[very thick,densely dashed,Orange] (-.5,.45) to[out=-90,in=180] (-.1,-.1) to[out=0,in=-90] (.5,1);
\draw[ultra thick,blue] (-.5,.45) -- (-.5, 1)node[pos=0, tikzdot]{};
\draw[thick, dotted] (-1,1) to (1,1);
\draw[thick, dotted] (-1,-1) to (1,-1);
}} \endxy}
\undB\{n-5\}
\xra{\xy (0,0)*{
\tikzdiagc[scale=.35]{
\draw[ultra thick,blue] (-.5,-1) -- (-.5, -.15)node[pos=1, tikzdot]{};
\draw[very thick, densely dashed,Orange] (.5,-1) -- (.5,.15);
\draw[ultra thick,blue] (.5,.15) -- (.5, 1)node[pos=0, tikzdot]{};
\draw[thick, dotted] (-1,1) to (1,1);
\draw[thick, dotted] (-1,-1) to (1,-1);
}} \endxy}
\cdots  . 
\end{equation*}

\end{prop}

\begin{rem}
Adding the signed permutation bimodule to $\eR,\eR'$ gives the 2-strand motion braid group action on the homotopy category of graded $R$-modules, see A.-L.~Thiel~\cite{T} for the corresponding action of the group of motion braids or virtual braids in the even case for any number of strands. 
\end{rem}


\section{Grothendieck ring} 
\label{sec_groth}
Recall that $\sbim$ is the category of 2-variable odd Soergel bimodules generated as the monoidal category by bimodules $B$ and $\undR$ and their grading shifts. Hom spaces in this category are all grading-preserving homomorphisms of bimodules. 
We can also define the larger spaces 
\begin{equation*}
    \HOM_{\sbim}(M,N) \ := \ \oplus_{n\in \Z} \Hom_{\sbim}(M\{n\},N).
\end{equation*}
These $\HOM$ spaces are naturally modules over the center of $R$, 
\begin{equation}
Z(R)\ \cong \ \Z[x_1^2,x_2^2]\subset R.
\end{equation}
The left and right actions of $Z(R)$ on these hom spaces are not equal, in general.

Indecomposable objects of $\sbim$, up to shifts, are 
$ R,  \undR,  B, \undB,$
with tensor product decompositions
\begin{equation*}
    \undR\otimes \undR \cong R, \ \   \undB \cong B\otimes \undR \cong \undR\otimes B, \ \  B\otimes B \cong B\{-1\}\oplus \undB\{1\}. 
\end{equation*}
Consider the split Grothendieck ring $K_0$ of $\sbim$. Grading shift functor induces a $\Z[q,q^{-1}]$-module structure on $K_0(\sbim)$. The latter is a free rank four $\Z[q,q^{-1}]$-module with a basis 
\[  1 = [R], \ \  c := [\undR], \ \ b := [B], \ \ bc = [\undB],
\]
generators $b,c$, and multiplication rules 
\begin{equation}\label{eq_mult_rules}
    c^2=1, \ \ cb=bc, \ \ b^2 = q^{-1}b+qbc . 
\end{equation}

Since the element $c$ is central in $K_0(\sbim)$, the latter is a commutative associative $\Z[q,q^{-1}]$-algebra (commutativity fails for analogous algebras for three or more strands).   

Define a $\Z[q,q^{-1}]$-semilinear form on $K_0(\sbim)$ by 
\begin{equation*}
    ([M],[N]) \ := \ \mathrm{gdim}(\HOM(M,N)), 
\end{equation*}
where $\mathrm{gdim}$ denotes the graded dimension. This form is $\Z[q,q^{-1}]$-linear in the second variable and $\Z[q,q^{-1}]$-antilinear in the first variable. We have $(cm,cn)=(m,n)$, for $m,n\in K_0(\sbim)$, and   
\begin{eqnarray}\label{form_one}
    (1,1) & =& \frac{1}{(1-q^4)^2}, \\ (1,c) & = & (c,1) \ = \ 0, \\ (b,1)  & = & (1,bc) \ = \  \frac{q}{(1-q^4)^2} ,
    \label{eq_three} \\ (1,b) & = & (bc,1) \ =\ \frac{q^3}{(1-q^4)^2}. \label{eq_four} 
\end{eqnarray}
The inner product $(1,c)=0$ since $\HOM(R,\undR)=0$, which follows by a direct computation. 
The inner product $(b,1)$ above is computed via adjointness isomorphism 
\begin{equation*}
    \HOM(B,R)=\HOM(R\otimes_{R^s}R \{-1\},R) \cong \Hom({}_{R^s}R_R,{}_{R^s} R_R) \{1\}. 
\end{equation*}
An endomorphism $\xi$ of the $(R_s,R)$-bimodule $R$ is determined by $\xi(1)\in R$ which we write as 
$\xi(1)=h_{00}+h_{01}x_1+h_{10}x_2+h_{11}x_1x_2$, where $h_{ij}\in \kk[x_1^2,x_2^2]=Z(R)$. Commutativity relations $f \xi(1)=\xi(1)f$ for $f\in R^s$ can be reduced to those for generators $x_1-x_2,x_1x_2$ of $R^s$, leading to the relations $h_{01}=h_{10}=h_{11}=0$. Consequently, endomorphisms of this bimodule are in a bijection with central elements of $R$, via $\xi(1)=h_{00}\in Z(R)$. Passing to the graded dimension results in the above formula $(b,1)=q(1-q^4)^{-2}$.

Likewise, the inner product $(bc,1)$ is the graded dimension of 
\begin{equation*}
    \HOM(\undB,R)=\HOM(R\otimes_{R^s}\undR \{-1\},R) \cong \Hom({}_{R^s}\undR_R,{}_{R^s}  R_R) \{1\}. 
\end{equation*}

The generator of the hom space $\HOM(\undB,R)$ is given by the degree 3 map below. There $x_1^2$ in a box denotes the bimodule map of multiplication by the central element $x_1^2$ of $R$. Replacing $x_1^2$ by $x_2^2$ in the middle box reverses the sign of the map. 

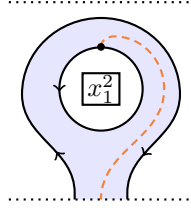
\begin{equation*}
\xy (0,0)*{
\tikzdiagc[scale=.7]{
\fill[blue!10] (-.5,-1.25) to (-.5,-1.125) to[out=90,in=-90] (-1.5,.75) to[out=90,in=180] (0,2.2) to[out=0,in=90] (1.5,.75) to [out=-90,in=90] (.5,-1.125) to (.5,-1.25)--cycle;
\draw[fill=white] (0,.85) circle (22.5pt);
\draw[thick] (-.35,.55) to (.35,.55) to (.35,1.15) to (-.35,1.15)--cycle;\node at (0,.85) {\small \text{$x_1^2$}};
\draw[thick,Orange,densely dashed]  (0,1.65) to[out=90,in=180] (.25,1.85) to[out=0,in=90] (1.2,.85) to[out=-90,in=90] (0,-1.25) ; 
\draw[black,fill=black] (0,1.65) circle (1.5pt);
\draw[thick,-to] (-.79,.85) -- (-.79,.8);  
\draw[thick,-to] (-.8,-.41) -- (-.85,-.34); 
\draw[thick,to-] (.8,-.42) -- (.85,-.35); 
\draw[thick] (-.5,-1.25) to (-.5,-1.125) to[out=90,in=-90] (-1.5,.75) to[out=90,in=180] (0,2.2) to[out=0,in=90] (1.5,.75) to [out=-90,in=90] (.5,-1.125) to (.5,-1.25);
\draw[thick, dotted] (-1.75,-1.25) to (1.75,-1.25); \draw[thick, dotted] (-1.75,2.5) to (1.75,2.5);
}} \endxy
\end{equation*}

A similar computation to the above shows that an $(R,R)$-bimodule maps $\undB\lra R$ are given by $1\longmapsto f(x_1-x_2)$, for any $f\in \kk[x_1^2,x_2^2]$. The generating map, for $f=1$, has degree $3$, due to shift in the degree of $\undB$ as defined. Consequently, the  inner product $(bc,1)$ is given by (\ref{eq_four}). 

Each object of $\sbim$ has a biadjoint object (since the generating objects do). Bimodules $B,\undB$ and $\undR,\undR$ define biadjoint pairs of functors.
Denote the corresponding "biadjointness" antiinvolution on $K_0(\sbim)$ by $\tau$.
It has the properties
\begin{equation*}
\tau(1)=1, \mspace{10mu} \tau(c)=c,\mspace{10mu} \tau(b)=bc, 
\end{equation*}
and
\begin{equation*}
\tau(xy) = \tau(y)\tau(x), \mspace{10mu}  \tau(qx) = q^{-1}\tau(x), \mspace{20mu} x,y\in K_0(\sbim).  \end{equation*}
    
In general, such "biadjointness" involutions reverse the order in the product, but due to commutativity of $K_0(\sbim)$ it does not matter in our case. This involution is $\Z[q,q^{-1}]$-antilinear and compatible with the bilinear form, 
\[ (xm,n) = (m,\tau(x)n), \  x,m,n\in K_0(\sbim).
\] 

Adjointness allows to finish the computation of the inner products (\ref{eq_three}) and (\ref{eq_four}). We can further compute that 
\begin{equation*}
       (b,b)=\frac{1+q^4}{(1-q^4)^2}, \ \ (b,bc)= \frac{2q^2}{(1-q^4)^2}.  
\end{equation*}

Some of the generating maps for $\HOM$ spaces between the four indecomposable bimodules in $\sbim$ are shown below, with each map of degree one. Generating maps in the opposite direction, all in degree 3, are not shown. 

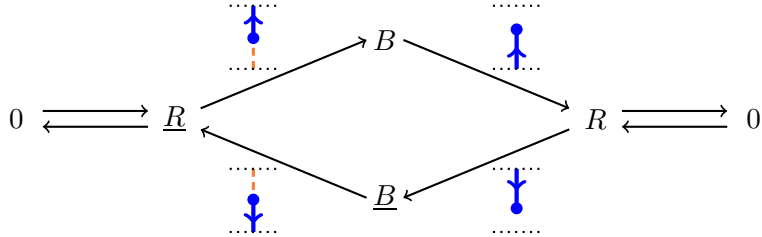
\begin{equation*}
\xy (0,0)*{
\tikzdiagc[scale=.7]{
\node at (0, 1.5) {\text{$B$}};
\node at (0,-1.5) {\text{$\undB$}};
\node at (-4,0) {\text{$\undR$}};
\node at ( 4,0) {\text{$R$}};
\node at (-7,0) {\text{$0$}};
\node at ( 7,0) {\text{$0$}};
\draw[thick,-to] (-6.5, .15) to (-4.5, .15);
\draw[thick,to-] (-6.5,-.15) to (-4.5,-.15);
\draw[thick,to-] ( 6.5, .15) to ( 4.5, .15);
\draw[thick,-to] ( 6.5,-.15) to ( 4.5,-.15);
\draw[thick,-to] (-3.5, .2) to (-.35, 1.5);
\draw[thick,to-] (-3.5,-.2) to (-.35,-1.5);
\draw[thick,to-] ( 3.5, .2) to ( .35, 1.5);
\draw[thick,-to] ( 3.5,-.2) to ( .35,-1.5);
\begin{scope}[shift={(2.5,1.55)},scale=.6]
\draw[ultra thick,blue] (0,-.4) -- (0, .25)node[pos=1, tikzdot]{};
\draw[ultra thick,blue,-to] (0,-1) to (0,-.3); 
\draw[thick, dotted] (-.75,1) to (.75,1);\draw[thick, dotted] (-.75,-1) to (.75,-1);
\end{scope}
\begin{scope}[shift={(2.5,-1.55)},scale=.6,yscale=-1]
\draw[ultra thick,blue] (0,-.4) -- (0, .25)node[pos=1, tikzdot]{};
\draw[ultra thick,blue,-to] (0,-1) to (0,-.3); 
\draw[thick, dotted] (-.75,1) to (.75,1);\draw[thick, dotted] (-.75,-1) to (.75,-1);
\end{scope}
\begin{scope}[shift={(-2.5,1.55)},scale=.6]
\draw[very thick,Orange,densely dashed] (0,-1.) to (0,0); 
\draw[ultra thick,blue,to-] (0,.75) -- (0, -.025)node[pos=1, tikzdot]{};
\draw[ultra thick,blue] (0,1) to (0,.7); 
\draw[thick, dotted] (-.75,1) to (.75,1);\draw[thick, dotted] (-.75,-1) to (.75,-1);
\end{scope}
\begin{scope}[shift={(-2.5,-1.55)},scale=.6,yscale=-1]
\draw[very thick,Orange,densely dashed] (0,-1.) to (0,0); 
\draw[ultra thick,blue,to-] (0,.75) -- (0, -.025)node[pos=1, tikzdot]{};
\draw[ultra thick,blue] (0,1) to (0,.7); 
\draw[thick, dotted] (-.75,1) to (.75,1);\draw[thick, dotted] (-.75,-1) to (.75,-1);
\end{scope}
}}\endxy
\end{equation*} 
 
For each of the four arrows between these four bimodules, 
the $\HOM$ space is a one-dimensional $Z(R)$ module with the generator shown. There are no homs between $R$ and $\undR$, and the corresponding compositions are $0$. The upper and lower portions of the diagram constitute two short exact sequences, and zero objects are added on the sides to emphasize that.

Note that trivalent vertices appear in this calculus 
when passing to the tensor products of $B$'s and $\undB$'s, to describe tensor product decompositions into direct sums.

The bilinear form is determined by $\Z[q,q^{-1}]$-linear trace form $\tr$ on $K_0(\sbim)$, where $\tr(a)=(1,a)$, with 
\begin{equation*}
     \tr(1)=\frac{1}{(1-q^4)^2} , \ \ \tr(b) = \frac{q^3}{(1-q^4)^2},  \ \  \tr(c)=0, \ \  \tr(bc)= \frac{q}{(1-q^4)^2}.  
\end{equation*}
The inner product and trace can be rescaled by $(1-q^4)^2$ to take values in $\Z[q,q^{-1}]$. This corresponds to viewing hom spaces as free graded modules over $Z(R)$ (under either left or right multiplications by central elements) and taking their graded ranks.

%
%

\section{An obstacle to the Reidemeister III relation} \label{sec-obstacle} 

Consider the ring $R_3$ of supercommuting polynomials in 3 variables, 
\[ R_3 = \kk\langle x_1,x_2,x_3\rangle/(x_ix_j+x_jx_i), \ \  1\le i<j\le 3. 
\] 
In this section let us denote $R_3$ by $R$.

The symmetric group $S_3$ acts on $R$, with 
\begin{eqnarray*}
   & &  s_1(x_1)=-x_2, \ s_1(x_2)=-x_1, \ s_1(x_3)=- x_3, \\
   & &  s_2(x_1)=-x_1, \ s_2(x_2)=-x_3, \ s_2(x_3)=- x_2,
\end{eqnarray*}
and $s_i(fg)=s_i(f)s_i(g)$ for $f,g\in R$. 

There are two odd Demazure operators, $\partial_1,\partial_2: R\lra R$. 
Operator $\partial_1 :R\lra R$ is given by: 
\begin{itemize}
\item $\partial_1(1)=0,$  $\partial_1(x_1)=\partial_1(x_2)=1$, $\partial_1(x_3)=0$,
\item the twisted Leibniz rule holds 
\[\partial_1 (fg)=(\partial_1 f)g+s_1(f)\partial_1 g ,
\]
\end{itemize}
and likewise for $\partial_2$.

The kernels of $\partial_1,\partial_2$
 are  subrings $R^1,R^2\subset R$. For instance, ring $R^1$ is the subring of $R$ generated by $x_1-x_2,x_1x_2,x_3$. 
 
 Form graded $R$-bimodules 
\[ B_i \ := \ R \otimes_{R^i}  R \{-1\}, \ \ i = 1, 2.
\] 
 Note that $x_3(1\otimes 1)=(1\otimes 1)x_3$, where $1\otimes 1$ is the generator of $B_1$. 
 
 Diagrammatic calculi of the earlier sections can be repeated separately for $B_1$ and $B_2$. In case of $B_1$ we would need the permutation bimodule, denoted $\undR^{1}\cong R\undone_1$, with the generator $\undone_1$ and $x\undone_1 = \undone_1 s_1(x)$ for $x\in R$ and can then form $\undB_1 = 
 B_1\otimes_R \undR^1\cong \undR^1\otimes_R B_1$. We have not tried to develop a diagrammatical calculus of odd 3-stranded Soergel bimodules which would add interactions between products of $B_1$ and $B_2$.

Consider graded $R$-bimodules 
\[ \undB_i \ := \ R \otimes_{R^{i}} \underline{R}^{i}\otimes_{R^{i}} R \{-1\} \ = \ \ R \undotimes_{R^{i}} R \{-1\}, \ \ i = 1, 2.
\] 
 Define the bimodule
\[ B_{\widehat{121}} := R\otimes_{R^{[2]}} R \{-3\}.
\] 
Here $R^{[2]}\subset R$ is the subring of odd symmetric functions in three variables, $R^{[2]}=\ker(\partial_1)\cap \ker(\partial_2)$.

\begin{prop}  There exists an exact sequence of graded $R$-bimodules 
\begin{equation}
    0 \lra B_{\widehat{121}} \lra B_1 \otimes B_2\otimes B_1 \lra \undB_1 \lra 0 
\end{equation}
This sequence does not split. 
\end{prop} 

The absence of a splitting creates a problem for the  Reidemeister III move invariance. When resolving complexes for $\eR_1\eR_2\eR_1$ and $\eR_2\eR_1\eR_2$ there are not enough contractible summands to slim the complexes down to those with the leftmost term $B_{\widehat{121}}$, which is how the isomorphism is proven in the even case. Instead, the leftmost terms are $B_1 \otimes B_2\otimes B_1$ and $B_2\otimes B_1\otimes B_2$, respectively, which are not isomorphic. This prevents the corresponding complexes of bimodules from being isomorphic in the homotopy category. Finding a way around this obstacle is an interesting problem.


\section{Comparison with the even case}\label{sec_comparison}

In this section we use the same notations to denote the corresponding structures in the even case: 
\begin{itemize}
    \item $R=\kk[x_1,x_2]$  is the ring of polynomials in two variables. $S_2$ acts on it by permuting the variables. 
    \item $\partial$ is the Demazure operator, $\partial(f)=\frac{f- {}^sf}{x_1-x_2}$. 
    \item $R^s\subset S$ is the ring of symmetric functions. 
    \item $B=R\otimes_{R^s}R \{-1\}$ is the generating Soergel bimodule for two variables. 
    \item $\undR$ is the transposition bimodule, $\undR= R\undone = \undone R$, $x_i \undone = \undone x_{s(i)}$. 
\end{itemize}
For the general theory of Soergel bimodules we refer to~\cite{EMTW,S-HarishChandra} and for the diagrammatic calculus of Soergel bimodules to~\cite{EKh,EMTW}.  
There is a natural isomorphism of $(R^s,R)$-bimodules 
\begin{equation*}
    {}_{R^s}R  \otimes_R \undR \ \cong \ {}_{R^s}R 
\end{equation*}
due to involution $s$ acting by identity on $R^s$.

Likewise, there's an isomorphism of $(R,R^s)$-bimodules 
\begin{equation*}
\undR \otimes_R R_{R^s}\cong R_{R^s}.
\end{equation*}
Tensoring these equations with the other ``halves'' of the bimodule $B$ gives bimodule isomorphisms
\begin{equation*}
     \undR \otimes_R B \cong B \cong B\otimes_R \undR. 
\end{equation*}
Odd replacement of these isomorphisms motivates introducing bimodules $\undB$ in that case, see earlier. 

Multiplication map $f\otimes g \longmapsto fg$ induces a surjective bimodule map $B\lra R\{-1\}$ which extends to a short exact sequence of bimodules
\begin{equation*}
    0 \lra \undR\{1\} \lra B\lra R\{-1\}\lra 0 
\end{equation*}
Tensoring all terms of this sequence with $\undR$ flips the sequence to the opposite
\begin{equation}\label{eq_exact_two}
    0 \lra R\{1\} \lra B\lra \undR\{-1\}\lra 0 
\end{equation}

These filtrations were emphasized in~\cite{KhRII}. They allow to think of Rouquier complexes as a sort of  \emph{homological perturbation} or \emph{homological quantization} of the permutation bimodule $\undR$. In the homotopy category (and ignoring $q$-gradings) the complex 
\begin{equation}\label{eq_comp_one}
    0\lra B\lra R\lra 0
\end{equation} 
is not isomorphic to 
\begin{equation}\label{eq_comp_two}
0\lra \undR \lra 0.
\end{equation} 
(in both complexes we place the leftmost nontrivial term in degree $0$). 
Homology groups of these complexes are isomorphic, though. Complex (\ref{eq_comp_one}) is giving by thickening complex (\ref{eq_comp_two}) by the contractible complex $0\lra R\stackrel{\id}{\lra} R\lra 0$, which does not change the homology but makes the complex more subtle on the homotopy category level. The transposition relation, which holds for $\undR$ (that $\undR\otimes \undR\cong R$) fails for the complex (\ref{eq_comp_one}). It's substituted by the weaker relation that (\ref{eq_comp_one}) is invertible in the homotopy category, with the quasi-inverse complex $0\lra R\lra B\lra 0$ given by truncating (\ref{eq_exact_two}). The quasi-inverse is another thickening of $\undR$. 

Before ``homological perturbation'', tensoring with the bimodule $\undR$ is a symmetry of order two, with $\undR\otimes_R \undR\cong R$. Homological perturbation results in an invertible functor of infinite order, while retaining the Reidemeister III relation (the braid relation) in the homotopy category of 
 complexes of bimodules over $\kk[x_1,x_2,x_3]$. 

Thus, upon this homological perturbation, action of the permutation group $S_n$ on the category of $\kk[x_1,\dots, x_n]$-modules given by tensoring with permutation bimodules $\undR_i$  becomes a much more subtle action of the $n$-stranded braid group on the homotopy category of $\kk[x_1,\dots, x_n]$-modules by Rouquier complexes.

\bibliographystyle{amsplain}
\bibliography{biblio}

\providecommand{\bysame}{\leavevmode\hbox to3em{\hrulefill}\thinspace}
\providecommand{\MR}{\relax\ifhmode\unskip\space\fi MR }
\providecommand{\MRhref}[2]{%
  \href{http://www.ams.org/mathscinet-getitem?mr=#1}{#2}
}
\providecommand{\href}[2]{#2}
\begin{thebibliography}{10}

\bibitem{BK}
Jonathan Brundan and Alexander Kleshchev, \emph{Odd {G}rassmannian bimodules
  and derived equivalences for spin symmetric groups},  \textbf{preprint
  \href{https://arxiv.org/abs/2203.14149}{arXiv: 2203.14149v1}} (2022).

\bibitem{ELV}
Mark Ebert, Aaron~D. Lauda, and Laurent Vera, \emph{Derived superequivalences
  for spin symmetric groups and odd $sl(2)$-categorifications},
  \textbf{preprint \href{https://arxiv.org/abs/2203.14153}{arXiv:
  2203.14153v1}} (2022).

\bibitem{EKh}
Ben Elias and Mikhail Khovanov, \emph{Diagrammatics for {S}oergel categories},
  Int. J. Math. Math. Sci. (2010), Art. ID 978635, 58. \MR{3095655}

\bibitem{EMTW}
Ben Elias, Shotaro Makisumi, Ulrich Thiel, and Geordie Williamson,
  \emph{Introduction to {S}oergel bimodules}, RSME Springer Series, vol.~5,
  Springer, Cham, 2020. \MR{4220642}

\bibitem{EKL}
Alexander~P. Ellis, Mikhail Khovanov, and Aaron~D. Lauda, \emph{The odd
  nil{H}ecke algebra and its diagrammatics}, Int. Math. Res. Not. IMRN (2014),
  no.~4, 991--1062. \MR{3168401}

\bibitem{EL}
Alexander~P. Ellis and Aaron~D. Lauda, \emph{An odd categorification of
  {$U_q(\mathfrak{sl}_2)$}}, Quantum Topol. \textbf{7} (2016), no.~2, 329--433.
  \MR{3459963}

\bibitem{KKO}
Seok-Jin Kang, Masaki Kashiwara, and Se-Jin Oh, \emph{Supercategorification of
  quantum {K}ac-{M}oody algebras}, Adv. Math. \textbf{242} (2013), 116--162.
  \MR{3055990}

\bibitem{Kh-Soergel}
Mikhail Khovanov, \emph{Triply-graded link homology and {H}ochschild homology
  of {S}oergel bimodules}, Internat. J. Math. \textbf{18} (2007), no.~8,
  869--885. \MR{2339573}

\bibitem{KhRI}
Mikhail Khovanov and Lev Rozansky, \emph{Matrix factorizations and link
  homology}, Fund. Math. \textbf{199} (2008), no.~1, 1--91. \MR{2391017}

\bibitem{KhRII}
\bysame, \emph{Matrix factorizations and link homology. {II}}, Geom. Topol.
  \textbf{12} (2008), no.~3, 1387--1425. \MR{2421131}

\bibitem{KhT}
Mikhail Khovanov and Richard Thomas, \emph{Braid cobordisms, triangulated
  categories, and flag varieties}, Homology Homotopy Appl. \textbf{9} (2007),
  no.~2, 19--94.

\bibitem{NV}
Gr\'{e}goire Naisse and Pedro Vaz, \emph{Odd {K}hovanov's arc algebra}, Fund.
  Math. \textbf{241} (2018), no.~2, 143--178. \MR{3766566}

\bibitem{NP}
Grégoire Naisse and Krzysztof Putyra, \emph{Odd {K}hovanov homology for
  tangles},  \textbf{preprint
  \href{https://arxiv.org/abs/2003.14290}{arXiv:2003.14290v1}} (2020).

\bibitem{ORS}
Peter~S. Ozsv\'{a}th, Jacob Rasmussen, and Zolt\'{a}n Szab\'{o}, \emph{Odd
  {K}hovanov homology}, Algebr. Geom. Topol. \textbf{13} (2013), no.~3,
  1465--1488. \MR{3071132}

\bibitem{P-odd}
Krzysztof~K. Putyra, \emph{A 2-category of chronological cobordisms and odd
  {K}hovanov homology}, Knots in {P}oland {III}. {P}art {III}, Banach Center
  Publ., vol. 103, Polish Acad. Sci. Inst. Math., Warsaw, 2014, pp.~291--355.
  \MR{3363817}

\bibitem{P-phd}
\bysame, \emph{On a triply-graded generalization of {K}hovanov homology}, 2014,
  Thesis (Ph.D.)--Columbia University. \MR{3232379}

\bibitem{S-HarishChandra}
Wolfgang Soergel, \emph{The combinatorics of {H}arish-{C}handra bimodules}, J.
  Reine Angew. Math. \textbf{429} (1992), 49--74. \MR{1173115}

\bibitem{T}
Anne-Laure Thiel, \emph{Categorification of the virtual braid groups}, Ann.
  Math. Blaise Pascal \textbf{18} (2011), no.~2, 231--243. \MR{2896487}

\end{thebibliography}

\end{document}